\documentclass{article}%
\usepackage{amssymb}
\usepackage{amsmath}
\usepackage{amsfonts}
\usepackage{graphicx}%
\newtheorem{theorem}{Theorem}

\newtheorem{corollary}[theorem]{Corollary}

\newtheorem{definition}[theorem]{Definition}

\newtheorem{lemma}[theorem]{Lemma}

\newtheorem{proposition}[theorem]{Proposition}
\newtheorem{remark}[theorem]{Remark}

\newenvironment{proof}[1][Proof]{\noindent\textbf{#1.} }{\ \rule{0.5em}{0.5em}}
\begin{document}

\title{Carleson Measures for the Drury-Arveson Hardy space and other Besov-Sobolev
Spaces on Complex Balls}
\author{N. Arcozzi\thanks{Partially supported by the COFIN project Analisi Armonica,
funded by the Italian Minister for Research}\\Dipartimento do Matematica\\Universita di Bologna\\40127 Bologna, ITALY
\and R. Rochberg\thanks{This material is based upon work supported by the National
Science Foundation under Grant No. 0400962.}\\Department of Mathematics\\Washington University\\St. Louis, MO 63130, U.S.A.
\and E. Sawyer\thanks{This material based upon work supported by the National
Science and Engineering Council of Canada.}\\Department of Mathematics \& Statistics\\McMaster University\\Hamilton, Ontairo, L8S 4K1, CANADA}

\begin{abstract}
For $0\leq\sigma<1/2$ we characterize Carleson measures $\mu$ for the analytic
Besov-Sobolev spaces $B_{2}^{\sigma}$ on the unit ball $\mathbb{B}_{n}$ in
$\mathbb{C}^{n}$ by the discrete tree condition
\[
\sum_{\beta\geq\alpha}\left[  2^{\sigma d\left(  \beta\right)  }I^{\ast}%
\mu\left(  \beta\right)  \right]  ^{2}\leq CI^{\ast}\mu\left(  \alpha\right)
<\infty,\;\;\;\;\;\alpha\in\mathcal{T}_{n},
\]
on the associated Bergman tree $\mathcal{T}_{n}$. Combined with recent results
about interpolating sequences this leads, for this range of $\sigma, $ to a
characterization of universal interpolating sequences for $B_{2}^{\sigma}$ and
also for its multiplier algebra.

However, the tree condition is \emph{not} necessary for a measure to be a
Carleson measure for the Drury-Arveson Hardy space $H_{n}^{2}=B_{2}^{1/2}$. We
show that $\mu$ is a Carleson measure for $B_{2}^{1/2}$ if and only if both
the simple condition
\[
2^{d\left(  \alpha\right)  }I^{\ast}\mu\left(  \alpha\right)  \leq
C,\;\;\;\;\;\alpha\in\mathcal{T}_{n},
\]
and the split tree condition
\[
\sum_{k\geq0}\sum_{\gamma\geq\alpha}2^{d\left(  \gamma\right)  -k}%
\sum_{\substack{\left(  \delta,\delta^{\prime}\right)  \in\mathcal{G}^{\left(
k\right)  }\left(  \gamma\right)  }}I^{\ast}\mu\left(  \delta\right)  I^{\ast
}\mu\left(  \delta^{\prime}\right)  \leq CI^{\ast}\mu\left(  \alpha\right)
,\;\;\;\;\;\alpha\in\mathcal{T}_{n},
\]
hold. This gives a sharp estimate for Drury's generalization of von Neumann's
operator inequality to the complex ball, and also provides a universal
characterization of Carleson measures, up to dimensional constants, for
Hilbert spaces with a complete continuous Nevanlinna-Pick kernel function.

We give a detailed analysis of the split tree condition for measures supported
on embedded two manifolds and we find that in some generic cases the condition
simplifies. We also establish a connection between function spaces on embedded
two manifolds and Hardy spaces of plane domains.

\end{abstract}
\maketitle
\tableofcontents

\section{Overview}

We give a description of Carleson measures for certain Hilbert spaces of
holomorphic functions on the ball in $\mathbb{C}^{n}.$ In the next section we
give background and a summary. We also describe ways the characterization can
be used and how the characterization simplifies in some special cases. In the
following section we collect certain technical tools.$\ $The main work of
characterizing the Carleson measures is in the section after that.$\ $A brief
final appendix has the real variable analog of our main results.

\section{Introduction}

\subsection{Function Spaces}

Let $\mathbb{B}_{n}$ be the unit ball in $\mathbb{C}^{n}$. Let $dz$ be
Lebesgue measure on $\mathbb{C}^{n}$ and let $d\lambda_{n}\left(  z\right)
=(1-\left\vert z\right\vert ^{2})^{-n-1}dz$ be the invariant measure on the
ball. For integer $m\geq0$, and for $0\leq\sigma<\infty,$ $1<p<\infty,$
$m+\sigma>n/p$ we define the analytic Besov-Sobolev spaces $B_{p}^{\sigma
}\left(  \mathbb{B}_{n}\right)  $ to consist of those holomorphic functions
$f$ on the ball such that
\begin{equation}
\left\{  \sum_{k=0}^{m-1}\left\vert f^{\left(  k\right)  }\left(  0\right)
\right\vert ^{p}+\int_{\mathbb{B}_{n}}\left\vert \left(  1-\left\vert
z\right\vert ^{2}\right)  ^{m+\sigma}f^{\left(  m\right)  }\left(  z\right)
\right\vert ^{p}d\lambda_{n}\left(  z\right)  \right\}  ^{\frac{1}{p}}<\infty.
\label{Hilbertnorm}%
\end{equation}
Here $f^{\left(  m\right)  }$ is the $m^{th}$ order complex derivative of $f$.
The spaces $B_{p}^{\sigma}\left(  \mathbb{B}_{n}\right)  $ are independent of
$m$ and are Banach spaces with norms given in (\ref{Hilbertnorm}).

For $p=2$ these are Hilbert spaces with the obvious inner product. This scale
of spaces includes the Dirichlet spaces $B_{2}\left(  \mathbb{B}_{n}\right)
=B_{2}^{0}\left(  \mathbb{B}_{n}\right)  ,$ weighted Dirichlet-type spaces
with $0<\sigma<1/2,$ the Drury-Arveson Hardy spaces $H_{n}^{2}=B_{2}%
^{1/2}\left(  \mathbb{B}_{n}\right)  $ (also known as the symmetric Fock
spaces over $\mathbb{C}^{n})$ (\cite{Dru}, \cite{Arv}, \cite{Che}), the Hardy
spaces $H^{2}\left(  \mathbb{B}_{n}\right)  =B_{2}^{n/2}\left(  \mathbb{B}%
_{n}\right)  $, and the weighted Bergman spaces with $\sigma>n/2.$

Alternatively these Hilbert spaces can be viewed as part of the Hardy-Sobolev
scale of spaces $J_{\gamma}^{2}\left(  \mathbb{B}_{n}\right)  $, $\gamma
\in\mathbb{R}$, consisting of all holomorphic functions $f$ in the unit ball
whose radial derivative $R^{\gamma}f$ of order $\gamma$ belongs to the Hardy
space $H^{2}\left(  \mathbb{B}_{n}\right)  $ ($R^{\gamma}f=\sum_{k=0}^{\infty
}\left(  k+1\right)  ^{\gamma}f_{k}$ if $f=\sum_{k=0}^{\infty}f_{k}$ is the
homogeneous expansion of $f$). The Hardy-Sobolev scale coincides with the
Besov-Sobolev scale and we have
\[
B_{2}^{\sigma}\left(  \mathbb{B}_{n}\right)  =J_{\gamma}^{2}\left(
\mathbb{B}_{n}\right)  ,\;\;\;\;\;\sigma+\gamma=\frac{n}{2}\;,0\leq\sigma
\leq\frac{n}{2}.
\]
Thus $\sigma$ measures the order of antiderivative required to belong to the
Dirichlet space $B_{2}\left(  \mathbb{B}_{n}\right)  $, and $\gamma=\frac
{n}{2}-\sigma$ measures the order of the derivative that belongs to the Hardy
space $H^{2}\left(  \mathbb{B}_{n}\right)  $.

\subsection{Carleson Measures}

By a Carleson measure for $B_{p}^{\sigma}\left(  \mathbb{B}_{n}\right)  $ we
mean a positive measure defined on $\mathbb{B}_{n}$ such that the following
Carleson embedding holds; for $f\in B_{p}^{\sigma}\left(  \mathbb{B}%
_{n}\right)  $%
\begin{equation}
\int_{\mathbb{B}_{n}}\left\vert f\left(  z\right)  \right\vert ^{p}d\mu\leq
C_{\mu}\left\Vert f\right\Vert _{B_{p}^{\sigma}\left(  \mathbb{B}_{n}\right)
}^{p}. \label{carldef}%
\end{equation}
The set of all such is denoted $CM(B_{p}^{\sigma}\left(  \mathbb{B}%
_{n}\right)  )$ and we define the Carleson measure norm $\left\Vert
\mu\right\Vert _{Carleson}$ to be the infimum of the possible values of
$C_{\mu}^{1/p}.$ In \cite{ArRoSa2} we described the Carleson measures for
$B_{p}^{\sigma}\left(  \mathbb{B}_{n}\right)  $ for $\sigma=0$ and
$1<p<2+\frac{1}{n-1}.$ Here we consider $\sigma>0$ and focus our attention on
the Hilbert space cases, $p=2.$ We show that, \textit{mutatis mutandis,} the
results for $\sigma=0$ extend to the range $0\leq\sigma<1/2.$ Fundamental to
this extension is the fact that for $0\leq\sigma<1/2$ the real part of the
reproducing kernel for $B_{2}^{\sigma}\left(  \mathbb{B}_{n}\right)  $ is
comparable to its modulus.\ In fact, in the appendix when we use the modulus
of the reproducing kernel in defining nonisotropic potential spaces, results
similar to those for $\sigma=0$ continue to hold for $0\leq\sigma<n/2$.
However even though the reproducing kernel for $B_{2}^{1/2}\left(
\mathbb{B}_{n}\right)  =H_{n}^{2}$ has positive real part, its real part is
not comparable to its modulus. For that space a new type of analysis must be
added and by doing that we describe the Carleson measures for $B_{2}%
^{1/2}\left(  \mathbb{B}_{n}\right)  $. For $1/2<\sigma<n/2$ the real part of
the kernel is not positive, our methods don't apply, and that range remains
mysterious. For $\sigma\geq n/2$ we are in the realm of the classical Hardy
and Bergman spaces and the description of the Carleson measures is well
established \cite{Rud}, \cite{Zhu}.

Let $\mathcal{T}_{n}$ denote the Bergman tree associated to the ball
$\mathbb{B}_{n}$ as in \cite{ArRoSa2}. We show (Theorem \ref{Besov'}) that the
tree condition,
\begin{equation}
\sum_{\beta\geq\alpha}\left[  2^{\sigma d\left(  \beta\right)  }I^{\ast}%
\mu\left(  \beta\right)  \right]  ^{2}\leq CI^{\ast}\mu\left(  \alpha\right)
<\infty,\;\;\;\;\;\alpha\in\mathcal{T}_{n}, \label{treecondo}%
\end{equation}
characterizes Carleson measures for the Besov-Sobolev space $B_{2}^{\sigma
}\left(  \mathbb{B}_{n}\right)  $ in the range $0\leq\sigma<1/2$. To help
place this condition in context we compare it with the corresponding simple
condition. The condition%
\begin{equation}
2^{2\sigma d\left(  \alpha\right)  }I^{\ast}\mu\left(  \alpha\right)  \leq C
\tag{SC$(\sigma)$}\label{SCsigma}%
\end{equation}
is necessary for $\mu$ to be a Carleson measure as is seen by testing the
embedding (\ref{carldef}) on reproducing kernels or by starting with
(\ref{treecondo}) and using the infinite sum there to dominate the single term
with $\beta=\alpha.$\ Although SC($\sigma$) is not sufficient to insure that
$\mu$ is a Carleson measure, slight strengthenings of it are sufficient, see
Lemma \ref{simptree} below. In particular, for any\ $\varepsilon>0$ the
condition SC($\sigma+\varepsilon)$ is sufficient.

On the other hand if $\sigma\geq1/2$ then, by the results in \cite{CaOr}
together with results in the Appendix, there are positive measures $\mu$ on
the ball that are Carleson for $J_{\frac{n}{2}-\sigma}^{2}\left(
\mathbb{B}_{n}\right)  =B_{2}^{\sigma}\left(  \mathbb{B}_{n}\right)  $ but
fail to satisfy the tree condition (\ref{treecondo}). Our analysis of Carleson
measures for the \textquotedblleft endpoint\textquotedblright\ case
$B_{2}^{1/2}\left(  \mathbb{B}_{n}\right)  $, the Drury-Arveson Hardy space,
proceeds in two stages. First, by a functional analytic argument we show that
the norm $\left\Vert \mu\right\Vert _{Carleson}$ is comparable, independently
of dimension, with the best constant $C$ in the inequality
\begin{equation}
\int_{\mathbb{B}_{n}}\int_{\mathbb{B}_{n}}\left(  \operatorname{Re}\frac
{1}{1-\overline{z}\cdot z^{\prime}}\right)  f\left(  z^{\prime}\right)
d\mu\left(  z^{\prime}\right)  g\left(  z\right)  d\mu\left(  z\right)  \leq
C\left\Vert f\right\Vert _{L^{2}\left(  \mu\right)  }\left\Vert g\right\Vert
_{L^{2}\left(  \mu\right)  }. \label{bilin}%
\end{equation}
We then proceed with a geometric analysis of the conditions under which this
inequality holds. If in (\ref{bilin}) we were working with the integration
kernel $\left\vert \frac{1}{1-\overline{z}\cdot z^{\prime}}\right\vert $
rather than $\operatorname{Re}\frac{1}{1-\overline{z}\cdot z^{\prime}}$ we
could do an analysis similar to that for $\sigma<1/2$ and would find the
inequality characterized by the tree condition with $\sigma=1/2:$%
\begin{equation}
\sum_{\beta\geq\alpha}2^{d\left(  \beta\right)  }I^{\ast}\mu\left(
\beta\right)  ^{2}\leq CI^{\ast}\mu\left(  \alpha\right)  <\infty
,\;\;\;\;\;\alpha\in\mathcal{T}_{n}, \label{treeArv}%
\end{equation}
see Subsection \ref{related} and the Appendix. However, as we will show, for
$n>1,$ the finiteness of $\left\Vert \mu\right\Vert _{Carleson}$ is equivalent
\emph{neither} to the tree condition (\ref{treeArv}), \emph{nor} to the simple
condition
\begin{equation}
2^{d\left(  \alpha\right)  }I^{\ast}\mu\left(  \alpha\right)  \leq
C,\;\;\;\;\;\alpha\in\mathcal{T}_{n}, \label{simp}%
\end{equation}
(SC$(1/2)$ in the earlier notation).

To proceed we will introduce additional structure on the Bergman tree
$\mathcal{T}_{n}.$ For $\alpha$ in $\mathcal{T}_{n}$, we denote by $\left[
\alpha\right]  $ an equivalence class in a certain quotient tree
$\mathcal{R}_{n}$ of \textquotedblleft rings\textquotedblright\ consisting of
elements in a \textquotedblleft common slice\textquotedblright\ of the ball
having common distance from the root. Using this additional structure we will
show in Theorem \ref{Arvcom} that the Carleson measures are characterized by
the simple condition (\ref{simp}) together with the \textquotedblleft
split\textquotedblright\ tree condition
\begin{equation}
\sum_{k\geq0}\sum_{\gamma\geq\alpha}2^{d\left(  \gamma\right)  -k}%
\sum_{\left(  \delta,\delta^{\prime}\right)  \in\mathcal{G}^{\left(  k\right)
}\left(  \gamma\right)  }I^{\ast}\mu\left(  \delta\right)  I^{\ast}\mu\left(
\delta^{\prime}\right)  \leq CI^{\ast}\mu\left(  \alpha\right)
,\;\;\;\;\;\alpha\in\mathcal{T}_{n}, \label{splittree}%
\end{equation}
and moreover we have the norm estimate
\begin{align}
\left\Vert \mu\right\Vert _{Carleson}  &  \approx\sup_{\alpha\in
\mathcal{T}_{n}}\sqrt{2^{d\left(  \alpha\right)  }I^{\ast}\mu\left(
\alpha\right)  }\label{s'}\\
&  +\sup_{\substack{\alpha\in\mathcal{T}_{n}\\I^{\ast}\mu\left(
\alpha\right)  >0}}\sqrt{\frac{1}{I^{\ast}\mu\left(  \alpha\right)  }%
\sum_{k\geq0}\sum_{\gamma\geq\alpha}2^{d\left(  \gamma\right)  -k}%
\sum_{\substack{\left(  \delta,\delta^{\prime}\right)  \in\mathcal{G}^{\left(
k\right)  }\left(  \gamma\right)  }}I^{\ast}\mu\left(  \delta\right)  I^{\ast
}\mu\left(  \delta^{\prime}\right)  }.\nonumber
\end{align}
The restriction $\left(  \delta,\delta^{\prime}\right)  \in\mathcal{G}%
^{\left(  k\right)  }\left(  \gamma\right)  $ in the sums above means that we
sum over all pairs $\left(  \delta,\delta^{\prime}\right)  $ of grand$^{k}%
$-children of $\gamma$ that have $\gamma$ as their minimum, that lie in
well-separated rings in the quotient tree, but whose predecessors of order
two, $A^{2}\delta$ and $A^{2}\delta^{\prime}$, lie in a common ring. That is,
the ring tree geodesics to $\delta$ and to $\delta^{\prime}$ have recently
split, at distance roughly $k$ from $\gamma$. Note that if we were to extend
the summation to all pairs $\left(  \delta,\delta^{\prime}\right)  $ of
grand$^{k}$-children of $\gamma$ then this condition would be equivalent to
the tree condition (\ref{treeArv}). More formally,

\begin{definition}
\label{grand}The set $\mathcal{G}^{\left(  k\right)  }\left(  \gamma\right)  $
consists of pairs $\left(  \delta,\delta^{\prime}\right)  $ of grand$^{k}%
$-children of $\gamma$ in $\mathcal{G}^{\left(  k\right)  }\left(
\gamma\right)  \times\mathcal{G}^{\left(  k\right)  }\left(  \gamma\right)  $
which satisfy $\delta\wedge\delta^{\prime}=\gamma$, $\left[  A^{2}%
\delta\right]  =\left[  A^{2}\delta^{\prime}\right]  $ (which implies
$d\left(  \left[  \delta\right]  ,\left[  \delta^{\prime}\right]  \right)
\leq4$) and $d^{\ast}\left(  \left[  \delta\right]  ,\left[  \delta^{\prime
}\right]  \right)  =4$.
\end{definition}

Here
\[
d^{\ast}\left(  \left[  \alpha\right]  ,\left[  \beta\right]  \right)
=\inf_{U\in\mathcal{U}_{n}}d\left(  \left[  t\left(  Uc_{\alpha}\right)
\right]  ,\left[  t\left(  Uc_{\beta}\right)  \right]  \right)  ,
\]
and $\mathcal{U}_{n}$ denotes the group of unitary rotations of the ball
$\mathbb{B}_{n}$. For $\alpha$ in the Bergman tree $\mathcal{T}_{n}$,
$c_{\alpha}$ is the \textquotedblleft center\textquotedblright\ of the Bergman
kube $K_{\alpha}$. For $z\in\mathbb{B}_{n}$, $t\left(  z\right)  $ denotes the
element $\alpha^{\prime}\in\mathcal{T}_{n}$ such that $z\in K_{\alpha^{\prime
}}$. Thus $d^{\ast}\left(  \left[  \alpha\right]  ,\left[  \beta\right]
\right)  $ measures an \textquotedblleft invariant\textquotedblright\ distance
between the rings $\left[  \alpha\right]  $ and $\left[  \beta\right]  $. Note
that $\mathcal{G}^{\left(  0\right)  }\left(  \gamma\right)  =\mathcal{G}%
\left(  \gamma\right)  $ is the set of grandchildren of $\gamma$. Further
details can be found in Subsection \ref{ring} below on a modified Bergman tree
and its quotient tree.

We noted before that for $0\leq\sigma<1/2$ the tree condition (\ref{treecondo}%
) implies the corresponding simple condition \ref{SCsigma}. However the split
tree condition (\ref{splittree}) does not imply the simple condition
(\ref{simp}). In fact, measures supported on a slice, i.e., on the
intersection of the ball with a complex line through the origin, satisfy the
split tree condition vacuously. This is because for measure supported on a
single slice and $\delta\ $and $\delta^{\prime}$ in different rings at most
one of the factors $I^{\ast}\mu\left(  \delta\right)  $, $I^{\ast}\mu\left(
\delta^{\prime}\right)  $ can be nonzero. However such measures may or may not
satisfy (\ref{simp}). Similarly the split tree condition is vacuously
satisfied when $n=1.$ In that case we have the classical Hardy space and
Carleson's classical condition SC$(1/2).$

In our proof of (\ref{s'}) the implicit constants of equivalence depend on the
dimension $n.$ One reason for attention to possible dimensional dependence of
constants arises in Subsubsection \ref{complete}. Roughly, a large class of
Hilbert spaces with reproducing kernels have natural realizations as subspaces
of the various $H_{n}^{2}$ and this occurs in ways that lets us use the
characterization of Carleson measures for $H_{n}^{2}$ to obtain descriptions
of the Carleson measures for these other spaces. However in the generic case,
as well as for the most common examples, $n=\infty.$ When $n=\infty$ we can
pull back characterizations of Carleson measures of the form (\ref{carldef})
or (\ref{bilin}) but, because of the dimensional dependence of the constants,
we cannot obtain characterizations using versions of (\ref{simp}) and
(\ref{splittree}).

Finally, we mention two technical refinements of these results. First, it
suffices to test the bilinear inequality (\ref{bilin}) over $f=g=\chi
_{T\left(  w\right)  }$ where $T\left(  w\right)  $ is a nonisotropic Carleson
region with vertex $w$. This holds because in Subsection \ref{split}, when
proving the necessity of the split tree condition, we only use that special
case of the bilinear inequality. However that observation commits us to a
chain of implications which uses (\ref{s'}) and thus we don't know that the
constants in the restricted condition are independent of dimension. Second,
the condition (\ref{splittree}) can be somewhat simplified by further
restricting the sum over $k$ and $\gamma$ on the left side to $k\leq
\varepsilon d\left(  \gamma\right)  $ for any fixed $\varepsilon>0$; the
resulting $\varepsilon$-split tree condition is
\begin{equation}
\sum_{\gamma\geq\alpha:0\leq k\leq\epsilon d\left(  \gamma\right)
}2^{d\left(  \gamma\right)  -k}\sum_{\left(  \delta,\delta^{\prime}\right)
\in\mathcal{G}^{\left(  k\right)  }\left(  \gamma\right)  }I^{\ast}\mu\left(
\delta\right)  I^{\ast}\mu\left(  \delta^{\prime}\right)  \leq CI^{\ast}%
\mu\left(  \alpha\right)  ,\;\;\;\;\;\alpha\in\mathcal{T}_{n}.
\label{epsilonsplit}%
\end{equation}
The reason (\ref{simp}) and (\ref{epsilonsplit}) suffice is that the sum in
(\ref{splittree}) over $k>\varepsilon d\left(  \gamma\right)  $ is dominated
by the left side of (\ref{treecondo}) with $\sigma=(1-\varepsilon)/2$, and
that this condition is in turn implied by the simple condition (\ref{simp}).
See Lemma \ref{simptree} below.

Finally, as we mentioned, the characterization of Carleson measures for
$B_{2}^{\sigma}\left(  \mathbb{B}_{n}\right)  $ remains open in the range
$1/2<\sigma<n/2$. The Carleson measures for the Hardy space, $\sigma=n/2,$ and
the weighted Bergman spaces, $\sigma>n/2$, are characterized by \ref{SCsigma};
see \cite{Rud} and \cite{Zhu}.

\subsection{Applications and special cases}

Before proving the characterizations of Carleson measures we present some uses
of those results and also describe how the general results simplify in some
cases.\ In doing this we will use the results and notation of later sections
but we will not use results from this section later.

We describe the multiplier algebra $M_{B_{2}^{\sigma}\left(  \mathbb{B}%
_{n}\right)  }$ of $B_{2}^{\sigma}\left(  \mathbb{B}_{n}\right)  $ for
$0\leq\sigma\leq1/2.$ For the smaller range $0\leq\sigma<1/2$ we describe the
interpolating sequences for $B_{2}^{\sigma}\left(  \mathbb{B}_{n}\right)  $
and for $M_{B_{2}^{\sigma}\left(  \mathbb{B}_{n}\right)  }.$ We give an
explicit formula for the norm which arises in Drury's generalization of von
Neumann's operator inequality to the complex ball $\mathbb{B}_{n}$. We give a
universal characterization of Carleson measures for Hilbert spaces with a
complete Nevanlinna-Pick kernel function.

To understand the split tree condition (\ref{splittree}) better we investigate
the structure of the Carleson measures for $B_{2}^{1/2}\left(  \mathbb{B}%
_{n}\right)  $ which are supported on real 2-manifolds embedded in
$\mathbb{B}_{n}$. This will also give information about Carleson measures for
spaces of functions on those manifolds. Suppose we have a $C^{1}$ embedding of
a real 2-manifold $\mathcal{S}$ into $\mathbb{B}_{n}$ and that $\mathcal{\bar
{S}}$ meets the boundary of the ball transversally in a curve $\Gamma.$
Suppose we have a Carleson measure for $B_{2}^{1/2}\left(  \mathbb{B}%
_{n}\right)  $ supported in $\mathcal{S}.$ We find that

\begin{itemize}
\item If $\Gamma$ is transverse to the complex tangential boundary directions
then (\ref{epsilonsplit}) becomes vacuous for small $\varepsilon$ and the
Carleson measures are described by the simple condition (\ref{simp}). In
particular this applies to $C^{1}$ embedded holomorphic curves and shows that
the Carleson measures for the associated spaces coincide with the Carleson
measures for the Hardy spaces of those curves.\ For planar domains we show
that if the embedding is $C^{2}$ then these spaces coincide with the Hardy spaces.

\item If $\Gamma$ is a complex tangential curve, that is if its tangent lies
in the complex tangential boundary directions then (\ref{epsilonsplit})
reduces to the tree condition (\ref{treeArv}) and the Carleson measures are
described by the tree condition. A similar result holds for measures supported
on embedded real $k$-manifolds which meet the boundary transversely and in the
complex tangential directions.
\end{itemize}

On the other hand, the embedding $\mathcal{S}\ $of the unit disk into
$\mathbb{B}_{\infty}$ associated with a space $B_{2}^{\sigma}\left(
\mathbb{B}_{1}\right)  ,$ $0\leq\sigma<1/2,$ extends to $\mathcal{\bar{S}}$,
is Lipschitz continuous of order $\sigma$ but not $C^{1}$ and is not
transverse to the boundary. In this more complicated situation neither of the
two simplifications occur.

\subsubsection{Multipliers}

A holomorphic function $f$ on the ball is called a multiplier for the space
$B_{2}^{\sigma}\left(  \mathbb{B}_{n}\right)  $ if the multiplication operator
$M_{f}$ defined by $M_{f}(g)=fg$ is a bounded linear operator on
$B_{2}^{\sigma}\left(  \mathbb{B}_{n}\right)  .$ In that case the multiplier
norm of $f$ is defined to be the operator norm of $M_{f}.$ The space of all
such is denoted $M_{B_{2}^{\sigma}\left(  \mathbb{B}_{n}\right)  }.$

Ortega and Fabrega \cite{OrFa} have characterized multipliers for the
Hardy-Sobolev spaces using Carleson measures. We refine their result by
including a geometric characterization of those measures.

\begin{theorem}
\label{multipliers}Suppose $0\leq\sigma\leq1/2.$ Then $f$ is in $M_{B_{2}%
^{\sigma}\left(  \mathbb{B}_{n}\right)  }$ if and only if $f$ is bounded and
for some, equivalently for any, $k>n/2-\sigma$
\[
d\mu_{f,k}=\left\vert (1-\left\vert z\right\vert ^{2})^{k}f^{(k)}\right\vert
^{2}(1-\left\vert z\right\vert ^{2})^{2\sigma}d\lambda_{n}(z)\in
CM(B_{2}^{\sigma}\left(  \mathbb{B}_{n}\right)  ).
\]
In that case we have%

\[
\left\|  f\right\|  _{M_{B_{2}^{\sigma}\left(  \mathbb{B}_{n}\right)  }}%
\sim\left\|  f\right\|  _{H^{\infty}\left(  \mathbb{B}_{n}\right)  }+\left\|
d\mu_{f,k}\right\|  _{CM(B_{2}^{\sigma}\left(  \mathbb{B}_{n}\right)  )}.
\]
If $0\leq\sigma<1/2$ the second summand can be evaluated using Theorem
\ref{Besov'}. For $\sigma=1/2$ the second summand can be evaluated using
Theorem \ref{Arvcom}.
\end{theorem}

In the familiar case of the one variable Hardy space, $n=1,$ $\sigma=1/2,$ and
$k=1;$ the Carleson measure condition need not be mentioned because it is
implied by the boundedness of $f,\ $for instance because of the inclusion
$H^{\infty}\left(  \mathbb{B}_{1}\right)  \subset BMO(\mathbb{B}_{1})$ and the
characterization of $BMO$ in terms of Carleson measures. Thus the multiplier
algebra consists of all bounded functions. However for $n>1$ and $0\leq
\sigma\leq1/2$ as well as $n=1$ and $0\leq\sigma<1/2,$ there are bounded
functions which are not multipliers. Because the constant functions are in all
the $B_{2}^{\sigma}$ we can establish this by exhibiting bounded functions not
in the $B_{2}^{\sigma}.$ In \cite{Che} Chen constructs such functions for
$n>1,$ $\sigma=1/2.$ If $\sigma<1/2$ then $B_{2}^{\sigma}\subset B_{2}^{1/2}$
and hence Chen's functions also fail to be in $B_{2}^{\sigma}.$ Similar but
simpler examples work for $n=1,0\leq\sigma<1/2.$ Other approaches to this are
in \cite{Dru} and \cite{Arv}.

\subsubsection{Interpolating sequences\label{IS}}

Given $\sigma,$ $0\leq\sigma<1/2$ and a discrete set $Z=\{z_{i}\}_{i=1}%
^{\infty}\subset\mathbb{B}_{n}$ we define the associated measure $\mu_{Z}%
=\sum_{j=1}^{\infty}(1-\left\vert z_{j}\right\vert ^{2})^{2\sigma}%
\delta_{z_{j}}.$ We say that $Z$ is an interpolating sequence for
$B_{2}^{\sigma}\left(  \mathbb{B}_{n}\right)  $ if the restriction map $R$
defined by $Rf(z_{i})=f(z_{i})$ for $z_{i}\in Z$ maps $B_{2}^{\sigma}\left(
\mathbb{B}_{n}\right)  $ into and onto $\ell^{2}(Z,\mu_{Z}).$ We say that $Z$
is an interpolating sequence for $M_{B_{2}^{\sigma}\left(  \mathbb{B}%
_{n}\right)  }$ if $R$ maps $M_{B_{2}^{\sigma}\left(  \mathbb{B}_{n}\right)
}$ into and onto $\ell^{\infty}(Z,\mu_{Z}).$ Using results of B. B\"{o}e
\cite{Boe}, J. Agler and J. E. M$^{c}$Carthy \cite{AgMc}, D. Marshall and C.
Sundberg \cite{MaSu}, along with the above Carleson measure characterization
for $B_{2}^{\sigma}\left(  \mathbb{B}_{n}\right)  $ we now characterize those
sequences.\ Denote the Bergman metric on the complex ball $\mathbb{B}_{n}$ by
$\beta$.

\begin{theorem}
\label{interpolation}Suppose $\sigma,Z,$ and $\mu_{Z}$ are as described. Then
$Z$ is an interpolating sequence for $B_{2}^{\sigma}\left(  \mathbb{B}%
_{n}\right)  $ \emph{if and only if} $Z$ is an interpolating sequence for the
multiplier algebra $M_{B_{2}^{\sigma}\left(  \mathbb{B}_{n}\right)  }$
\emph{if and only if} $Z$ satisfies the separation condition $\inf_{i\neq
j}\beta\left(  z_{i},z_{j}\right)  >0$ and $\mu_{Z}$ is a $B_{2}^{\sigma
}\left(  \mathbb{B}_{n}\right)  $ Carleson measure, i.e. it satisfies the tree
condition (\ref{treecondo}).
\end{theorem}

%

\proof
The case $\sigma=0$ was proved in \cite{MaSu} when $n=1$ and in \cite{ArRoSa2}
when $n>1$. If $0<\sigma<1/2$, then Corollary 1.12 of \cite{AgMc} shows that
the reproducing kernel $k\left(  z,w\right)  =\left(  \frac{1}{1-\overline
{w}\cdot z}\right)  ^{2\sigma}$ has the complete Nevanlinna-Pick property.
Indeed, the corollary states that $k$ has the complete Nevanlinna-Pick
property if and only if for any finite set $\left\{  z_{1},z_{2}%
,...,z_{m}\right\}  $, the matrix $H_{m}$ of reciprocals of inner products of
reproducing kernels $k_{z_{i}}$ for $z_{i}$, i.e.
\[
H_{m}=\left[  \frac{1}{\left\langle k_{z_{i}},k_{z_{j}}\right\rangle }\right]
_{i,j=1}^{m},
\]
has exactly one positive eigenvalue counting multiplicities. We may expand
$\left\langle k_{z_{i}},k_{z_{j}}\right\rangle ^{-1}$ by the binomial theorem
as
\[
\left(  1-\overline{z_{j}}\cdot z_{i}\right)  ^{2\sigma}=1-\sum_{\ell
=1}^{\infty}c_{\ell}\left(  \overline{z_{j}}\cdot z_{i}\right)  ^{\ell},
\]
where $c_{\ell}=(-1)^{\ell+1}\left(
\begin{array}
[c]{c}%
2\sigma\\
\ell
\end{array}
\right)  \geq0$ for $\ell\geq1$ and $0<2\sigma<1$. Now the matrix $\left[
\overline{z_{j}}\cdot z_{i}\right]  _{i,j=1}^{m}$ is nonnegative semidefinite
since
\[
\sum_{i,j=1}^{m}\zeta_{i}\left(  \overline{z_{j}}\cdot z_{i}\right)
\overline{\zeta_{i}}=\left\vert \left(  \zeta_{1}z_{1},...,\zeta_{m}%
z_{m}\right)  \right\vert ^{2}\geq0.
\]
Thus by Schur's theorem so is $\left[  \left(  \overline{z_{j}}\cdot
z_{i}\right)  ^{\ell}\right]  _{i,j=1}^{m}$ for every $\ell\geq1,$ and hence,
also, so is the sum with positive coefficients. Thus the positive part of the
matrix $H_{m}$ is $\left[  1\right]  _{i,j=1}^{m}$ which has rank $1$, and
hence the sole positive eigenvalue of $H_{m}$ is $m$. Once we know that
$B_{2}^{\sigma}\left(  \mathbb{B}_{n}\right)  $ has the Pick property then it
follows from a result of Marshall and Sundberg (Theorem 9.19 of \cite{AgMc2})
that the interpolating sequences for $M_{B_{2}^{\sigma}\left(  \mathbb{B}%
_{n}\right)  }$ are the same as those for $B_{2}^{\sigma}\left(
\mathbb{B}_{n}\right)  .$ Thus we need only consider the case of
$B_{2}^{\sigma}\left(  \mathbb{B}_{n}\right)  .$

We now invoke a theorem of B. B\"{o}e \cite{Boe} which says that for certain
Hilbert spaces with reproducing kernel, in the presence of the separation
condition (which is necessary for an interpolating sequence, see Ch. 9 of
\cite{AgMc2}) a necessary and sufficient condition for a sequence to be
interpolating is that the Grammian matrix associated with $Z$ is bounded. That
matrix is built from normalized reproducing kernels; it is%

\begin{equation}
\left[  \left\langle \frac{k_{z_{i}}}{\left\Vert k_{z_{i}}\right\Vert }%
,\frac{k_{z_{j}}}{\left\Vert k_{z_{j}}\right\Vert }\right\rangle \right]
_{i,j=1}^{\infty}. \label{GR}%
\end{equation}
The spaces to which B\"{o}e's theorem applies are those where the kernel has
the complete Nevanlinna-Pick property, which we have already noted holds in
our case, and which have the following additional technical property. Whenever
we have a sequence for which the matrix (\ref{GR}) is bounded on $\ell^{2}$
then the matrix with absolute values%

\[
\left[  \left\vert \left\langle \frac{k_{z_{i}}}{\left\Vert k_{z_{i}%
}\right\Vert },\frac{k_{z_{j}}}{\left\Vert k_{z_{j}}\right\Vert }\right\rangle
\right\vert \right]  _{i,j=1}^{\infty}%
\]
is also bounded on $\ell^{2}.$ This property holds in our case because, for
$\sigma$ in the range of interest, $\operatorname{Re}\left(  \frac
{1}{1-\overline{z_{j}}\cdot z_{i}}\right)  ^{2\sigma}\approx\left\vert
\frac{1}{1-\overline{z_{j}}\cdot z_{i}}\right\vert ^{2\sigma}$ which, as noted
in \cite{Boe}, insures that the Gramm matrix has the desired property. (It is
this step that precludes considering $\sigma=1/2.)$ Finally, by Proposition
9.5 of \cite{AgMc2}, the boundedness on $\ell^{2}$ of the Grammian matrix is
equivalent to $\mu_{Z}=\sum_{j=1}^{\infty}\left\Vert k_{z_{j}}\right\Vert
^{-2}\delta_{z_{j}}=\sum_{j=1}^{\infty}(1-\left\vert z_{j}\right\vert
^{2})^{2\sigma}\delta_{z_{j}}$ being a Carleson measure. Thus the obvious
generalization to higher dimensions of the interpolation theorem of B\"{o}e in
\cite{Boe} completes the proof. (B\"{o}e presents his work in dimension $n=1,$
but, as he notes, the proof extends to spaces with the above properties.$)$

(When we defined "interpolating sequence" we required that $R$ map into and
onto $\ell^{2}(Z,\mu_{Z}).$ In the most well known case, the classical Hardy
space, $n=1,\sigma=1/2,$ if $R$ is onto it must be into. However for the
classical Dirichlet space the map can be onto without being into. Hence one
can ask for a characterization of those maps for which $R$ is onto. The
question is open; partial results are in \cite{Bis}, \cite{Boe}, and
\cite{ArRoSa3}.)

\subsubsection{The Drury-Arveson Hardy space and von Neumann's inequality}

It is a celebrated result of von Neumann \cite{Neu} that if $T$ is a
contraction on a Hilbert space and $f$ is a complex polynomial then
$\left\Vert f(T)\right\Vert \leq\sup\left\{  \left\vert f(\gamma)\right\vert
:|\gamma|=1\right\}  .$ An extension of this to $n$-tuples of operators was
given by Drury \cite{Dru}. Let $A=\left(  A_{1},...,A_{n}\right)  $ be an
$n$-(row)-contraction on a complex Hilbert space $\mathcal{H}$, i.e. an
$n$-tuple of commuting linear operators on $\mathcal{H}$ satisfying
\[
\sum_{j=1}^{n}\left\Vert A_{j}h\right\Vert ^{2}\leq\left\Vert h\right\Vert
^{2}\text{ for all }h\in\mathcal{H}.
\]
Drury showed in \cite{Dru} that if $f$ is a complex polynomial on
$\mathbb{C}^{n}$ then
\begin{equation}
\sup_{\substack{A\;\text{an } \\n\text{-contraction}}}\left\Vert f\left(
A\right)  \right\Vert =\left\Vert f\right\Vert _{M_{\mathcal{K}\left(
\mathbb{B}_{n}\right)  }}, \label{Druryshift}%
\end{equation}
where $\left\Vert f\left(  A\right)  \right\Vert $ is the operator norm of
$f\left(  A\right)  $ on $\mathcal{H}$, and $\left\Vert f\right\Vert
_{M_{\mathcal{K}\left(  \mathbb{B}_{n}\right)  }}$ denotes the multiplier norm
of the polynomial $f$ on Drury's Hardy space of holomorphic functions
\[
\mathcal{K}\left(  \mathbb{B}_{n}\right)  =\left\{  \sum_{k}a_{k}z^{k}%
,\;z\in\mathbb{B}_{n}:\sum_{k}\left\vert a_{k}\right\vert ^{2}\frac
{k!}{\left\vert k\right\vert !}<\infty\right\}  .
\]
This space is denoted $H_{n}^{2}$ by Arveson in \cite{Arv} (who also proves
(\ref{Druryshift}) in Theorem 8.1). For $n=1$, $M_{\mathcal{K}\left(
\mathbb{B}_{n}\right)  }=H^{\infty}\left(  \mathbb{B}_{n}\right)  $ and this
is the classical result of von Neumann. However, as we mentioned, for $n\geq2$
the multiplier space $M_{\mathcal{K}\left(  \mathbb{B}_{n}\right)  }$ is
strictly smaller than $H^{\infty}\left(  \mathbb{B}_{n}\right)  $.\ 

Chen \cite{Che} has shown that the Drury-Arveson Hardy space $\mathcal{K}%
\left(  \mathbb{B}_{n}\right)  =H_{n}^{2}$ is isomorphic to the Besov-Sobolev
space $B_{2}^{1/2}\left(  \mathbb{B}_{n}\right)  $ which can be characterized
as consisting of those holomorphic functions $\sum_{k}a_{k}z^{k}$ in the ball
with coefficients $a_{k}$ satisfying
\[
\sum_{k}\left\vert a_{k}\right\vert ^{2}\frac{\left\vert k\right\vert
^{n-1}\left(  n-1\right)  !k!}{\left(  n-1+\left\vert k\right\vert \right)
!}<\infty.
\]
Indeed, the coefficient multipliers in the two previous conditions are easily
seen to be comparable for $k>0$. The comparability of the multiplier norms
follows:
\[
\left\Vert f\right\Vert _{M_{\mathcal{K}\left(  \mathbb{B}_{n}\right)  }%
}\approx\left\Vert f\right\Vert _{M_{B_{2}^{1/2}\left(  \mathbb{B}_{n}\right)
}}.
\]
Hence using Theorem \ref{Arvcom}, i.e. (\ref{s'}), and Theorem
\ref{multipliers} we can give explicit estimates for the function norm in
Drury's result. Note however that we only have equivalence of the Hilbert
space norms and multiplier space norms, not equality, and that distinction
persists in, for instance, the theorem which follows.

\begin{theorem}
For any $m>\frac{n-1}{2}$ set $d\mu_{f}^{m}\left(  z\right)  =\left\vert
f^{\left(  m\right)  }\left(  z\right)  \right\vert ^{2}\left(  1-\left\vert
z\right\vert ^{2}\right)  ^{2m-n}dz.$ We have
\begin{align}
\sup_{\substack{A\;\text{an } \\n\text{-contraction}}}\left\Vert f\left(
A\right)  \right\Vert  &  \approx\left\Vert f\right\Vert _{\infty}%
+\sup_{\alpha\in\mathcal{T}_{n}}\sqrt{2^{d\left(  \alpha\right)  }I^{\ast}%
\mu_{f}^{m}\left(  \alpha\right)  }\label{vND}\\
&  +\sup_{\alpha\in\mathcal{T}_{n}}\sqrt{\frac{1}{I^{\ast}\mu_{f}^{m}\left(
\alpha\right)  }\sum_{k\geq0}\sum_{\gamma\geq\alpha}2^{d\left(  \gamma\right)
-k}\sum_{\substack{\delta,\delta^{\prime}\in\mathcal{G}^{\left(  k\right)
}\left(  \gamma\right)  }}I^{\ast}\mu_{f}^{m}\left(  \delta\right)  I^{\ast
}\mu_{f}^{m}\left(  \delta^{\prime}\right)  },\nonumber
\end{align}
for all polynomials $f$ on $\mathbb{C}^{n}$.
\end{theorem}

The right side of (\ref{vND}) can of course be transported onto the ball using
that $\cup_{\beta\geq\alpha}K_{\beta}$ is an appropriate nonisotropic tent in
$\mathbb{B}_{n}$, and that $2^{-d\left(  \alpha\right)  }\approx1-\left\vert
z\right\vert ^{2}$ for $z\in K_{\alpha}$.

In passing we mention that, inspired partly by the work of Arveson in
\cite{Arv}, the space $H_{n}^{2}$ plays a substantial role in modern operator
theory. For more recent work see for instance, \cite{AgMc2}, \cite{BTV}, and
\cite{EP} .

\subsubsection{Carleson measures for Hilbert spaces with a complete \emph{N-P}
kernel\label{complete}}

The \emph{universal} complete Nevanlinna-Pick property of the Drury-Arveson
space $H_{n}^{2}$ allows us to use our description of Carleson measures for
$H_{n}^{2}$ to describe Carleson measures for certain other Hilbert spaces. In
\cite{AgMc}, Agler and M$^{c}$Carthy consider Hilbert spaces with a complete
Nevanlinna-Pick kernel $k\left(  x,y\right)  $. We recall their setup, keeping
in mind the classical model of the Szeg\"{o} kernel $k\left(  x,y\right)
=\frac{1}{1-\overline{x}y}$ on the unit disc $\mathbb{B}_{1}$. Let $X$ be an
infinite set and $k\left(  x,y\right)  $ be a positive definite kernel
function on $X$, i.e. for all finite subsets $\left\{  x_{i}\right\}
_{i=1}^{m}$of $X$,
\[
\sum_{i,j=1}^{m}a_{i}\overline{a_{j}}k\left(  x_{i},x_{j}\right)  \geq0\text{
with equality }\Leftrightarrow\text{ all }a_{i}=0.
\]
Denote by $\mathcal{H}_{k}$ the Hilbert space obtained by completing the space
of finite linear combinations of $k_{x_{i}}$'s, where $k_{x}\left(
\cdot\right)  =k\left(  x,\cdot\right)  $, with respect to the inner product
\[
\left\langle \sum_{i=1}^{m}a_{i}k_{x_{i}},\sum_{j=1}^{m}b_{j}k_{y_{j}%
}\right\rangle =\sum_{i,j=1}^{m}a_{i}\overline{b_{j}}k\left(  x_{i}%
,y_{j}\right)  .
\]
The kernel $k$ is called a \emph{complete Nevanlinna-Pick kernel} if the
solvability of the matrix-valued Nevanlinna-Pick problem is characterized by
the contractivity of a certain family of adjoint operators $R_{x,\Lambda}$ (we
refer to \cite{AgMc}, \cite{AgMc2} for an explanation of this generalization
of the classical Pick condition).

Let $\mathbb{B}_{n}$ be the open unit ball in $n$-dimensional Hilbert space
$\ell_{n}^{2}$; for $n=\infty$, $\ell_{\infty}^{2}=\ell^{2}(\mathbb{Z}^{+}).$
For $x,y\in\mathbb{B}_{n}$ set $a_{n}\left(  x,y\right)  =\frac{1}%
{1-\left\langle y,x\right\rangle }$ and denote the Hilbert space
$\mathcal{H}_{a_{n}}$ by $H_{n}^{2}$ (so that $H_{n}^{2}=B_{2}^{1/2}\left(
\mathbb{B}_{n}\right)  $ when $n$ is finite). Theorem 4.2 of \cite{AgMc} shows
that if $k$ is an irreducible kernel on $X$, and if for some fixed point
$x_{0}\in X$, the Hermitian form
\[
F\left(  x,y\right)  =1-\frac{k\left(  x,x_{0}\right)  k\left(  x_{0}%
,y\right)  }{k\left(  x,y\right)  k\left(  x_{0},y_{0}\right)  }%
\]
has rank $n$, then $k$ is a complete Nevanlinna-Pick kernel if and only if
there is an injective function $f:X\rightarrow\mathbb{B}_{n}$ and a nowhere
vanishing function $\delta$ on $X$ such that
\[
k\left(  x,y\right)  =\overline{\delta\left(  x\right)  }\delta\left(
y\right)  a_{n}\left(  f\left(  x\right)  ,f\left(  y\right)  \right)
=\frac{\overline{\delta\left(  x\right)  }\delta\left(  y\right)
}{1-\left\langle f\left(  y\right)  ,f\left(  x\right)  \right\rangle }.
\]
Moreover, if this happens, then the map $k_{x}\rightarrow\overline
{\delta\left(  x\right)  }\left(  a_{n}\right)  _{f\left(  x\right)  }$
extends to an isometric linear embedding $T$ of $\mathcal{H}_{k}$ into
$H_{n}^{2}$. If in addition there is a topology on $X$ so that $k$ is
continuous on $X\times X$, then the map $f$ will be a continuous embedding of
$X$ into $\mathbb{B}_{n} $. If $X$ has holomorphic structure and the $k_{x}$
are holomorphic then $f$ will be holomorphic.

\emph{For the remainder of this subsubsection we will assume that }%
$X$\emph{\ is a topological space and that the kernel function }$k$\emph{\ is
continuous on }$X\times X.$

In that context we can define a Carleson measure for $\mathcal{H}_{k}$ to be a
positive Borel measure on $X$ for which we have the embedding
\begin{equation}
\int_{X}\left\vert h\left(  x\right)  \right\vert ^{2}d\mu\left(  x\right)
\leq C\left\Vert h\right\Vert _{\mathcal{H}_{k}}^{2},\;\;\;\;\;h\in
\mathcal{H}_{k}, \label{HkCar}%
\end{equation}
with the standard definition of the Carleson norm. We can now use the
description of the Carleson measure norm for $H_{n}^{2}=B_{2}^{1/2}\left(
\mathbb{B}_{n}\right)  ,$ given in (\ref{splittree}) or in (\ref{s'}) if $n$
is finite and by (\ref{bilin}) in any case, to give a necessary and sufficient
condition for $\mu$ defined on $X$ to be a Carleson measure for $\mathcal{H}%
_{k}.$ To see this, consider first the case where the Hermitian form $F$ above
has finite rank ($F$ is positive semi-definite if $k$ is a complete
Nevanlinna-Pick kernel by Theorem 2.1 in \cite{AgMc}). Denote by $f_{\ast}\nu$
the pushforward of a Borel measure $\nu$ on $X$ under the continuous map $f$.
If $\mu$ is a positive Borel measure on $X$ then $\mu$ is $\mathcal{H}_{k}%
$-Carleson, i.e. (\ref{HkCar}), if and only if the measure $\mu^{\natural
}=f_{\ast}\left(  \left\vert \delta\right\vert ^{2}\mu\right)  $ is $H_{n}%
^{2}$-Carleson, i.e.
\begin{equation}
\int_{\mathbb{B}_{n}}\left\vert G\right\vert ^{2}d\mu^{\natural}\leq
C\left\Vert G\right\Vert _{B_{2}^{1/2}\left(  \mathbb{B}_{n}\right)  }%
^{2},\;\;\;\;\;G\in B_{2}^{1/2}\left(  \mathbb{B}_{n}\right)  . \label{BCar}%
\end{equation}
Indeed, the functions $h=\sum_{i=1}^{m}c_{i}k_{x_{i}}$ are dense in
$\mathcal{H}_{k}.$ Setting $H=Th=\sum_{i=1}^{m}c_{i}\overline{\delta\left(
x_{i}\right)  }\left(  a_{n}\right)  _{f\left(  x_{i}\right)  }$ we have:%
\begin{align*}
\left\Vert h\right\Vert _{\mathcal{H}_{k}}^{2}  &  =\left\langle \sum
_{i=1}^{m}c_{i}k_{x_{i}},\sum_{i=1}^{m}c_{i}k_{x_{i}}\right\rangle
_{\mathcal{H}_{k}}=\sum_{i,j=1}^{m}c_{i}\overline{c_{j}}k\left(  x_{i}%
,x_{j}\right) \\
&  =\sum_{i,j=1}^{m}c_{i}\overline{c_{j}}\overline{\delta\left(  x_{i}\right)
}\delta\left(  x_{j}\right)  a_{n}\left(  f\left(  x_{i}\right)  ,f\left(
x_{j}\right)  \right) \\
&  =\left\langle \sum_{i=1}^{m}c_{i}\overline{\delta\left(  x_{i}\right)
}\left(  a_{n}\right)  _{f\left(  x_{i}\right)  },\sum_{i=1}^{m}c_{i}%
\overline{\delta\left(  x_{i}\right)  }\left(  a_{n}\right)  _{f\left(
x_{i}\right)  }\right\rangle _{H_{n}^{2}}\\
&  =\left\Vert H\right\Vert _{H_{n}^{2}}^{2}.
\end{align*}
Also, the change of variable $f$ yields
\begin{align*}
\int_{X}\left\vert h\left(  y\right)  \right\vert ^{2}d\mu\left(  y\right)
&  =\int_{X}\left\vert \sum_{i=1}^{m}c_{i}k\left(  x_{i},y\right)  \right\vert
^{2}d\mu\left(  y\right) \\
&  =\int_{X}\left\vert \sum_{i=1}^{m}c_{i}\overline{\delta\left(
x_{i}\right)  }a_{n}\left(  f\left(  x_{i}\right)  ,f\left(  y\right)
\right)  \right\vert ^{2}\left\vert \delta\left(  y\right)  \right\vert
^{2}d\mu\left(  y\right) \\
&  =\int_{f\left(  X\right)  }\left\vert H\right\vert ^{2}d\mu^{\natural}%
=\int_{\mathbb{B}_{n}}\left\vert H\right\vert ^{2}d\mu^{\natural},
\end{align*}
and it follows immediately that (\ref{BCar}) implies (\ref{HkCar}).

For the converse, we observe that if $G\in H_{n}^{2}=B_{2}^{1/2}\left(
\mathbb{B}_{n}\right)  $, then we can write $G=H+J$ where $H\in T\left(
\mathcal{H}_{k}\right)  $ and $J$ is orthogonal to the closed subspace
$T\left(  \mathcal{H}_{k}\right)  $. Now since $J$ is orthogonal to all
functions $\overline{\delta\left(  x\right)  }\left(  a_{n}\right)  _{f\left(
x\right)  }$ with $x\in X$, and since $\delta$ is nonvanishing on $X$, we
obtain that $J$ vanishes on the subset $f\left(  X\right)  $ of the ball
$\mathbb{B}_{n}$. Since $\mu^{\natural}$ is carried by $f\left(  X\right)  $
and orthogonal projections have norm $1$, we then have with $H=Th$,
\begin{align*}
\int_{\mathbb{B}_{n}}\left\vert G\right\vert ^{2}d\mu^{\natural}  &
=\int_{\mathbb{B}_{n}}\left\vert H\right\vert ^{2}d\mu^{\natural}=\int
_{X}\left\vert h\right\vert ^{2}d\mu,\\
&  \text{and}\\
\left\Vert h\right\Vert _{\mathcal{H}_{k}}  &  =\left\Vert H\right\Vert
_{H_{n}^{2}}\leq\left\Vert G\right\Vert _{H_{n}^{2}}.
\end{align*}
It follows immediately that (\ref{HkCar}) implies (\ref{BCar}).

We now extend the above characterization to the case of infinite rank. We
first characterize Carleson measures on $H_{\infty}^{2}$ as follows. Given a
finite dimensional subspace $L$ of $\mathbb{C}^{\infty}$, let $P_{L}$ denote
orthogonal projection onto $L$ and set $\mathbb{B}_{L}=\mathbb{B}_{\infty}\cap
L$, which we identify with the complex ball $\mathbb{B}_{n}$, $n=\dim L$. We
say that a positive measure $\nu$ on $\mathbb{B}_{L}$ is $H_{n}^{2}\left(
\mathbb{B}_{L}\right)  $-Carleson if, when viewed as a measure on
$\mathbb{B}_{n}$, $n=\dim L$, it is $H_{n}^{2}\left(  \mathbb{B}_{n}\right)  $-Carleson.

\begin{lemma}
\label{unifdim}A positive Borel measure $\nu$ on $\mathbb{B}_{\infty}$ is
$H_{\infty}^{2}$-Carleson if and only if $\left(  P_{L}\right)  _{*}\nu$ is
uniformly $H_{n}^{2}\left(  \mathbb{B}_{L}\right)  $-Carleson, $n=\dim L$, for
all finite-dimensional subspaces $L$ of $\mathbb{C}^{\infty}$.
\end{lemma}

%

\proof
Suppose that $\left(  P_{L}\right)  _{\ast}\nu$ is uniformly $H_{n}^{2}\left(
\mathbb{B}_{L}\right)  $-Carleson for all finite-dimensional subspaces $L$ of
$\mathbb{C}^{\infty}$, $n=\dim L$. Let
\begin{equation}
g\left(  z\right)  =\sum_{i=1}^{m}c_{i}a_{\infty}\left(  w_{i},z\right)
=\sum_{i=1}^{m}c_{i}\frac{1}{1-\left\langle z,w_{i}\right\rangle }
\label{fform}%
\end{equation}
for a finite sequence $\left\{  w_{i}\right\}  _{i=1}^{m}\subset
\mathbb{B}_{\infty}$ (such functions are dense in $H_{\infty}^{2}$). If we let
$L$ be the linear span of $\left\{  w_{i}\right\}  _{i=1}^{m}$ in
$\mathbb{C}^{\infty} $, then since $g\left(  P_{L}z\right)  =g\left(
z\right)  $, we can view $g$ as a function on both $\mathbb{B}_{\infty}$ and
$\mathbb{B}_{L}$, and from our hypothesis we have
\begin{equation}
\int_{\mathbb{B}_{\infty}}\left\vert g\right\vert ^{2}d\nu=\int_{\mathbb{B}%
_{L}}\left\vert g\right\vert ^{2}d\left(  P_{L}\right)  _{\ast}\nu\leq
C\left\Vert g\right\Vert _{H_{n}^{2}\left(  \mathbb{B}_{L}\right)  }%
^{2}=C\left\Vert g\right\Vert _{H_{\infty}^{2}}^{2}, \label{feq}%
\end{equation}
with a constant $C$ independent of $g$. Since such functions $g$ are dense in
$H_{\infty}^{2}$, we conclude that $\nu$ is $H_{\infty}^{2}$-Carleson.
Conversely, given a subspace $L$ and a measure $\nu$ that is $H_{\infty}^{2}%
$-Carleson, functions of the form (\ref{fform}) with $\left\{  w_{i}\right\}
_{i=1}^{m}\subset\mathbb{B}_{L}$ are dense in $H_{n}^{2}\left(  \mathbb{B}%
_{L}\right)  $ and so (\ref{feq}) shows that $\left(  P_{L}\right)  _{\ast}%
\nu$ is a $H_{n}^{2}\left(  \mathbb{B}_{L}\right)  $-Carleson measure on
$\mathbb{B}_{L}$ with constant $C$ independent of $L$, $n=\dim L$.

\medskip

The above lemma together with Lemma \ref{carllem} below now yields the
following characterization of Carleson measures on any Hilbert space
$\mathcal{H}_{k}$ with a complete continuous irreducible Nevanlinna-Pick
kernel $k$. Note that the irreducibility assumption on $k$ can be removed
using Lemma 1.1 of \cite{AgMc}.

\begin{theorem}
\label{completeNP}With notation as above let $k$ be a complete continuous
irreducible Nevanlinna-Pick kernel on a set $X$ and $rank\left(  F\right)
=n.$

If $n<\infty$ then a positive measure $\mu$ on $X$ is $\mathcal{H}_{k}%
$-Carleson if and only if $\mu^{\natural}=f_{\ast}(\left\vert \delta
\right\vert ^{2}\mu)$ is $B_{2}^{1/2}\left(  \mathbb{B}_{n}\right)
$-Carleson. That will hold if and only if $\mu^{\natural}$ satisfies
(\ref{bilin}) or, equivalently, (\ref{simp}) and (\ref{splittree}).

For $n=\infty,$ for each finite dimensional subspace $L$ of $\mathbb{C}%
^{\infty}$ set
\[
\mu_{L}=\left(  P_{L}\right)  _{\ast}f_{\ast}(\left\vert \delta\right\vert
^{2}\mu)=\left(  P_{L}\circ f\right)  _{\ast}(\left\vert \delta\right\vert
^{2}\mu).
\]
A measure $\mu$ on $X$ is $\mathcal{H}_{k}$-Carleson if and only if there is a
positive constant $C$ such that for all $L$
\[
\left\Vert \mu_{L}\right\Vert _{Carleson}\leq C,
\]
Here $\left\Vert \nu\right\Vert _{Carleson}$ denotes the norm of the embedding
$H_{\dim L}^{2}\left(  \mathbb{B}_{L}\right)  \subset L^{2}\left(  \nu\right)
.$ This holds if and only if (\ref{bilin}) holds (with $\mathbb{B}_{n}$ taking
the role of $\mathbb{B}_{L})$ uniformly in $L.$
\end{theorem}

Because the comparability constants implicit in our proof of (\ref{s'}) depend
on dimension we cannot use the right side of (\ref{s'}) in place of
$\left\Vert \mu_{L}\right\Vert _{Carleson}$ above.

\subsubsection{Measures supported on embedded two-manifolds\label{curves}}

In the previous discussion we began with a set $\Omega$ and kernel function
$k$ which satisfied conditions which insured that $k$ could be obtained
through a function $f$ mapping $\Omega$ into some $\mathbb{B}_{n}.$
Alternatively we can start the analysis with $\Omega$ and $f$. Given a set
$\Omega$ and an injective map $f$ of $\Omega$ into $\mathbb{B}_{n}$ set
$k(x,y)=a_{n}(f(x),f(y)).$ These kernels generate a Hilbert space
$\mathcal{H}_{k}$ with a complete Nevanlinna-Pick kernel and the previous
theorem describes the Carleson measures of $\mathcal{H}_{k}.$ During that
proof we also showed that the map $T$ which takes $k(x,\cdot)$ to
$a_{n}(f(x),\cdot)$ extends to an isometric isomorphism of $\mathcal{H}_{k}$
to the closed span of $\left\{  \left(  a_{n}\right)  _{f\left(  x\right)
}:x\in\Omega\right\}  $ in $H_{n}^{2}.$ The orthogonal complement of that set
is $V_{f(\Omega)},$ the subspace of $H_{n}^{2}$ consisting of functions which
vanish on $f(\Omega).$ We have
\begin{align}
T\left(  \mathcal{H}_{k}\right)   &  =\text{closed span of }\left\{  \left(
a_{n}\right)  _{f\left(  x\right)  }:x\in\Omega\right\} \label{tee}\\
&  =\{h\in H_{n}^{2}:h(f(x))=0\text{ }\forall x\in\Omega\}^{\bot}\nonumber\\
&  =\left(  V_{f(\Omega)}\right)  ^{\bot}=H_{n}^{2}/V_{f(\Omega)}.\nonumber
\end{align}
The quotient $H_{n}^{2}/V_{f(\Omega)}$ can be regarded as a space of functions
on $f(\Omega)$ normed by the quotient norm. That space is isometrically
isomorphic to $\mathcal{H}_{k}$ under the mapping which takes $[h]$ in
$H_{n}^{2}/V_{f(\Omega)}$ to $h\circ f$ in $\mathcal{H}_{k}.$

We now investigate such embeddings for simple $\Omega.$ The $L^{2}$ Sobolev
space on $\left[  0,1\right]  $ is an example with 1-dimensional $\Omega
$.\ However for this space, and similar 1-dimensional examples, the Carleson
measure theory is trivial; a measure is a Carleson measure if and only if it
has finite mass. This is reflected in the fact that the associated mapping $f
$ of $\left[  0,1\right]  $ into $\mathbb{B}_{\infty}$ maps the interval into
a \emph{proper }sub-ball. (The mappings $f$ associated with this and similar
examples are described in the final section of \cite{BTV}.)

Suppose $\Omega\ $is a bounded domain in the plane and $\partial\Omega$
consists of finitely many smooth curves. (We leave to the reader the
straightforward extension to nonplanar domains.) Let $f$ be a nonsingular
$C^{1}$ embedding of $\Omega$ into $\mathbb{B}_{n};$ $\mathcal{S=}%
f\mathcal{(}\Omega\mathcal{)}$. Suppose $f$ extends to a $C^{1}$ map of
$\bar{\Omega}$ into $\overline{\mathbb{B}_{n}}$ with $\Gamma=\partial
\mathcal{\bar{S}=}f(\partial\bar{\Omega})\subset\partial\overline
{\mathbb{B}_{n}}.$ We will say $\mathcal{S}$ meets the boundary transversally
if
\begin{equation}
\operatorname{Re}\left\langle f^{\prime}\left(  x\right)  \mathbf{n},f\left(
x\right)  \right\rangle \neq0,\;\;\;\;\;x\in\partial\bar{\Omega},
\label{transembed}%
\end{equation}
where $\mathbf{n}$ denotes the unit outward normal vector to $\partial
\bar{\Omega}$, and $f\left(  x\right)  $ is of course the unit outward normal
vector to $\partial\mathbb{B}_{n}$. In order to discuss various geometric
notions of contact at the boundary, we also introduce the unit tangent vector
$\mathbf{T}$ to $\partial\bar{\Omega}$ that points in the positive direction,
i.e. $\mathbf{T}=i\mathbf{n}$ if the tangent space to $\mathbb{R}^{2}$ is
identified with the complex plane in the usual way. Since the vector
$f^{\prime}\left(  x\right)  \mathbf{T}$ is tangent to $\Gamma$, we always
have
\[
\operatorname{Re}\left\langle f^{\prime}\left(  x\right)  \mathbf{T},f\left(
x\right)  \right\rangle =0,\ x\in\partial\bar{\Omega}.
\]
It may also hold that the curve $\Gamma$ is a complex tangential curve, that
is, its tangent lies in the complex tangential tangent direction. This means
that the tangent to $\Gamma$ is perpendicular to the tangential slice
direction $if\left(  x\right)  $, i.e. $\operatorname{Re}\left\langle
f^{\prime}\left(  x\right)  \mathbf{T},if\left(  x\right)  \right\rangle =0$
for all $x\in\partial\bar{\Omega}$, or equivalently
\begin{equation}
\operatorname{Im}\left\langle f^{\prime}\left(  x\right)  \mathbf{T},f\left(
x\right)  \right\rangle =0,\text{ \ }x\in\partial\bar{\Omega}.
\label{comptanembed}%
\end{equation}
We will say that at the boundary $\mathcal{S}$ is perpendicular to the
tangential slice direction and it meets the boundary in the \emph{complex
tangential directions}. At the other extreme it may be that $\mathcal{S}$
satisfies (\ref{transembed}) and meets the boundary \emph{transverse} to the
complex tangential directions, i.e. $f^{\prime}\left(  x\right)  \mathbf{T}$,
the tangent to $\Gamma,$ always has a component in the direction $if\left(
x\right)  $;
\begin{equation}
\operatorname{Im}\left\langle f^{\prime}\left(  x\right)  \mathbf{T},f\left(
x\right)  \right\rangle =\operatorname{Re}\left\langle f^{\prime}\left(
x\right)  \mathbf{T},if\left(  x\right)  \right\rangle \neq0,\text{ \ }%
x\in\partial\bar{\Omega}, \label{transcomp}%
\end{equation}
In particular this applies to a \emph{holomorphic} curve, i.e. $\Omega
\subset\mathbb{C}$ and $f$ is holomorphic, that satisfies (\ref{transembed})
since then we have that $f^{\prime}(z)$ is \emph{complex} linear and
\begin{align}
\operatorname{Im}\left\langle f^{\prime}\left(  z\right)  \mathbf{T},f\left(
z\right)  \right\rangle  &  =\operatorname{Im}\left\langle f^{\prime}\left(
z\right)  i\mathbf{n},f\left(  z\right)  \right\rangle =\operatorname{Im}%
i\left\langle f^{\prime}\left(  z\right)  \mathbf{n},f\left(  z\right)
\right\rangle \label{holoembed}\\
&  =\operatorname{Re}\left\langle f^{\prime}\left(  z\right)  \mathbf{n}%
,f\left(  z\right)  \right\rangle \neq0,\text{ \ }z\in\partial\bar{\Omega
}.\nonumber
\end{align}

Suppose that $\mu$ is a positive measure supported on $\mathcal{S}$ that is
transverse at the boundary. We will show that if we have additional geometric
information about the embedding geometry then the condition for $\mu$ to be a
Carleson measure for $H_{n}^{2}$ can be simplified. Also, as indicated in the
previous subsection, this description can be pulled back to give a description
of measures on $\Omega$ which are Carleson measures for $\mathcal{H}_{k}.$
More precisely we will show that if $\mathcal{S}$ meets the boundary in the
complex tangential directions then $\mu$ is $H_{n}^{2}$-Carleson if and only
if $\mu$ satisfies the tree condition (\ref{treeArv}). On the other hand we
show that if $\mathcal{S}$ meets the boundary transverse to the complex
tangential directions then $\mu$ is $H_{n}^{2}$-Carleson if and only if $\mu$
satisfies the simple condition (\ref{simp}). Finally we will show that if $f$
extends continuously but not differentiably to $\partial\bar{\Omega}$ then
more complicated situations arise.

To prove these results we use the refined tree structure described in
Subsubsection \ref{ring}. It is convenient to begin the analysis with the
second of the two cases.

\paragraph{$\mathcal{S}$ meets the boundary transverse to the complex
tangential directions}

By Theorem \ref{Arvcom}, it is enough to show that when $\mathcal{S}$
satisfies (\ref{transembed}) and (\ref{transcomp}), and $\mu$ is supported on
$\mathcal{S}$ and satisfies the simple condition (\ref{simp}) then for some
$\varepsilon>0$ the $\varepsilon$-split tree condition (\ref{epsilonsplit}) is
satisfied. The transversality hypothesis on $\mathcal{S}$ will permit us to
establish a geometric inequality of the following form:
\[
d^{\ast}\left(  \left[  \alpha\right]  ,\left[  \beta\right]  \right)  \leq
d\left(  \alpha,\beta\right)  -\log_{2}\frac{1}{\left\vert \alpha
-\beta\right\vert }+c,\;\;\;\;\;\text{when }\mathcal{S}\cap K_{\alpha}\neq
\phi\text{, }\mathcal{S}\cap K_{\beta}\neq\phi,
\]
at least for $\alpha,\beta\in\mathcal{T}_{n}$ with $d\left(  \alpha\right)
\approx d\left(  \beta\right)  $ sufficiently large. This in turn will show
that the left side of the $\varepsilon$-split tree condition
(\ref{epsilonsplit}) vanishes for $\varepsilon$ small enough and $d\left(
\alpha\right)  $ large enough, in fact $0<\varepsilon<1/4$ will suffice.

Denote by $P_{z}w$ the projection of $w$ onto the slice $S_{z}$. Suppose that
$\mathcal{S}$ satisfies (\ref{transembed}) and (\ref{transcomp}) and fix
$z,w\in\mathcal{S}\cap\mathbb{B}_{n}$ with $1-\left\vert z\right\vert
\approx1-\left\vert w\right\vert $, where for the remainder of this subsection
the symbol $\approx$ means that the error is small compared to $\left\vert
z-w\right\vert $ times the quantity $\inf_{x\in\partial\Omega}\left\vert
\operatorname{Im}\left\langle f^{\prime}\left(  x\right)  \mathbf{T},f\left(
x\right)  \right\rangle \right\vert $ appearing in (\ref{transcomp}). Then for
$1-\left\vert z\right\vert $ small enough and $\left\vert z-w\right\vert \geq
c\left(  1-\left\vert z\right\vert \right)  $, we have
\begin{equation}
\left\vert z-P_{z}w\right\vert \geq c\left\vert w-P_{z}w\right\vert .
\label{finobtain}%
\end{equation}
Indeed, if $z=f\left(  u\right)  $ and $w=f\left(  v\right)  $, then using
$f\in C^{1}\left(  \overline{\Omega}\right)  $ with (\ref{transembed}) and
(\ref{transcomp}) we obtain $c\left\vert z-w\right\vert \leq\left\vert
u-v\right\vert \leq C\left\vert z-w\right\vert $ and
\[
z-w=f\left(  u\right)  -f\left(  v\right)  \approx f^{\prime}\left(  u\right)
\left(  u-v\right)  .
\]
Now let $x\in\partial\Omega$ be closest to $u$. Using that $u-v\approx
\mathbf{T}\left\vert u-v\right\vert $ we then have
\[
f^{\prime}\left(  u\right)  \left(  u-v\right)  \approx f^{\prime}\left(
x\right)  \left(  u-v\right)  \approx f^{\prime}\left(  x\right)
\mathbf{T}\left\vert u-v\right\vert .
\]
Since $\left\vert z-w\right\vert \geq c\left(  1-\left\vert z\right\vert
\right)  $, we also have $f\left(  x\right)  \approx f\left(  u\right)  =z$,
and altogether then (\ref{transcomp}) yields
\[
\left\vert \operatorname{Im}\left\langle z-w,z\right\rangle \right\vert
\approx\left\vert \operatorname{Im}\left\langle f^{\prime}\left(  x\right)
\mathbf{T},f\left(  x\right)  \right\rangle \right\vert \left\vert
u-v\right\vert \geq c\left\vert u-v\right\vert \geq c\left\vert z-w\right\vert
.
\]
Thus we obtain (\ref{finobtain}) as follows:
\begin{align*}
\left\vert z-P_{z}w\right\vert  &  =\left\vert z-\frac{\left\langle
w,z\right\rangle }{\left\langle z,z\right\rangle }z\right\vert =\frac
{1}{\left\vert z\right\vert }\left\vert \left\langle z,z\right\rangle
-\left\langle w,z\right\rangle \right\vert =\frac{1}{\left\vert z\right\vert
}\left\vert \left\langle z-w,z\right\rangle \right\vert \\
&  \geq\left\vert \operatorname{Im}\left\langle z-w,z\right\rangle \right\vert
\geq c\left\vert w-z\right\vert \geq c\left\vert w-P_{z}w\right\vert .
\end{align*}

For $x,y\in\mathbb{B}_{n}$, define $d\left(  x,y\right)  $ to be the
corresponding distance in the Bergman tree $\mathcal{T}_{n}$, i.e. $d\left(
x,y\right)  =d\left(  \alpha,\beta\right)  $ where $x\in K_{\alpha}$ and $y\in
K_{\beta}$, and $d\left(  \left[  x\right]  ,\left[  y\right]  \right)  $ to
be the corresponding distance in the ring tree $\mathcal{R}_{n}$. Recalling
that $1-\left\vert z\right\vert \approx1-\left\vert w\right\vert $, and using
$A\asymp B$ to mean that $A-B$ is bounded
\begin{align}
d^{\ast}\left(  \left[  z\right]  ,\left[  w\right]  \right)   &  \asymp
d^{\ast}\left(  \left[  P_{z}w\right]  ,\left[  w\right]  \right)  \asymp
\log_{\sqrt{2}}\frac{\left\vert w-P_{z}w\right\vert }{\sqrt{1-\left\vert
z\right\vert }}=\log_{2}\frac{\left\vert w-P_{z}w\right\vert ^{2}%
}{1-\left\vert z\right\vert },\label{dapp}\\
d\left(  z,w\right)   &  \geq\max\left\{  d\left(  \left[  z\right]  ,\left[
w\right]  \right)  ,d\left(  z,P_{z}w\right)  \right\} \\
&  \geq\max\left\{  \log_{2}\frac{\left\vert w-P_{z}w\right\vert ^{2}%
}{1-\left\vert z\right\vert },\log_{2}\frac{\left\vert z-P_{z}w\right\vert
}{1-\left\vert z\right\vert }\right\}  -c.
\end{align}
Combined with (\ref{finobtain}) this yields
\begin{align*}
d^{\ast}\left(  \left[  z\right]  ,\left[  w\right]  \right)   &  \leq\log
_{2}\frac{\left\vert w-P_{z}w\right\vert ^{2}}{1-\left\vert z\right\vert
}+C=\log_{2}\frac{\left\vert w-P_{z}w\right\vert }{1-\left\vert z\right\vert
}+\log_{2}\left\vert w-P_{z}w\right\vert +C\\
&  \leq\log_{2}\frac{\left\vert z-P_{z}w\right\vert }{1-\left\vert
z\right\vert }+\log_{2}\left\vert w-z\right\vert +C\\
&  \leq d\left(  z,w\right)  -\log_{2}\frac{1}{\left\vert w-z\right\vert }+C.
\end{align*}
Using
\begin{align*}
d\left(  z,w\right)   &  =d\left(  z\right)  +d\left(  w\right)  -2d\left(
z\wedge w\right)  ,\\
d^{\ast}\left(  \left[  z\right]  ,\left[  w\right]  \right)   &  =d\left(
\left[  z\right]  \right)  +d\left(  \left[  w\right]  \right)  -2d^{\ast
}\left(  \left[  z\right]  \wedge\left[  w\right]  \right)  ,\\
d\left(  z\right)   &  =d\left(  \left[  z\right]  \right)  ,
\end{align*}
together with $d\left(  z\right)  \asymp d\left(  w\right)  $, we obtain
\begin{align}
d\left(  z\wedge w\right)  -d^{\ast}\left(  \left[  z\right]  \wedge\left[
w\right]  \right)   &  =\frac{1}{2}\left[  d^{\ast}\left(  \left[  z\right]
,\left[  w\right]  \right)  -d\left(  z,w\right)  \right] \label{wedges}\\
&  \leq\frac{1}{2}\left[  C-\log_{2}\frac{1}{\left\vert w-z\right\vert
}\right]  ,\nonumber
\end{align}
for $z,w\in\mathcal{S}\cap\partial\mathbb{B}_{n}$ with $1-\left\vert
z\right\vert \approx1-\left\vert w\right\vert $ sufficiently small.

Now let $\alpha,\gamma,\delta,\delta^{\prime}$ and $k$ be as in the left side
of the split tree condition (\ref{splittree}) with $K_{\delta}\cap
\mathcal{S}\neq\phi$ and $K_{\delta^{\prime}}\cap\mathcal{S}\neq\phi$. Thus
$\delta\wedge\delta^{\prime}=\gamma$, $d\left(  \delta\right)  =d\left(
\delta^{\prime}\right)  =d\left(  \gamma\right)  +k+2$, $\left[  A^{2}%
\delta\right]  =\left[  A^{2}\delta^{\prime}\right]  $ and $d^{\ast}\left(
\left[  \delta\right]  ,\left[  \delta^{\prime}\right]  \right)  =4$. Clearly
$\left\vert \delta-\delta^{\prime}\right\vert \leq2^{-\frac{1}{2}d\left(
\gamma\right)  }$ since $\delta,\delta^{\prime}\geq\gamma$. On the other hand
(\ref{wedges}) yields
\[
d\left(  \gamma\right)  -\left(  d\left(  \gamma\right)  +k\right)  \leq
\frac{1}{2}\left[  C-\log_{2}\frac{1}{\left\vert \delta-\delta^{\prime
}\right\vert }\right]  ,
\]
or $\left\vert \delta-\delta^{\prime}\right\vert \geq c2^{-2k}$. Combining
these two inequalities for $\left\vert \delta-\delta^{\prime}\right\vert $
yields
\[
k\geq\frac{1}{4}d\left(  \gamma\right)  -C.
\]
Thus the $\varepsilon$-split tree condition (\ref{epsilonsplit}) for a measure
$\mu$ supported on $\mathcal{S}$ is vacuous (i.e. the left side vanishes) if
$0<\varepsilon<\frac{1}{4}$ and $\alpha\in\mathcal{T}_{n}$ is restricted to
$d\left(  \alpha\right)  $ large enough. Note that we used only the following
consequence of our hypotheses (\ref{transembed}) and (\ref{transcomp}): there
are positive constants $C,\varepsilon,\delta$ such that $\mathcal{S}$ is a
subset of $\mathbb{B}_{n}$ satisfying
\begin{equation}
\left\vert x-P_{x}y\right\vert \geq\varepsilon\left\vert y-P_{x}y\right\vert
,\;\;\;\;\;x,y\in\mathcal{S}, \label{moregenstat}%
\end{equation}
whenever $\left\vert x\right\vert =\left\vert y\right\vert $, $\left\vert
x-y\right\vert \geq C\left(  1-\left\vert x\right\vert \right)  $ and
$1-\left\vert x\right\vert <\delta$. We have thus proved the following proposition.

\begin{proposition}
\label{embedd}Suppose $\mathcal{S}$ is a $C^{1}$ surface that meets
$\partial\mathbb{B}_{n}$ transversely$,$ i.e. (\ref{transembed}) holds, and
suppose further that the curve of intersection $\Gamma$ is transverse to the
complex tangential directions, i.e. (\ref{transcomp}) holds. In particular,
$\mathcal{S}$ could be a holomorphic curve embedded in $\mathbb{B}_{n}$ that
is transverse at the boundary $\partial\mathbb{B}_{n}$. \emph{More generally},
suppose there are positive constants $C,\varepsilon,\delta$ such that
$\mathcal{S}$ is a subset of $\mathbb{B}_{n}$ satisfying (\ref{moregenstat})
whenever $\left\vert x\right\vert =\left\vert y\right\vert $, $\left\vert
x-y\right\vert \geq C\left(  1-\left\vert x\right\vert \right)  $ and
$1-\left\vert x\right\vert <\delta$. Let $\mu$ be a positive measure supported
on $\mathcal{S}$. Then $\mu$ is $H_{n}^{2}$-Carleson if and only if $\mu$
satisfies the simple condition (\ref{simp}).
\end{proposition}

\begin{corollary}
Suppose that $\mathcal{S}=f\left(  \Omega\right)  $ is a $C^{1}$ surface that
meets the boundary $\partial\mathbb{B}_{n}$ transversely and that the curve of
intersection $\Gamma$ is transverse to the complex tangential directions. Let
$\mathcal{H}_{k}$ denote the Hilbert space generated by the kernels $k\left(
z,w\right)  =a_{n}\left(  f\left(  z\right)  ,f\left(  w\right)  \right)  $,
$z,w\in\Omega$. Then the Carleson measures for $\mathcal{H}_{k}$ are
characterized by the simple condition (\ref{classiccar}). In particular this
applies to a Riemann surface $S$ and a $C^{1}$ embedding $f$ of $\bar{S}$ into
$\overline{\mathbb{B}_{n}}$, holomorphic on $S$, with $f\left(  \partial
\bar{S}\right)  \subset\partial\mathbb{B}_{n}$ so that $\mathcal{S}=f\left(
S\right)  $ is transverse at the boundary.
\end{corollary}

\paragraph{$\mathcal{S}$ meets the boundary in the complex tangential
directions}

We now suppose $\mathcal{S}=f\left(  \Omega\right)  $ meets the boundary
transversely and in the \emph{complex tangential directions}, i.e.
$\left\langle f^{\prime}\left(  x\right)  \mathbf{T},f\left(  x\right)
\right\rangle =0$ for $x\in\partial\bar{\Omega}$. It follows from (2.4) of
\cite{AhCo} that
\[
1-\left\langle f\left(  x\right)  ,f\left(  x+\delta\mathbf{T}\right)
\right\rangle =\delta^{2}\frac{\left\vert f^{\prime}\left(  x\right)
\right\vert ^{2}}{2}+o\left(  \delta^{2}\right)  ,\;\;\;\;\;\text{for }%
x\in\partial\bar{\Omega}\text{, as\ }\delta\rightarrow0,
\]
where by $x+\delta\mathbf{T}$ we mean the point in $\partial\bar{\Omega}$ that
is obtained by flowing along $\partial\bar{\Omega}$ from $x$ a distance
$\delta$ in the direction of $\mathbf{T}$. From this we obtain
\begin{equation}
\left\vert z-P_{z}w\right\vert \leq C\left\vert w-P_{z}w\right\vert ^{2}
\label{square}%
\end{equation}
for $z,w\in\mathcal{S}\cap\mathbb{B}_{n}$ with $1-\left\vert z\right\vert
\approx1-\left\vert w\right\vert $ sufficiently small, and $\left\vert
z-w\right\vert \geq c\left(  1-\left\vert z\right\vert \right)  $. Then we
obtain from (\ref{dapp}) that for such $z,w$ we have
\[
d^{\ast}\left(  \left[  z\right]  ,\left[  w\right]  \right)  \asymp d\left(
z,w\right)  .
\]
Hence for $\mu$ supported on $\mathcal{S}$, the operator $T_{\mu}$ in
(\ref{deffracnew}) below satisfies
\[
T_{\mu}g\left(  \alpha\right)  \approx\sum_{\beta\in\mathcal{T}_{n}%
}2^{d\left(  \alpha\wedge\beta\right)  }g\left(  \beta\right)  \mu\left(
\beta\right)  ,\;\;\;\;\;\alpha\in\mathcal{T}_{n},
\]
whose boundedness on $\ell^{2}\left(  \mu\right)  $ is equivalent, by Theorem
\ref{Besov'}, to the tree condition (\ref{treeArv}) with $\sigma=1/2$ i.e.
(\ref{simp}). Thus Theorem \ref{Achar} completes the proof of the following
proposition (once we note that if the simple condition holds for a fixed
Bergman tree then it holds uniformly for all unitary rotations as well).

\begin{proposition}
\label{embedd'}Suppose that $\mathcal{S}$ is a real $2$-manifold embedded in
the ball $\mathbb{B}_{n}$ that meets the boundary transversely and in the
complex tangential directions, i.e. both (\ref{transembed}) and
(\ref{comptanembed}) hold. \emph{More generally}, suppose there are positive
constants $C,c,\delta$ such that $\mathcal{S}$ is a subset of $\mathbb{B}_{n}
$ satisfying (\ref{square}) whenever $\left\vert x\right\vert =\left\vert
y\right\vert $, $\left\vert x-y\right\vert \geq c\left(  1-\left\vert
x\right\vert \right)  $ and $1-\left\vert x\right\vert <\delta$. Let $\mu$ be
a positive measure supported on $\mathcal{S}$. Then $\mu$ is $H_{n}^{2}%
$-Carleson if and only if $\mu$ satisfies the tree condition (\ref{treeArv}).
\end{proposition}

\begin{remark}
This proposition generalizes easily to the case where $\mathcal{S}=f\left(
\Omega\right)  $, $\Omega\subset\mathbb{R}^{k}$, is a real $k$-manifold
embedded in the ball $\mathbb{B}_{n}$ that meets the boundary transversely and
in the complex tangential directions, i.e.
\[
\left\langle f^{\prime}\left(  x\right)  \mathbf{T},f\left(  x\right)
\right\rangle =0,\text{ \ }x\in\partial\bar{\Omega},
\]
for all tangent vectors $\mathbf{T}$ to $\partial\bar{\Omega}$ at $x$.
\end{remark}

For an example of such an embedding let $\Omega=\mathbb{B}_{1}$ with
coordinate $z=x+iy$ and define a mapping into $\mathbb{B}_{2}$ by
$f(z)=(x,y).$ The space $\mathcal{H}_{k}$ is the Hilbert space of functions on
the unit disk with reproducing kernel $k(z,w)=\frac{1}{1-\operatorname{Re}%
(\bar{z}w)}.$ The sublevel sets of this kernel are intersections of the disk
with halfplanes and testing against these kernel functions quickly shows that
the classic Carleson condition (\ref{classiccar}) does not describe the
Carleson measures for this space. However the previous proposition together
with Theorem \ref{completeNP} gives a description of those measures which turn
out to form a subset of the classical Carleson measures. We now provide the details.

Pulling back the kube decomposition from $\mathbb{B}_{2}$ will give a kube
decomposition of $\mathbb{B}_{1}$ and a tree structure on that set of kubes.
However this structure will not be the familiar one from, for instance, Hardy
space theory or from \cite{ArRoSa}. The familiar structure is the following.
We define a set of kubes on $\mathbb{B}_{1}$ by splitting the disk at radii
$r_{n}=1-2^{-n}$ and splitting each ring $\left\{  r_{n}<\left\vert
z\right\vert \leq r_{n+1}\right\}  $ into $2^{n}$ congruent kubes with radial
cuts. The tree structure, $\mathcal{T}$, on this set of kubes is described by
declaring that $\alpha$ is a successor of $\beta$ if the radius through the
center of $\alpha$ cuts $\beta.$ On the other hand $\mathcal{F}$, the kube and
tree structure pulled back from $\mathbb{B}_{2}$ by $f,$ is the following. We
again split the disk into the same rings and again divide each ring into
congruent kubes with radial cuts, but now the number of kubes in that ring is
to be $\left[  2^{n/2}\right]  $. Again the tree structure is described by
declaring that $\alpha$ is a successor of $\beta$ if the radius through the
center of $\alpha$ cuts $\beta.$ Thus the successor sets $S(\alpha
)=\cup_{\beta\succeq\alpha}\beta$ are approximately rectangles of dimension
$2^{-n}\times2^{-n/2},$ roughly comparable to the complements of sublevel sets
of the reproducing kernels for $\mathcal{H}_{k}.$ Note that the number of
descendents of a vertex after $n$ generations is quite different for the two
trees; in the terminology of \cite{ArRoSa2} $\mathcal{F}$ has tree dimension
$1/2$ and $\mathcal{T}$ has tree dimension $1$.

We now compare the classes of measures described by (\ref{treeArv}) for the
two different tree structures. We define $B_{2}^{1/2}(\mathcal{Q)}$ on a tree
$\mathcal{Q}$ by the norm
\[
\left\Vert f\right\Vert _{B_{2}^{1/2}(\mathcal{Q)}}^{2}=\sum_{\alpha
\in\mathcal{Q}:\alpha\neq o}2^{-d\left(  \alpha\right)  }\left\vert f\left(
\alpha\right)  -f\left(  A\alpha\right)  \right\vert ^{2}+\left\vert f\left(
o\right)  \right\vert ^{2},
\]
for $f$ on the tree $\mathcal{Q}$. Here $A\alpha$ denotes the immediate
predecessor of $\alpha$ in the tree $\mathcal{Q}$. We set
\begin{align}
I_{\mathcal{Q}}f\left(  \alpha\right)   &  =\sum_{\beta\in\mathcal{Q}%
:\beta\leq\alpha}f\left(  \beta\right)  ,\label{deflll}\\
I_{\mathcal{Q}}^{\ast}\left(  g\right)  \left(  \alpha\right)   &
=\sum_{\beta\in\mathcal{Q}:\beta\geq\alpha}g\left(  \beta\right)  .\nonumber
\end{align}
We say that $\mu$ is a $B_{2}^{1/2}(\mathcal{Q)}$-Carleson measure on the tree
$\mathcal{Q}$ if $B_{2}^{1/2}(\mathcal{Q)}$ imbeds continuously into $L_{\mu
}^{2}(\mathcal{Q)}$, i.e.
\begin{equation}
\sum_{\alpha\in\mathcal{Q}}I_{\mathcal{Q}}f\left(  \alpha\right)  ^{2}%
\mu\left(  \alpha\right)  \leq C\sum_{\alpha\in\mathcal{Q}}2^{-d\left(
\alpha\right)  }f\left(  \alpha\right)  ^{2},\;\;\;\;\;f\geq0.
\label{Carlesonppem}%
\end{equation}
We know from \cite{ArRoSa} that a necessary and sufficient condition for
(\ref{Carlesonppem}) is the discrete tree condition
\begin{equation}
\sum_{\beta\in\mathcal{Q}:\beta\geq\alpha}2^{d\left(  \beta\right)
}I_{\mathcal{Q}}^{\ast}\mu\left(  \beta\right)  ^{2}\leq CI_{\mathcal{Q}%
}^{\ast}\mu\left(  \alpha\right)  <\infty,\;\;\;\;\;\alpha\in\mathcal{Q}.
\tag{$T_\mathcal{Q}$}\label{NASCem}%
\end{equation}
We note a simpler necessary condition for (\ref{Carlesonppem})
\begin{equation}
2^{d\left(  \alpha\right)  }I_{\mathcal{Q}}^{\ast}\mu\left(  \alpha\right)
\leq C, \tag{$S_\mathcal{Q}$}\label{simple'em}%
\end{equation}
which is obtained using the sum in (\ref{NASCem}) to dominate its largest
term. However, condition (\ref{simple'em}) is not in general sufficient for
(\ref{Carlesonppem}) as evidenced by certain Cantor-like measures $\mu$.

These considerations apply when $\mathcal{Q}$ is either of the two trees,
$\mathcal{T}$ and $\mathcal{F}$ just described on $\mathbb{B}_{1}$. However
the associated geometries are different; we will refer to conditions
associated to $\mathcal{F}$ as "fattened".

\begin{theorem}
Let $\mu$ be a positive measure on the disk $\mathbb{B}_{1}$. Then the
fattened tree condition ($T_{\mathcal{F}}$) implies the standard tree
condition ($T_{\mathcal{T}}$), but not conversely.
\end{theorem}

%

\proof
First we show that the standard tree condition ($T_{\mathcal{T}}$) is not
sufficient for the fattened tree condition ($T_{\mathcal{F}}$), in fact not
even for the fattened simple condition ($S_{\mathcal{F}}$). For this, let
$\rho>-1$ and set
\[
d\mu\left(  z\right)  =\left(  1-\left\vert z\right\vert \right)  ^{\rho}dz.
\]
Then
\[
I_{\mathcal{T}}^{\ast}\mu\left(  \beta\right)  \approx2^{-d\left(
\beta\right)  }\int_{1-2^{-d\left(  \beta\right)  }}^{1}\left(  1-r\right)
^{\rho}dr\approx2^{-d\left(  \beta\right)  }\left(  2^{-d\left(  \beta\right)
}\right)  ^{\rho+1}=2^{-d\left(  \beta\right)  \left(  \rho+2\right)  },
\]
and the left side of ($T_{\mathcal{T}}$) satisfies
\begin{align*}
\sum_{\beta\in\mathcal{T}:\beta\geq\alpha}2^{d\left(  \beta\right)
}I_{\mathcal{T}}^{\ast}\mu\left(  \beta\right)  ^{2}  &  \approx\sum_{\beta
\in\mathcal{T}:\beta\geq\alpha}2^{-d\left(  \beta\right)  \left(
2\rho+3\right)  }\\
&  =\sum_{k=d\left(  \alpha\right)  }^{\infty}2^{k-d\left(  \alpha\right)
}2^{-k\left(  2\rho+3\right)  }\\
&  =2^{-d\left(  \alpha\right)  }\sum_{k=d\left(  \alpha\right)  }^{\infty
}2^{-k\left(  2\rho+2\right)  }\\
&  \approx2^{-d\left(  \alpha\right)  \left(  2\rho+3\right)  },
\end{align*}
which is dominated by
\[
2^{-d\left(  \alpha\right)  \left(  \rho+2\right)  }\approx I_{\mathcal{T}%
}^{\ast}\mu\left(  \alpha\right)
\]
if $\rho>-1$. Thus $\mu$ satisfies the standard tree condition
($T_{\mathcal{T}}$) for all $\rho>-1$. On the other hand,
\[
I_{\mathcal{F}}^{\ast}\mu\left(  a\right)  \approx2^{-\frac{d\left(  a\right)
}{2}}\int_{1-2^{-d\left(  a\right)  }}^{1}\left(  1-r\right)  ^{\rho}%
dr\approx2^{-\frac{d\left(  a\right)  }{2}}\left(  2^{-d\left(  a\right)
}\right)  ^{\rho+1}=2^{-d\left(  a\right)  \left(  \rho+\frac{3}{2}\right)
},
\]
and so the left side of the fattened simple condition ($S_{\mathcal{F}}$)
satisfies
\[
2^{d\left(  a\right)  }I_{\mathcal{F}}^{\ast}\mu\left(  a\right)
\approx2^{d\left(  a\right)  }2^{-d\left(  a\right)  \left(  \rho+\frac{3}%
{2}\right)  }=2^{-d\left(  a\right)  \left(  \rho+\frac{1}{2}\right)  },
\]
which is unbounded if $\rho<-1/2$. So with $-1<\rho<-1/2$, ($T_{\mathcal{T}}$)
holds but not ($S_{\mathcal{F}}$).

Now we turn to proving that the fattened tree condition ($T_{\mathcal{F}}$)
implies the standard tree condition ($T_{\mathcal{T}}$). Decompose the left
side of ($T_{\mathcal{T}}$) into the following two pieces:
\begin{align*}
\sum_{\beta\in\mathcal{T}:\beta\geq\alpha}2^{d\left(  \beta\right)
}I_{\mathcal{T}}^{\ast}\mu\left(  \beta\right)  ^{2}  &  =\sum_{\beta
\in\mathcal{T}:\beta\geq\alpha\text{ and }d\left(  \beta\right)  \leq2d\left(
\alpha\right)  }2^{d\left(  \beta\right)  }I_{\mathcal{T}}^{\ast}\mu\left(
\beta\right)  ^{2}\\
&  +\sum_{\beta\in\mathcal{T}:\beta\geq\alpha\text{ and }d\left(
\beta\right)  >2d\left(  \alpha\right)  }2^{d\left(  \beta\right)
}I_{\mathcal{T}}^{\ast}\mu\left(  \beta\right)  ^{2}\\
&  =A+B.
\end{align*}
Now let $a\in\mathcal{F}$ satisfy $d\left(  a\right)  =2d\left(
\alpha\right)  $ and
\begin{equation}%
{\displaystyle\bigcup_{\beta\in\mathcal{T}:\beta\geq\alpha\text{ and }d\left(
\beta\right)  =2d\left(  \alpha\right)  }}
K_{\beta}\subset K_{a}, \label{union}%
\end{equation}
where by $K_{a}$ for $a\in\mathcal{F}$ we mean the fattened kube in the disk
corresponding to $a$ (it is roughly a $2^{-d\left(  \beta\right)  }%
\times2^{-\frac{d\left(  \beta\right)  }{2}}$ rectangle - which is
$2^{-2d\left(  \alpha\right)  }\times2^{-d\left(  \alpha\right)  }$ - oriented
so that its long side is parallel to the nearby boundary of the disk, and so
that its distance from the boundary is about $2^{-d\left(  a\right)  }$). It
may be that two such adjacent kubes $K_{a}$ and $K_{a^{\prime}}$ are required
to cover the left side of (\ref{union}), but the argument below can be easily
modified to accommodate this upon replacing $\mu$ by $\mu\chi$ where $\chi$
denotes the characteristic function of the successor set $S_{\alpha}%
=\cup_{\beta\in\mathcal{T}:\beta\geq\alpha}K_{\beta}$ and noting from
(\ref{Carlesonppem}) that if $\mu$ satisfies ($T_{\mathcal{F}}$) then so does
$\mu\chi$. Then we have
\begin{align*}
B  &  =\sum_{\beta\in\mathcal{T}:\beta\geq\alpha\text{ and }d\left(
\beta\right)  >2d\left(  \alpha\right)  }2^{d\left(  \beta\right)
}I_{\mathcal{T}}^{\ast}\mu\left(  \beta\right)  ^{2}\\
&  \leq\sum_{b\in\mathcal{F}:b\geq a}2^{d\left(  b\right)  }\sum_{\beta
\in\mathcal{T}:K_{\beta}\subset K_{b}}I_{\mathcal{T}}^{\ast}\mu\left(
\beta\right)  ^{2}\\
&  \leq\sum_{b\in\mathcal{F}:b\geq a}2^{d\left(  b\right)  }\left(
\sum_{\beta\in\mathcal{T}:K_{\beta}\subset K_{b}}I_{\mathcal{T}}^{\ast}%
\mu\left(  \beta\right)  \right)  ^{2}\\
&  \leq\sum_{b\in\mathcal{F}:b\geq a}2^{d\left(  b\right)  }I_{\mathcal{F}%
}^{\ast}\mu\left(  b\right)  ^{2}.
\end{align*}
The fattened tree condition ($T_{\mathcal{F}}$) shows that the final term
above is dominated by $CI_{\mathcal{F}}^{\ast}\mu\left(  a\right)  $, which is
at most $CI_{\mathcal{T}}^{\ast}\mu\left(  \alpha\right)  $, and hence we
have
\[
B\leq CI_{\mathcal{T}}^{\ast}\mu\left(  \alpha\right)  .
\]
To handle term $A$ we write the geodesic in $\mathcal{F}$ consisting of $a$
together with the $d\left(  \alpha\right)  $ terms immediately preceding $a$
in $\mathcal{F}$ as
\[
\left\{  a_{d\left(  \alpha\right)  },a_{d\left(  \alpha\right)
+1},...,a_{2d\left(  \alpha\right)  }=a\right\}  ,
\]
where $d\left(  a_{k}\right)  =k$ and $a_{k}<a_{k+1}$. Then
\begin{align*}
A  &  \leq\sum_{k=d\left(  \alpha\right)  }^{2d\left(  \alpha\right)  }%
2^{k}\sum_{\beta\in\mathcal{T}:\beta\geq\alpha\text{ and }d\left(
\beta\right)  =k}I_{\mathcal{T}}^{\ast}\mu\left(  \beta\right)  ^{2}\\
&  \leq\sum_{k=d\left(  \alpha\right)  }^{2d\left(  \alpha\right)  }%
2^{k}\left(  \sum_{\beta\in\mathcal{T}:\beta\geq\alpha\text{ and }d\left(
\beta\right)  =k}I_{\mathcal{T}}^{\ast}\mu\left(  \beta\right)  \right)
^{2}\\
&  \leq\sum_{k=d\left(  \alpha\right)  }^{2d\left(  \alpha\right)  }%
2^{k}I_{\mathcal{F}}^{\ast}\left(  \chi\mu\right)  \left(  a_{k}\right)  ^{2}.
\end{align*}
Now for $j\geq0$, let $E_{j}$ consist of those integers $k$ in $\left[
d\left(  \alpha\right)  ,2d\left(  \alpha\right)  \right]  $ satisfying
\begin{equation}
2^{-j-1}I_{\mathcal{F}}^{\ast}\left(  \chi\mu\right)  \left(  a_{d\left(
\alpha\right)  }\right)  <I_{\mathcal{F}}^{\ast}\left(  \chi\mu\right)
\left(  a_{k}\right)  \leq2^{-j}I_{\mathcal{F}}^{\ast}\left(  \chi\mu\right)
\left(  a_{d\left(  \alpha\right)  }\right)  , \label{jsat}%
\end{equation}
and provided $E_{j}\neq\phi$, let $k_{j}=\max_{E_{j}}k$ be the largest integer
in $E_{j}$, so that
\begin{equation}
2^{-j-1}I_{\mathcal{F}}^{\ast}\left(  \chi\mu\right)  \left(  a_{d\left(
\alpha\right)  }\right)  <I_{\mathcal{F}}^{\ast}\left(  \chi\mu\right)
\left(  a_{k_{j}}\right)  \leq2^{-j}I_{\mathcal{F}}^{\ast}\left(  \chi
\mu\right)  \left(  a_{d\left(  \alpha\right)  }\right)  . \label{ksat}%
\end{equation}
Using (\ref{jsat}) and (\ref{ksat}), we then have
\begin{align*}
A  &  \leq2\sum_{j\geq0}2^{-j}I_{\mathcal{F}}^{\ast}\left(  \chi\mu\right)
\left(  a_{d\left(  \alpha\right)  }\right)  I_{\mathcal{F}}^{\ast}\left(
\chi\mu\right)  \left(  a_{k_{j}}\right)  \left\{  \sum_{k\in E_{j}}%
2^{k}\right\} \\
&  \leq4\sum_{j\geq0}2^{-j}I_{\mathcal{F}}^{\ast}\left(  \chi\mu\right)
\left(  a_{d\left(  \alpha\right)  }\right)  \left\{  I_{\mathcal{F}}^{\ast
}\left(  \chi\mu\right)  \left(  a_{k_{j}}\right)  2^{k_{j}}\right\} \\
&  \leq CI_{\mathcal{F}}^{\ast}\left(  \chi\mu\right)  \left(  a_{d\left(
\alpha\right)  }\right)  ,
\end{align*}
where the last line follows from the fattened simple condition
($S_{\mathcal{F}}$) applied to $a_{k_{j}}$ since $d\left(  a_{k_{j}}\right)
=k_{j}$. Since
\[
I_{\mathcal{F}}^{\ast}\left(  \chi\mu\right)  \left(  a_{d\left(
\alpha\right)  }\right)  \leq CI_{\mathcal{T}}^{\ast}\mu\left(  \alpha\right)
,
\]
we have altogether,
\[
A+B\leq CI_{\mathcal{T}}^{\ast}\mu\left(  \alpha\right)  ,
\]
which completes the proof that the standard tree condition ($T_{\mathcal{T}}$)
holds when the fattened tree condition ($T_{\mathcal{F}}$) holds.

\paragraph{The embedding is Lipschitz continuous to the boundary but not
$C^{1}$}

In the next section we will see that if $\mathbb{B}_{1}$ is embedded
holomorphically in $\mathbb{B}_{n}$ and the embedding has a transverse $C^{2}
$ extension that takes $\partial\mathbb{B}_{1}$ into $\partial\mathbb{B}_{n}$
then the induced space of functions on the embedded disk is the Hardy space of
the disk. The proof is given for finite $n$ but it only uses the fact that the
kernel functions on the disk have useful second order Taylor expansions; hence
an analog of the result holds if $n=\infty.$ We now give an example where the
embedding extends continuously to the boundary but the induced function space
on the disk is $B_{2}^{\sigma}\left(  \mathbb{B}_{1}\right)  $ with
$0<\sigma<1/2$ and not the Hardy space $B_{2}^{1/2}\left(  \mathbb{B}%
_{1}\right)  .$ In fact in Subsubsection \ref{complete} we saw that there must
be an embedding of the disk into a $\mathbb{B}_{n}$ so that the induced
function space on $\mathbb{B}_{1}$ is $B_{2}^{\sigma}\left(  \mathbb{B}%
_{1}\right)  .$ Here we write the map explicitly and do certain computations.

Pick and fix $\sigma,$ $0<\sigma<1/2.$ We want a map $f$ of $\mathbb{B}_{1}$
into $\mathbb{B}_{\infty}$ so that
\begin{equation}
\frac{1}{\left(  1-\bar{x}y\right)  ^{2\sigma}}=\frac{1}{\left(
1-\overline{f(x)}\cdot f(y)\right)  }. \label{id}%
\end{equation}
Define $c_{n}$ by
\[
1-\left(  1-z\right)  ^{2\sigma}=\sum_{1}^{\infty}c_{n}z^{n}%
\]
and define $f:\mathbb{B}_{1}\rightarrow$ $\mathbb{B}_{\infty}$ by
\[
f(z)=\left(  \sqrt{c_{n}}z^{n}\right)  _{1}^{\infty},
\]
hence (\ref{id}) holds.

We know $c_{n}$ are positive and
\begin{align*}
c_{n}  &  =\left\vert \left(
\begin{array}
[c]{l}%
2\sigma\\
n
\end{array}
\right)  \right\vert =\frac{2\sigma}{n}\left(  1-\frac{2\sigma}{1}\right)
...\left(  1-\frac{2\sigma}{n-1}\right) \\
&  \approx\frac{2\sigma}{n}e^{-\left(  \frac{2\sigma}{1}+...+\frac{2\sigma
}{n-1}\right)  }\approx\frac{2\sigma}{n}e^{-2\sigma\ln n}\approx
n^{-1-2\sigma}.
\end{align*}
Thus $f$ extends continuously to the boundary but, for $z\in\partial
\mathbb{B}_{1},$ $f^{\prime}(z)$ fails to be in $l^{\infty}$ much less
$l^{2}.$ To estimate the behavior of $f$ near the boundary we use the fact
that $1-r^{n}\approx n\left(  1-r\right)  $ for $n\leq\frac{1}{1-r}$ and
estimate
\begin{align*}
\left\vert f(1)-f(r)\right\vert ^{2}  &  =\sum_{n=1}^{\infty}\left\vert
\sqrt{c_{n}}\left(  1-r^{n}\right)  \right\vert ^{2}\\
&  \approx\sum_{n=1}^{\frac{1}{1-r}}n^{-1-2\sigma}\left(  1-r^{n}\right)
^{2}+\sum_{n=\frac{1}{1-r}}^{\infty}n^{-1-2\sigma}\left(  1-r^{n}\right)
^{2}\\
&  \approx\sum_{n=1}^{\frac{1}{1-r}}n^{-1-2\sigma}n^{2}\left(  1-r\right)
^{2}+\sum_{n=\frac{1}{1-r}}^{\infty}n^{-1-2\sigma}\approx\left(  1-r\right)
^{2\sigma},
\end{align*}
so that
\[
\left\vert f(1)-f(r)\right\vert \approx\left(  1-r\right)  ^{\sigma}.
\]
Thus $f$ is $Lip\;\sigma$.

Suppose we now take a point $x$ on the positive real axis near the boundary.
The image point is $f(x)=\left(  \sqrt{c_{n}}x^{n}\right)  $ and the distance
of $f(x)$ to the boundary is
\begin{align*}
1-\left(  \overline{f(x)}\cdot f(x)\right)  ^{1/2}  &  =1-\left(  \sum
c_{n}x^{2n}\right)  ^{1/2}\\
&  =1-\left(  1-\left(  1-x^{2}\right)  ^{2\sigma}\right)  ^{1/2}\\
&  \sim1-\left(  1-\frac{1}{2}(1-x^{2})^{2\sigma}\right) \\
&  \sim(1-x^{2})^{2\sigma}.
\end{align*}
Because $f$ is not differentiable at the boundary our earlier definition of
transverse does not apply.\ However $f$ does fail to be transverse at the
boundary in the sense that
\[
\frac{dist\left(  f(r),\partial\mathbb{B}_{\infty}\right)  }{dist\left(
f(r),f(1)\right)  }=\frac{1-\left(  \overline{f(r)}\cdot f(r)\right)  ^{1/2}%
}{\left\vert f(1)-f(r)\right\vert }\approx\frac{(1-r)^{2\sigma}}{\left(
1-r\right)  ^{\sigma}}=\left(  1-r\right)  ^{\sigma}%
\]
is not bounded below as $r\rightarrow1;$ as it would be if we had
(\ref{transembed})\textbf{.}

Now consider Carleson measures. We know that a measure $\mu$ on the disk is a
Carleson measure for $B_{2}^{\sigma}\left(  \mathbb{B}_{1}\right)  $ if and
only if $f_{\ast}\mu$ is a Carleson measure for $B_{2}^{1/2}\left(
\mathbb{B}_{\infty}\right)  .$ Here we just note that it is straightforward to
check that the simple condition \ref{SCsigma} for $\mu$ corresponds to the
SC($1/2$) condition for $f_{\ast}\mu.$ Fix $x$ in the disk, near the boundary.
The SC($1/2$) condition for $f_{\ast}\mu$ states that the $\mu$ mass of the
set of $y$ for which
\[
\left\vert 1-\overline{f(y)}\cdot\frac{f(x)}{\left\Vert f(x)\right\Vert
}\right\vert \leq1-\left\Vert f(x)\right\Vert .
\]
is dominated by $C\left(  1-\left\Vert f(x)\right\Vert \right)  .$ Using the
closed form for $%
{\textstyle\sum}
c_{n}z^{n}$ to evaluate the norms and the inner product and doing a bit of
algebra we find that set is the same as the set of $y$ for which
\[
\left\vert \left(  1-\frac{x(\overline{y-x})}{1-\left\vert x\right\vert ^{2}%
}\right)  ^{2\sigma}-1\right\vert \leq1-\left(  1-\left\vert x\right\vert
^{2}\right)  ^{2\sigma}.
\]
This is in turn equivalent to $\left\vert y-x\right\vert \leq C(1-\left\vert
x\right\vert )$ which describes a set of $y$'s comparable in size and shape
with the set of $y$ for which
\[
\left\vert 1-\overline{y}\cdot\frac{x}{\left\vert x\right\vert }\right\vert
\leq C\left(  1-\left\vert x\right\vert \right)  .
\]
The conclusion now follows from the comparison $1-\left\Vert f(x)\right\Vert
\sim(1-\left\vert x\right\vert )^{2\sigma}.$

We just studied $f$ using the Euclidean metric for both $\mathbb{B}_{1}$ and
$\mathbb{B}_{\infty}.$ There are other natural metrics in this context. For
fixed $n,$ $\sigma$ we can define the metric $\delta_{\sigma}$ on
$\mathbb{B}_{n}$ by
\begin{align*}
\delta_{\sigma}(x,y)  &  =\sqrt{1-\frac{\left\vert k(x,y)\right\vert ^{2}%
}{k(x,x)k(y,y)}}\\
&  =\sin(\operatorname*{angle}(k(x,\cdot)k(y,\cdot))),
\end{align*}
where $k_{\sigma}(x,y)=k_{\sigma}(x,y)=\left(  1-\overline{x}\cdot y\right)
^{-2\sigma}$ is the reproducing kernel for $B_{2}^{\sigma}\left(
\mathbb{B}_{n}\right)  .$ This is a general construction of a metric
associated with a reproducing kernel Hilbert space and is related to the
themes we have been considering, see Section 9.2 of \cite{AgMc2}. For the
particular map $f=f_{\sigma}$ we defined it is a consequence of (\ref{id}) and
the definitions that $f=f_{\sigma}$ will be an \emph{isometry} from
$(\mathbb{B}_{1},\delta_{\sigma})$ into $(\mathbb{B}_{\infty},\delta_{1/2}).$

\subsubsection{Hardy spaces on planar domains}

Suppose now that $\Omega=R,$ a domain in $\mathbb{C}$ with boundary
$\Gamma=\partial\bar{R}$ consisting of a finite collection of $C^{2}$ curves.
Suppose that $f$ is a holomorphic map of $R$ into some $\mathbb{B}_{n}$, that
$f$ extends to a $C^{1}$ map of $\bar{R}$ into $\mathbb{\bar{B}}_{n}$ which
takes $\Gamma$ into $\mathbb{\partial\bar{B}}_{n}$, which is one to one on
$\Gamma,$ and which satisfies the transversality condition
\begin{equation}
\left\langle f^{\prime}(z),f(z)\right\rangle \neq0\text{ for }z\in
\Gamma\label{trans}%
\end{equation}
which combines (\ref{transcomp}) and (\ref{transembed}); recall that these two
conditions are equivalent for holomorphic embeddings. We denote the space
generated by the kernel functions $k(x,y)=a_{n}(f(x),f(y))$ by $\mathcal{H}%
_{k}(R)$. (This is a minor variation on what was described earlier; here we do
not require that $f$ be injective on $R.)$

We want to study the relation between $\mathcal{H}_{k}(R)$ and the Hardy space
of $R,$ $H^{2}(R),$ which we now define. Let $d\sigma$ be arclength measure on
$\Gamma$ and define $H^{2}=H^{2}(R)$ to be the closure in $L^{2}%
(\Gamma,d\sigma)$ of the subspace consisting of restrictions to $\Gamma$ of
functions holomorphic on $\bar{R}.$ We refer to \cite{Ab} and \cite{Du} for
the basic theory of these spaces. In particular there is a natural isometric
identification of $H^{2}$ as a space of nontangential boundary values of a
certain space of holomorphic functions on $R$, we also denote that space by
$H^{2}.$ The choice of the measure $d\sigma$ is not canonical but all the
standard choices lead to the same space of holomorphic functions on $R$ with
equivalent norms. The Carleson measures for $H^{2}$ are those described by the
classical Carleson condition, measures $\mu$ for which there is a constant $C$
so that for all $r>0,$ $z\in\Gamma$%
\begin{equation}
\mu(B(z,r)\cap R)\leq Cr. \label{classiccar}%
\end{equation}
That is, $\mu$ satisfies the appropriate version of the simple condition
(\ref{simp}). For small positive $\varepsilon$ and $z\in\Gamma$ let
$\varepsilon(z)$ be the inward pointing normal at $z$ of length $\varepsilon.$
Because the norm in $H^{2}$ can be computed as
\[
\overline{\lim_{\varepsilon\rightarrow0^{+}}}\int_{\Gamma}\left\vert
f(z+\varepsilon(z))\right\vert ^{2}d\sigma
\]
and those integrals are, in fact, the integration of $\left\vert f\right\vert
^{2}$ against a measure on $R$ which satisfies (\ref{classiccar}); we have
that $H^{2}$ is saturated with respect to its Carleson measures; $H^{2}$
consists of exactly those holomorphic functions for which
\begin{equation}
\sup\left\{  \int_{R}\left\vert f\right\vert ^{2}d\mu:\mu\in CM(H^{2}),\text{
}\left\Vert \mu\right\Vert _{Carleson}=1\right\}  <\infty. \label{sat}%
\end{equation}
(We note in passing that if a Banach space of holomorphic functions $B$ is
saturated with respect to its Carleson measures then the multiplier algebra of
$B$ will be $H^{\infty}$: thus, by the comments following Theorem 2, the
Besov-Sobolev spaces $B_{2}^{\sigma}\left(  \mathbb{B}_{n}\right)  $ are not
saturated for $0\leq\sigma\leq1/2$ except for the classical Hardy space,
$n=1,\sigma=1/2.)$

We saw earlier (Proposition \ref{embedd}) that as a consequence of the
transversallity condition the Carleson measures for $\mathcal{H}_{k}(R)$ are
exactly those which satisfy (\ref{classiccar}). Hence every $f\in
\mathcal{H}_{k}(R)$ satisfies (\ref{sat}) and thus we have a continuous
inclusion $i$
\begin{equation}
i:\mathcal{H}_{k}(R)\rightarrow H^{2}. \label{inclusion}%
\end{equation}

To this point we have only used that $f$ is $C^{1}.$ We will be able to get
much more precise information about the relation between $\mathcal{H}%
_{k}(R)\ $and $H^{2}$ if we assume that $f$ is $C^{2}.$ We now make that
assumption$.$

The prototype for our analysis is the proof by D. Alpay, M. Putinar, and V.
Vinnakov \cite{APV} that if $R=\mathbb{B}_{1},\ f$ is $C^{2},$ and if the
differential $df$ is nonvanishing, then $\mathcal{H}_{k}(\mathbb{B}_{1}%
)=H^{2}(\mathbb{B}_{1});$ the spaces of functions coincide and the norms are
equivalent. This insures that the spaces have the same multipliers. We know
from the classical theory that the multiplier algebra of $H^{2}(\mathbb{B}%
_{1})$ is $H^{\infty}(\mathbb{B}_{1})$ and hence the multiplier algebra of
$\mathcal{H}_{k}(\mathbb{B}_{1})$ is also $H^{\infty}(\mathbb{B}_{1})$.
Application of the theory of complete Nevanlinna-Pick kernels then gives an
interesting consequence; any bounded holomorphic function on $f(\mathbb{B}%
_{1})$ has a holomorphic extension to all $\mathbb{B}_{n}$ which is bounded
and, in fact, is in the multiplier algebra $M_{B_{2}^{1/2}\left(
\mathbb{B}_{n}\right)  }$. Indeed, this uses the fact that the multiplier
algebra of $\mathcal{H}_{k}(\mathbb{B}_{1})$ is $H^{\infty}(\mathbb{B}_{1})$
as follows. If $h\in H^{\infty}\left(  f\left(  \mathbb{B}_{1}\right)
\right)  $ with norm $1$, then $h\circ f\in H^{\infty}(\mathbb{B}%
_{1})=M_{\mathcal{H}_{k}(\mathbb{B}_{1})}$ with norm $M<\infty$, and thus the
matrices
\[
\left[  \left(  M^{2}-h\left(  f\left(  z_{i}\right)  \right)  \overline
{h\left(  f\left(  z_{j}\right)  \right)  }k\left(  z_{i},z_{j}\right)
\right)  \right]  _{i,j=1}^{m}%
\]
are positive semi-definite for all infinite sequences $\left\{  z_{i}\right\}
_{i=1}^{\infty}$ in $\mathbb{B}_{1}$ and $m$ finite. By the definition of $k$,
this says that the matrices
\[
\left[  \left(  M^{2}-h\left(  w_{i}\right)  \overline{h\left(  w_{j}\right)
}a_{n}\left(  w_{i},w_{j}\right)  \right)  \right]  _{i,j=1}^{m}%
\]
are positive semi-definite for all infinite sequences $\left\{  w_{i}\right\}
_{i=1}^{\infty}$ in $f\left(  \mathbb{B}_{1}\right)  $ and $m$ finite. Taking
$\left\{  w_{i}\right\}  _{i=1}^{\infty}$ to be dense in $f\left(
\mathbb{B}_{1}\right)  $, the Pick property for $H_{n}^{2}$ shows that there
is $\varphi\in M_{H_{n}^{2}}$ with $\left\Vert \varphi\right\Vert
_{M_{H_{n}^{2}}}\leq M$ and that agrees with $h$ on $\left\{  w_{i}\right\}
_{i=1}^{\infty}$, hence on $f\left(  \mathbb{B}_{1}\right)  $ as required. See
\cite{APV} for details. (In fact there is a minor error in that paper; a
nonsingularity hypothesis is needed as shown by the map of $\mathbb{B}_{1}%
$into $\mathbb{B}_{2}$ given by $f(z)=2^{-1/2}(z^{2},z^{3})$.\ For this choice
of $f$ the space $\mathcal{H}_{k}(\mathbb{B}_{1})$ will not contain any $g$
with $g^{\prime}(0)\neq0.$ The hypothesis is needed to insure that the
function $\phi^{-1}$ constructed at the end of Section 3 of \cite{APV} has the
required properties. Also, the continuity properties of the function $L$ in
\cite{APV} follow if $f$ is assumed to be $C^{2}.)$

We know the inclusion $i$ is bounded, we now turn attention to its adjoint
$i^{\ast}$ mapping $H^{2}$ to $\mathcal{H}_{k}.$ We want to compute the norm
$\left\Vert i^{\ast}g\right\Vert _{\mathcal{H}_{k}}$ for $g$ a finite linear
combination of kernel functions. We denote the kernel functions for
$\mathcal{H}_{k}$ by $k_{x}=k(x,\cdot)$ and those for the Hardy space $H^{2}$
by $h_{x}.$ It is a direct computation that for any $x,$ $i^{\ast}h_{x}%
=k_{x}.$ Thus if $g=\sum a_{i}h_{x_{i}}$ then $i^{\ast}(g)=\sum a_{i}k_{x_{i}%
}$ and
\begin{align*}
\left\Vert i^{\ast}g\right\Vert _{\mathcal{H}_{k}}^{2}  &  =\left\langle
{\textstyle\sum_{i}}
a_{i}k_{x_{i}},%
{\textstyle\sum_{j}}
a_{j}k_{x_{j}}\right\rangle _{\mathcal{H}_{k}}\\
&  =%
{\textstyle\sum_{i,j}}
a_{i}\bar{a}_{j}\left\langle k_{x_{i}},k_{x_{j}}\right\rangle _{\mathcal{H}%
_{k}}=%
{\textstyle\sum_{i,j}}
a_{i}\bar{a}_{j}k_{x_{i}}(x_{j}).
\end{align*}
Alternatively, setting $\tilde{S}=ii^{\ast},$ we have
\begin{align*}
\left\Vert i^{\ast}g\right\Vert _{\mathcal{H}_{k}}^{2}  &  =\left\langle
i^{\ast}g,i^{\ast}g\right\rangle _{\mathcal{H}_{k}}=\left\langle \tilde
{S}g,g\right\rangle _{H^{2}}\\
&  =%
{\textstyle\sum_{i,j}}
a_{i}\bar{a}_{j}\left\langle \tilde{S}h_{x_{i}},h_{x_{j}}\right\rangle
_{H^{2}}.
\end{align*}
Thus $\tilde{S}$ is a positive operator on $H^{2}\ $and we know $\tilde{S}$ is
bounded because we know $i$ is bounded. We record the consequence
\begin{equation}
\left\langle \tilde{S}h_{x_{i}},h_{x_{j}}\right\rangle _{H^{2}}=k_{x_{i}%
}(x_{j}). \label{tilda}%
\end{equation}
We now give an integral representation of $\tilde{S}$ and using that show that
$\tilde{S}$ is a Fredholm operator. For $h\in H^{2}$ define $S$ by%

\[
Sh(x)=\int_{\Gamma}k_{\omega}(x)h(\omega)d\sigma(\omega),\text{ \ }x\in R.
\]
In particular, setting $h=h_{x_{i}}$ we have%
\begin{align*}
Sh_{_{x_{i}}}\left(  x\right)   &  =\int_{\Gamma}k_{\omega}\left(  x\right)
h_{_{x_{i}}}\left(  \omega\right)  d\sigma(\omega)\\
&  =\left\langle k_{(\mathbf{\cdot})}\left(  x\right)  ,\overline{h_{x_{i}%
}(\cdot)}\right\rangle _{H^{2}}\\
&  =k_{x_{i}}(x).
\end{align*}
In the last equality we used the fact that $k_{\omega}(x)$ is bounded and
conjugate holomorphic in $\omega$ and that taking the inner product with
$\bar{h}_{x_{i}}$ evaluates such a function at $x_{i}.$ Hence we have
\[
\left\langle Sh_{_{x_{i}}},h_{x_{j}}\right\rangle _{H^{2}}=k_{x_{i}}(x_{j}).
\]
Comparing with (\ref{tilda}) we conclude that $S=\tilde{S}.$ Following
\cite{APV} we now compare the integration kernel for $S$ with the Cauchy
kernel. For $\omega\in\Gamma,\zeta\in R$ we set
\begin{align*}
L(\omega,\zeta)  &  =(\omega-\zeta)k_{\omega}(\zeta)=\frac{(\omega-\zeta
)}{1-f(\zeta)\cdot\overline{f(\omega)}}\\
&  =\frac{(\omega-\zeta)}{\left(  f(\omega)-f(\zeta)\right)  \cdot
\overline{f(\omega)}}.
\end{align*}
The transversality hypothesis insures that $L$ extends continuously to
$\bar{R}\times\bar{R}$ and that for $\omega\in\Gamma$ we have $L(\omega
,\omega)=\left\langle f^{\prime}(\omega),f(\omega)\right\rangle ^{-1},$ a
continuous function that is bounded away from zero. We now write the
integration kernel for $S$ as
\begin{align*}
k_{\omega}(\zeta)  &  =\frac{L(\omega,\zeta)}{\omega-\zeta}\\
&  =\frac{L(\omega,\zeta)-L(\omega,\omega)}{\omega-\zeta}+\frac{L(\omega
,\omega)}{\omega-\zeta}\\
&  =k_{1,\omega}(\zeta)+k_{2,\omega}(\zeta).
\end{align*}
This lets us split $S=S_{1}+S_{2}.$ The hypothesis that $f$ be $C^{2}$ insures
that $k_{1}$ extends $\bar{R}\times\bar{R}$ with
\[
k_{1,\omega}(\omega)=\frac{f^{\prime\prime}(\omega)\cdot\overline{f(\omega)}%
}{2\left(  f^{\prime}(\omega)\cdot\overline{f(\omega)}\right)  ^{2}},
\]
a continuous function, and hence $S_{1}$ is compact. Along $\Gamma$ we can
write $dz(\omega)=v(\omega)d\sigma(\omega)$ for a continuous function $v$
which is bounded away from $0$. Thus%
\[
S_{2}h(x)=\frac{1}{2\pi i}\int_{\Gamma}\frac{1}{x-\omega}s(\omega
)h(\omega)dz(\omega).
\]
with $s$ continuous and bounded away from zero. The operator
\[
Ph=\frac{1}{2\pi i}\int_{\Gamma}\frac{1}{x-\omega}h(\omega)dz(\omega)
\]
gives a bounded projection of $L^{2}(\Gamma,d\sigma)$ onto $H^{2}$ \cite{Ab}.
Thus $S_{2}$ is a Toeplitz operator with a symbol that is continuous and
bounded away from zero. Hence, by the Fredholm theory for Toeplitz operators,
$S_{2}$ is a Fredholm operator \cite{Ab}. Hence $S,$ a compact perturbation of
$S_{2},$ is also Fredholm.

$S$ is a positive Fredholm operator on $H^{2}$. Hence $\operatorname*{Ker}(S)$
is finite dimensional and $\operatorname*{Ran}(S)$ is the closed subspace
$\operatorname*{Ker}(S)^{\perp}.$ The restriction of $S$ to that closed
subspace is an isomorphism of that space; one-to-one, onto, bounded, and
bounded below.

We can identify $\mathcal{H}_{k}$ with $H_{n}^{2}/V_{f(X)}$. In particular if
$P$ is any polynomial on $\mathbb{C}^{n}$ then there is a function in
$\mathcal{H}_{k}$ of the form $\tilde{P}=P\circ f.$ Furthermore we have the
norm estimate%
\[
\left\Vert \tilde{P}\right\Vert _{\mathcal{H}_{k}}=\left\Vert P\right\Vert
_{H_{n}^{2}/V_{f(X)}}\leq\left\Vert P\right\Vert _{H_{n}^{2}}.
\]
Using this and the fact that the polynomials are dense in $H_{n}^{2}$ we
conclude that the set $B_{0}(R)=\{\tilde{P}\}$ is dense in $\mathcal{H}_{k}.$
Let $A(R)$ be the algebra of functions holomorphic in $R$ which extend
continuously to $\bar{R},$ normed by the uniform norm. Because $f$ extends
continuously to $\bar{R}$ the set $B_{0}(R)$ is a subalgebra of $A(R)$. Let
$B(R)$ denote the closure of $B_{0}(R)$ in $A(R)$ and let $\overline{B(R)}$
denote the closure of $B(R)$ (or, equivalently, $B_{0}(R)$) in $H^{2}.$

Suppose now that we have $f\in H^{2}$ in $\operatorname*{Ker}(S).$ We know $i$
is injective hence we must have $i^{\ast}f=0.$ Thus, for any $P\in$ $B_{0}(R),
$ $\left\langle i^{\ast}f,P\right\rangle _{\mathcal{H}_{k}}=\left\langle
f,iP\right\rangle _{H^{2}}=0$. Hence $\operatorname*{Ker}(S)\subset
\overline{B(R)}^{\perp}.$ When we pass to orthogonal complements and recall
that $\operatorname*{Ran}(S)=\operatorname*{Ker}(S)^{\perp}$ we find that
$\operatorname*{Ran}(S)\supset\overline{B(R)}.$ On the other hand $S=ii^{\ast
}$ and hence $\operatorname*{Ran}(S)\subset\operatorname*{Ran}(i).$ We know
$B_{0}(R)$ is $\mathcal{H}_{k}$-dense in $\mathcal{H}_{k}$ and that $i$ is
continuous. Thus we continue the inclusions with $\operatorname*{Ran}%
(S)\subset\operatorname*{Ran}(i)\subset\overline{B(R)}.$ Combining these
ingredients we have $\operatorname*{Ran}(S)=\operatorname*{Ran}(i)=\overline
{B(R)}.$ In particular we have that $i$ is a continuous one-to-one map onto
its closed range and hence must be a norm isomorphism of $\mathcal{H}_{k}(R)$
and $\overline{B(R)}$. In sum we have

\begin{theorem}
Suppose $f,$ $R,$ and $\overline{B(R)}$ are as described. Then $\overline
{B(R)}$ has finite codimension in $H^{2}$ and $i$ is a norm isomorphism
between $\mathcal{H}_{k}(R)$ and $\overline{B(R)}.$

If
\begin{equation}
\dim(A(R)/B(R))=s<\infty. \label{codim}%
\end{equation}
then the codimension of $\overline{B(R)}\ $in $H^{2}$ is $s.$ In particular if
$B_{0}(R)$ is dense in $A(R)$ then $i$ is a norm isomorphism of $\mathcal{H}%
_{k}(R)$ onto $H^{2}.$
\end{theorem}

\begin{proof}
We have established the first statements. Suppose now that (\ref{codim})
holds. By work of T. Gamelin \cite{Ga} we have a complete structural
description of $B(R).$ The algebra $B(R)$ can be obtained from $A(R)$ by a
chain of passages to subalgebras each of codimension one in the previous
subalgebra. Each of these steps is of one of two possible forms. One possible
step consists of selecting two points $x$ and $y$ of $R$ and passing to the
subalgebra of functions which take the same values at $x$ and $y.$ The other
possibility is picking a point $x$ in $R$ and passing to the kernel of a point
derivation (in the algebra considered) supported at $x.$ In particular, at
each step we pass to the kernel of a linear functional which can extended
continuously to $H^{2}$ and is thus given by an inner product with a $h\in
H^{2}.$ (This is because the points $x,y$ were in $R,$ not in $\Gamma.)$ Thus
there are $s$ elements in $H^{2}$ such that $B(R)=A(R)\cap\left(
\operatorname*{span}\{h_{1},...,h_{s}\}\right)  ^{\perp}.$ This insures that
when we pass to closures in $H^{2}$ we will have $\overline{B(R)}=\left(
\operatorname*{span}\{h_{1},...,h_{s}\}\right)  ^{\perp}$ which has
codimension $s.$
\end{proof}

\begin{corollary}
\label{NP}If $R$ is a domain in $\mathbb{C}$ with boundary consisting of a
finite collection of smooth curves then the Hardy space $H^{2}(R)$ admits an
equivalent norm with the property that with the new norm the space is a
reproducing kernel Hilbert space with complete N-P kernel.
\end{corollary}

\begin{proof}
It suffices to find a mapping $f$ of $R$ into some $\mathbb{B}_{n}$ to which
the previous theorem applies and so that $B_{0}(R)$ is dense in $A(R).$ It is
a theorem of Stout \cite{Sto} that one can find a set of three holomorphic
functions $\left\{  f_{i}\right\}  _{i=1,2,3}$ which separate points, with
each $f_{i}$ having modulus identically one on each boundary component and
such that there is no point at which all three functions have vanishing
derivative. In the same paper he shows that under these assumptions the
polynomials in the $f_{i}$ are dense in $A(R).$ We now claim that the mapping
of $z$ to $f(z)=(f_{1}(z),f_{2}(z),f_{3}(z))/\sqrt{3}$ is a map to which the
theorem can be applied. To see that each $f_{i}$ has the required smoothness
note that if one precomposes with a conformal map $\phi$ which takes part of
the unit disc near $1$ to the part of $R$ near a boundary point $z$ then by
the reflection principle the composite is real analytic at $1.$ Hence near $z$
the $f_{i}$ are as smooth as $\phi$ and the local smoothness of $\phi$ is
determined by the smoothness of $\Gamma.$ Also note that, assuming $f_{i}$ is
not constant, $(f_{i}\circ\phi)^{\prime}(1)\neq0,$ for otherwise the image of
a neighborhood of $1$ under the holomorphic map $f_{i}\circ\phi$ would not
stay \emph{inside} the unit disk. In particular, for each nonconstant $f_{i}$
we have $f_{i}^{\prime}\neq0$ on\ $\Gamma.$ To finish we need to verify the
transversality condition (\ref{trans}). For $z\in\Gamma,$ for the nonconstant
$f_{i},$ $f_{i}^{\prime}(z)\overline{f_{i}(z)}\neq0.$ We need to insure that
if several such terms are added there is no cancellation. That follows from
applying the following lemma to each of the $f_{i}$ and noting that the number
$\alpha$ in the lemma is determined by the geometry at the point $z$ but is
independent of the function $g.$
\end{proof}

\begin{lemma}
Suppose $R$ is a domain in $\mathbb{C}$, $\gamma$ is a $C^{2}$ arc forming
part of $\partial\bar{R}$ and $z\in\gamma.$ There is a real number $\alpha$ so
that for any $g$ holomorphic on $R$ with $\left\vert g(z)\right\vert =1$ on
$\gamma,$ $g^{\prime}(z)\neq0$ and $\left\vert g(w)\right\vert \leq1$ on the
intersection of $R$ with a neighborhood of $z$ we have $g^{\prime}%
(z)\overline{g(z)}=e^{i\alpha}r(g)$ for some positive real number $r(g).$
\end{lemma}

\begin{proof}
First consider the case when $\gamma$ is part of the unit circle near $z=1$
and locally $R$ is inside the circle. By the reflection principle $g$ extends
to a holomorphic function on a neighbourhood of $z$ which insures that the
hypotheses about the boundary behavior of $g$ are well formulated. We have
$g(1)=\eta$ with $\left\vert \eta\right\vert =1.$ By conformality and the fact
that $g$ takes part of the circle to part of the circle, the linearization of
$g$ must map the outward pointing normal at $1$ to the outward pointing normal
at $\eta.$ Thus $g^{\prime}(1)=\eta r$ for some positive $r$ and hence
$g^{\prime}(1)\overline{g(1)}=\eta r\bar{\eta}=r$ as required. For the general
case let $\phi$ be a conformal mapping of the part of the unit disc near $1$
to the interior of $R$ near $z$ which takes $1 $ to $z$. If $\gamma$ is
$C^{2}$ near $z$ then $\phi$ will be at least $C^{1}$ at $1$ and thus we can
apply the result from the special case to $g\circ\phi.$ That gives%
\[
0<(g\circ\phi)^{\prime}(1)\overline{(g\circ\phi)(1)}=g^{\prime}(\phi
(1))\phi^{\prime}(1)\overline{g(\phi(1))}=g^{\prime}(z)\phi^{\prime
}(1)\overline{g(z)}%
\]
and $\arg(g^{\prime}(z)\overline{g(z)})$ is independent of $g$ as
required.\medskip
\end{proof}

\begin{remark}
A reason for taking note of this corollary is that, while it is known that the
classical Hardy space of the disc does have a complete N-P kernel, the various
classically defined norms on the Hardy spaces of multiply connected domains do
not have this property. (Actually it is not known if the property always
fails; it is known to fail sometimes and there are no known cases until now
using the classically defined norms where it holds.) Hence it is interesting
that the spaces do carry relatively natural equivalent norms with the
property. See \cite{AgMc2} for further discussion.
\end{remark}

Also, as in \cite{APV}, we obtain extension theorems as follows. Suppose now
that we are in the situation of the previous theorem and (\ref{codim}) holds.
As noted in that proof, we will have $\mathcal{H}_{k}(R)=V^{\perp}$ where $V$
is a \emph{finite} dimensional subspace of $H^{2}(R)$ and the orthogonality is
in $H^{2}(R).$ In that case the multiplier algebra $\mathcal{H}_{k}(R),$
$M_{\mathcal{H}_{k}(R)},$ will be
\[
H^{\infty}(R)\cap\mathcal{H}_{k}(R)=H^{\infty}(R)\cap V^{\perp}=w^{\ast
}\text{-closure of }B(R)\text{ in }H^{\infty}(R).
\]
The facts that multipliers must be bounded and that $1\in\mathcal{H}_{k}(R)$
insure $M_{\mathcal{H}_{k}(R)}$ is contained in that space. On the other hand
if $b\in H^{\infty}(R)$ then $b$ multiplies $\mathcal{H}_{k}(R)$ into
$H^{2}(R).$ We then need to know that if $b$ is also in $V^{\perp}$ and that
if $g\in V^{\perp}$ then $bg,$ which we know to be in $H^{2}(R),$ is also in
$V^{\perp}.$ That is insured by the fact that membership in $V^{\perp}$ is
determined by local conditions which have the form that if two functions
satisfy them then so does the product.

The fact that membership in $V^{\perp}$ is determined locally allows us to
have a more intrinsic description of the multipliers. First note that for any
function $h$ in $H^{\infty}(R)\cap\mathcal{H}_{k}(R),$ the function $Th $
defined on $f(R)$ by $Th(f(z))=h(z)$ is a function on $f(R)$ which can be
obtained by restricting a function in $H_{n}^{2}$ to $f(R).$ This insures that
given $z\in f(R)$ there is a neighbourhood $V_{z}\subset\mathbb{B}_{n}$ and a
holomorphic function $h_{z}^{\ast}$ defined on $V_{z}$ such that $h_{z}^{\ast
}=Th$ on $V_{z}\cap f(R).$ We will say that a function $j$ on $f(R)$ that has
this property, i.e. for each $z$ in $f(R)$ one can find a holomorphic
extension of $j$ to a full neighbourhood of $z$ in $\mathbb{C}^{n},$ has the
\emph{local extension property}. Suppose conversely that $h\in H^{\infty}(R)$
is such that $Th$ has that local extension property. The function $Th$ will
then be the uniform limit on compact subsets of $V_{z}\cap f(R)$ of
polynomials. However any polynomial on $\mathbb{C}^{n}$ when restricted to
$f(R)$ gives a function of the form $Tb$ for some $b\in B(R).$ Thus at each
point of $R$ there is neighbourhood in which $h$ can be locally uniformly
approximated by elements of a $B(R).$That insures that the bounded function
$h$ is in the $w^{\ast}$-closure of $B(R)$ in $H^{\infty}(R)$ which, we just
noted, equals $M_{\mathcal{H}_{k}(R)}.$

We have established the following corollary.

\begin{corollary}
Suppose we are in the situation of the previous theorem and (\ref{codim})
holds. If $h$ is a bounded holomorphic function on $f(R)$ and which has the
local extension property$\ $then there is a bounded function $H$ in $H_{n}%
^{2}$ such that $H$ restricted to $f(R)$ agrees with $h;$ in fact $H$ can be
chosen in $M_{H_{n}^{2}}.$ If the codimension $s=0$ then the local extension
property is automatically satisfied.
\end{corollary}

This result applies, for instance, to the maps $f$ used in the proof of
Corollary \ref{NP}. A different type of example is the following. Pick and fix
$L>1$ and let $R$ be the ring domain $R=R_{L}=\left\{  z:L^{-1}<\left\vert
z\right\vert <L\right\}  .$\ Let $f$ be the mapping of $R_{L}$ into
$\mathbb{B}_{2}$ given by
\[
f(z)=c(z,z^{-1})\text{ with }c=\frac{L^{2}}{1+L^{4}}.
\]
In this case $s=0.$ By the theorem $\mathcal{H}_{k}$ is isomorphic to
$H^{2}(R_{L})$ and by the corollary $f(R_{L})$ has the extension property.

In fact, for this particular map there is no need for a general theorem. We
can define $H^{2}(R_{L})$ using a computationally convenient boundary
measure;\ let $d\sigma_{L^{-1}}$ and $d\sigma_{L}$ be arc length measure on
the two circles which form $\partial\bar{R}$ and set $d\tau=$ $(2\pi
L^{-1})^{-1}d\sigma_{L^{-1}}+(2\pi L)^{-1}d\sigma_{L}\ $giving mass $1$ to
each boundary component. Let $H^{2}(R_{L})$ be the closure in $L^{2}%
(\partial\bar{R},d\tau)$ of the rational functions with poles off $R$ or,
equivalently, the closure of the space of polynomials in $z$ and $1/z.$ The
monomials $\left\{  z^{n}\right\}  _{n=-\infty}^{\infty}$ are an orthogonal
basis for $H^{2}(A)$ and we have
\begin{equation}
\left\Vert z^{n}\right\Vert _{H^{2}(R_{L})}^{2}=L^{2\left\vert n\right\vert
}+L^{-2\left\vert n\right\vert }. \label{norm}%
\end{equation}
\ On the other hand $\mathcal{H}_{k}(R)$ has reproducing kernels
\[
k(z,w)=a_{2}(f(z),f(w))=\frac{1}{1-c\bar{z}w-c\frac{1}{\bar{z}w}}.
\]
The norm on $\mathcal{H}_{k}(R)$ is rotationally invariant and hence the
monomials are again an orthogonal basis. Thus to compare $\mathcal{H}_{k}(R)$
to $H^{2}(R_{L})$ it is enough to compute the norm of the monomials in
$\mathcal{H}_{k}(R).$ Doing a partial fraction decomposition of the
reproducing kernel and then a power series expansion gives \
\[
k(z,w)=\frac{L^{4}+1}{L^{2}-1}\sum_{n=-\infty}^{\infty}\frac{\left(  \bar
{z}w\right)  ^{n}}{L^{2\left\vert n\right\vert }}%
\]
and hence
\[
\left\Vert z^{n}\right\Vert _{\mathcal{H}_{k}}^{2}=\frac{L^{2}-1}{L^{4}%
+1}L^{2\left\vert n\right\vert }.
\]
Comparison with (\ref{norm}) shows that the identity map between the two
spaces is an isomorphism.

It is not clear what the natural hypotheses are to insure that (\ref{codim})
holds, however results of B. Lund \cite{Lu} and E. Stout \cite{Sto} cover a
large category of cases. See also Theorem 3 of E. Bishop in \cite{Bish}.

\begin{theorem}
Suppose $B(R)$ contains a nonconstant function $h_{1}$ which has modulus
identically one on $\partial\bar{R}$. Suppose further that there are
$h_{2},...,h_{n}$ in $B(R)$ so that the mapping $H=(h_{1},h_{2},...,h_{n})$
separates all but finitely many points of $R.$ Then
\[
\dim(A(R)/B(R))=s<\infty.
\]
If in fact $H$ can be chosen so that $H$ separates every pair of points and
the differential $dH$ is nonvanishing then $s=0.$
\end{theorem}

\begin{proof}
The first statement is in \cite{Lu}, the second in \cite{Sto}. (The result in
\cite{Lu} is for the case in which $H$ separates all pairs of points. The
extension to the more general situation is straightforward.)
\end{proof}

Other constructions which can be used to form maps $f$ of interest in this
context are in \cite{Rud2} and \cite{Be}.

\begin{remark}
It was pointed out to us by John M$^{c}$Carthy that by using Corollary
\ref{NP} together with techniques from Chapter 14 of \cite{AgMc2} it is
possible to prove dilation and extension theorems for operators $T$ which have
spectrum in $\bar{R}$ and which satisfy the operator inequality%
\[
I-%
{\textstyle\sum_{i=1}^{3}}
f_{i}(T)f_{i}(T)^{\ast}\geq0,
\]
where the $f_{i}$ are the functions from the proof of Corollary \ref{NP}. We
plan to return to this issue in a later paper.
\end{remark}

\section{Inequalities on trees}

We now recall some of our earlier results in \cite{ArRoSa} and \cite{ArRoSa2}
on Carleson measures for the Dirichlet space $B_{2}\left(  \mathbb{B}%
_{n}\right)  $ on the unit ball $\mathbb{B}_{n}$, as well as for certain
$B_{2}\left(  \mathcal{T}\right)  $ spaces on trees $\mathcal{T}$, including
the Bergman trees $\mathcal{T}_{n}$. By a tree we mean a connected loopless
graph $\mathcal{T}$ with a root $o$ and a partial order $\leq$ defined by
$\alpha\leq\beta$ if $\alpha$ belongs to the geodesic $\left[  o,\beta\right]
$. See for example \cite{ArRoSa} for more details. We define $B_{2}\left(
\mathbb{B}_{1}\right)  $ on the unit disc and $B_{2}\left(  \mathcal{T}%
\right)  $ on a tree respectively by the norms
\[
\left\Vert f\right\Vert _{B_{2}\left(  \mathbb{B}_{1}\right)  }=\left(
\int_{\mathbb{B}_{1}}\left\vert \left(  1-\left\vert z\right\vert ^{2}\right)
f^{\prime}\left(  z\right)  \right\vert ^{2}\frac{dz}{\left(  1-\left\vert
z\right\vert ^{2}\right)  ^{2}}+\left\vert f\left(  0\right)  \right\vert
^{2}\right)  ^{\frac{1}{2}},
\]
for $f$ holomorphic on $\mathbb{B}_{1}$, and
\[
\left\Vert f\right\Vert _{B_{2}\left(  \mathcal{T}\right)  }=\left(
\sum_{\alpha\in\mathcal{T}:\alpha\neq o}\left\vert f\left(  \alpha\right)
-f\left(  A\alpha\right)  \right\vert ^{2}+\left\vert f\left(  o\right)
\right\vert ^{2}\right)  ^{\frac{1}{2}},
\]
for $f$ on the tree $\mathcal{T}$. Here $A\alpha$ denotes the immediate
predecessor of $\alpha$ in the tree $\mathcal{T}$. We define the weighted
Lebesgue space $L_{\mu}^{2}\left(  \mathcal{T}\right)  $ on the tree by the
norm
\[
\left\Vert f\right\Vert _{L_{\mu}^{2}\left(  \mathcal{T}\right)  }=\left(
\sum_{\alpha\in\mathcal{T}}\left\vert f\left(  \alpha\right)  \right\vert
^{2}\mu\left(  \alpha\right)  \right)  ^{\frac{1}{2}},
\]
for $f$ and $\mu$ on the tree $\mathcal{T}$. We say that $\mu$ is a
$B_{2}\left(  \mathcal{T}\right)  $-Carleson measure on the tree $\mathcal{T}$
if $B_{2}\left(  \mathcal{T}\right)  $ embeds continuously into $L_{\mu}%
^{2}\left(  \mathcal{T}\right)  $, i.e.
\begin{equation}
\left(  \sum_{\alpha\in\mathcal{T}}If\left(  \alpha\right)  ^{2}\mu\left(
\alpha\right)  \right)  ^{1/2}\leq C\left(  \sum_{\alpha\in\mathcal{T}%
}f\left(  \alpha\right)  ^{2}\right)  ^{1/2},\;\;\;\;\;f\geq0,
\label{Carleson}%
\end{equation}
or equivalently, by duality,
\begin{equation}
\left(  \sum_{\alpha\in\mathcal{T}}I^{\ast}\left(  g\mu\right)  \left(
\alpha\right)  ^{2}\right)  ^{1/2}\leq C\left(  \sum_{\alpha\in\mathcal{T}%
}g\left(  \alpha\right)  ^{2}\mu\left(  \alpha\right)  \right)  ^{1/2}%
,\;\;\;\;\;g\geq0, \label{Carleson'}%
\end{equation}
where
\begin{align*}
If\left(  \alpha\right)   &  =\sum_{\beta\in\mathcal{T}:\beta\leq\alpha
}f\left(  \beta\right)  ,\\
I^{\ast}\left(  g\mu\right)  \left(  \alpha\right)   &  =\sum_{\beta
\in\mathcal{T}:\beta\geq\alpha}g\left(  \beta\right)  \mu\left(  \beta\right)
.
\end{align*}
If (\ref{Carleson}) is satisfied, we say that $\mu$ is a $B_{2}\left(
\mathcal{T}\right)  $-Carleson measure on the tree $\mathcal{T}$. A necessary
and sufficient condition for (\ref{Carleson}) given in \cite{ArRoSa} is the
discrete tree condition
\begin{equation}
\sum_{\beta\in\mathcal{T}:\beta\geq\alpha}I^{\ast}\mu\left(  \beta\right)
^{2}\leq C^{2}I^{\ast}\mu\left(  \alpha\right)  <\infty,\;\;\;\;\;\alpha
\in\mathcal{T}, \label{NASC}%
\end{equation}
which is obtained by testing (\ref{Carleson'}) over $g=\chi_{S_{\alpha}}$,
$\alpha\in\mathcal{T}$. We note that a simpler necessary condition for
(\ref{Carleson}) is
\begin{equation}
d\left(  \alpha\right)  I^{\ast}\mu\left(  \alpha\right)  \leq C^{2},
\label{simple}%
\end{equation}
which one obtains by testing (\ref{Carleson}) over $f=I^{\ast}\delta_{\alpha
}=\chi_{\left[  0,\alpha\right]  }$. However, condition (\ref{simple}) is not
in general sufficient for (\ref{Carleson}) as evidenced by certain Cantor-like
measures $\mu$.

We also have the more general two-weight tree theorem from \cite{ArRoSa}.

\begin{theorem}
\label{Tars}Let $w$ and $v$ be nonnegative weights on a tree $\mathcal{T}$.
Then,
\begin{equation}
\left(  \sum_{\alpha\in\mathcal{T}}If\left(  \alpha\right)  ^{2}w\left(
\alpha\right)  \right)  ^{1/2}\leq C\left(  \sum_{\alpha\in\mathcal{T}%
}f\left(  \alpha\right)  ^{2}v\left(  \alpha\right)  \right)  ^{1/2}%
,\;\;\;\;\;f\geq0, \label{star}%
\end{equation}
if and only if
\[
\sum_{\beta\geq\alpha}I^{*}w\left(  \beta\right)  ^{2}v\left(  \beta\right)
^{-1}\leq CI^{*}w\left(  \alpha\right)  <\infty,\;\;\;\;\;\alpha\in
\mathcal{T}.
\]

\end{theorem}

We now specialize the tree $\mathcal{T}$ to the Bergman tree $\mathcal{T}_{n}
$ associated with the usual decomposition of the unit ball $\mathbb{B}_{n}$
into top halves of Carleson \textquotedblleft boxes\textquotedblright\ or
Bergman \textquotedblleft kubes\textquotedblright\ $K_{\alpha}$. See
Subsection 2.2 in \cite{ArRoSa2} and Subsection 2.4 in \cite{Zhu} for details.
The following characterization of $B_{2}\left(  \mathbb{B}_{n}\right)
$-Carleson measures on the unit ball $\mathbb{B}_{n}$ is from \cite{ArRoSa2}.
Given a positive measure $\mu$ on the ball, we denote by $\widehat{\mu}$ the
associated measure on the Bergman tree $\mathcal{T}_{n}$ given by
$\widehat{\mu}\left(  \alpha\right)  =\int_{K_{\alpha}}d\mu$ for $\alpha
\in\mathcal{T}_{n}$. We say that $\mu$ is a $B_{2}\left(  \mathbb{B}%
_{n}\right)  $-Carleson measure on the unit ball $\mathbb{B}_{n}$ if
\begin{equation}
\left(  \int_{\mathbb{B}_{n}}\left\vert f\left(  z\right)  \right\vert
^{2}d\mu\left(  z\right)  \right)  ^{\frac{1}{2}}\leq C\left\Vert f\right\Vert
_{B_{2}},\;\;\;\;\;f\in B_{2}, \label{C'''}%
\end{equation}
and that $\widehat{\mu}$ is a $B_{2}\left(  \mathcal{T}_{n}\right)  $-Carleson
measure on the Bergman tree $\mathcal{T}_{n}$ if
\begin{equation}
\left(  \sum_{\alpha\in\mathcal{T}_{n}}If\left(  \alpha\right)  ^{2}%
\widehat{\mu}\left(  \alpha\right)  \right)  ^{1/2}\leq C\left(  \sum
_{\alpha\in\mathcal{T}_{n}}f\left(  \alpha\right)  ^{2}\right)  ^{1/2}%
,\;\;\;\;\;f\geq0. \label{T'''}%
\end{equation}

\begin{theorem}
\label{Besov}Suppose $\mu$ is a positive measure on the unit ball
$\mathbb{B}_{n}$. Then with constants depending only on dimension $n$, the
following conditions are equivalent:

\begin{enumerate}
\item $\mu$ is a $B_{2}\left(  \mathbb{B}_{n}\right)  $-Carleson measure on
$\mathbb{B}_{n}$, i.e. (\ref{C'''}) holds.

\item $\widehat{\mu}=\left\{  \widehat{\mu}\left(  \alpha\right)  \right\}
_{\alpha\in\mathcal{T}_{n}}$ is a $B_{2}\left(  \mathcal{T}_{n}\right)
$-Carleson measure on the Bergman tree $\mathcal{T}_{n}$, i.e. (\ref{T'''}) holds.

\item There is $C<\infty$ such that
\[
\sum_{\beta\geq\alpha}I\widehat{\mu}\left(  \beta\right)  ^{2}\leq CI^{\ast
}\widehat{\mu}\left(  \alpha\right)  <\infty,\;\;\;\;\;\alpha\in
\mathcal{T}_{n}.
\]

\end{enumerate}
\end{theorem}

\subsection{Unified proofs for trees}

We begin with some notation. Let $\mathcal{G}_{\mathcal{T}}$ be the set of
maximal geodesics of $\mathcal{T}$ starting at the root. For $\alpha
\in\mathcal{T}$ let $\mathcal{S}\left(  \alpha\right)  \subset\mathcal{G}%
_{\mathcal{T}}$ denote the collection of all geodesics passing through
$\alpha$ (i.e. that are eventually in the successor set $S\left(
\alpha\right)  ).$ To unify considerations involving both the tree
$\mathcal{T}$ and its ideal boundary $\mathcal{G}_{\mathcal{T}}$ we set
$\mathcal{T}^{\ast}=\mathcal{T}\cup\mathcal{G}_{\mathcal{T}}$ and let
$\mathcal{S}^{\ast}\left(  \alpha\right)  =S\left(  \alpha\right)
\cup\mathcal{S}\left(  \alpha\right)  $ be the union of the successor set
$S\left(  \alpha\right)  $ with its boundary geodesics. We suppose $\mu,$
$\sigma,$ $\omega,$ and $\nu$ are finite positive measures on $\mathcal{T}%
^{\ast}$ with, for the moment, $\mu$, $\omega$ and $\sigma$ supported in the
tree $\mathcal{T}$, and $\nu$ supported in the boundary $\mathcal{G}%
_{\mathcal{T}}.$

We now give a short proof that the two weight tree condition,
\begin{equation}
\sum_{\beta\in\mathcal{T}:\beta\geq\alpha}I^{\ast}\mu\left(  \beta\right)
^{2}\omega\left(  \beta\right)  \leq C_{0}^{2}I^{\ast}\mu\left(
\alpha\right)  <\infty,\;\;\;\;\;\alpha\in\mathcal{T}, \label{NASC2}%
\end{equation}
implies the dual Besov-Carleson embedding (which is equivalent to (\ref{star})
with $\mu=w$ and $\omega=1/v).$
\begin{equation}
\sum_{\alpha\in\mathcal{T}}I^{\ast}\left(  g\mu\right)  \left(  \alpha\right)
^{2}\omega\left(  \alpha\right)  \leq C^{2}\sum_{\alpha\in\mathcal{T}}g\left(
\alpha\right)  ^{2}\mu\left(  \alpha\right)  ,\;\;\;\;\;g\geq0,
\label{Carleson'2}%
\end{equation}
Moreover, we will unify this result and the well-known equivalence of the
Hardy-Carleson embedding on the tree,
\begin{equation}
\sum_{\alpha\in\mathcal{T}}\left(  \frac{1}{\left\vert \mathcal{S}\left(
\alpha\right)  \right\vert _{\nu}}\int_{\mathcal{S}\left(  \alpha\right)
}fd\nu\right)  ^{2}\sigma\left(  \alpha\right)  \leq C^{2}\int_{\mathcal{G}%
_{\mathcal{T}}}f^{2}d\nu,\;\;\;\;\;f\geq0\text{ on }\mathcal{G}_{\mathcal{T}},
\label{Cargeo}%
\end{equation}
with the simple condition on geodesics,
\begin{equation}
\sum_{\beta\geq\alpha}\sigma\left(  \beta\right)  \leq C_{0}^{2}\left\vert
\mathcal{S}\left(  \alpha\right)  \right\vert _{\nu},\;\;\;\;\;\alpha
\in\mathcal{T}. \label{Simgeo}%
\end{equation}
We rewrite (\ref{Carleson'2}) as
\begin{equation}
\int_{\mathcal{T}}\left(  \frac{1}{\left\vert \mathcal{S}^{\ast}\left(
\alpha\right)  \right\vert _{\mu}}\int_{\mathcal{S}^{\ast}\left(
\alpha\right)  }gd\mu\right)  ^{2}\left\vert \mathcal{S}^{\ast}\left(
\cdot\right)  \right\vert _{\mu}^{2}d\omega\left(  \cdot\right)  \leq
C^{2}\int_{\mathcal{T}^{\ast}}g^{2}d\mu,\;\;\;\;\;g\geq0\text{ on }%
\mathcal{T}^{\ast}, \label{SstarBesov}%
\end{equation}
and rewrite (\ref{Cargeo}) as
\begin{equation}
\int_{\mathcal{T}}\left(  \frac{1}{\left\vert \mathcal{S}^{\ast}\left(
\alpha\right)  \right\vert _{\nu}}\int_{\mathcal{S}^{\ast}\left(
\alpha\right)  }fd\nu\right)  ^{2}d\sigma\left(  \alpha\right)  \leq C^{2}%
\int_{\mathcal{T}^{\ast}}f^{2}d\nu,\;\;\;\;\;f\geq0\text{ on }\mathcal{T}%
^{\ast}. \label{SstarHardy}%
\end{equation}
Thus we see that the inequality (\ref{SstarBesov}) has exactly the same form
as inequality (\ref{SstarHardy}), but with $\left\vert \mathcal{S}^{\ast
}\left(  \cdot\right)  \right\vert _{\mu}^{2}d\omega\left(  \cdot\right)  $ in
place of $d\sigma$ and $\mu$ in place of $\nu$. Note that the integrations on
the left are over $\mathcal{T}$, where the averages on $\mathcal{S}^{\ast
}\left(  \alpha\right)  $ are defined. Moreover, the tree condition
(\ref{NASC2}) is just the simple condition (\ref{Simgeo}) for the measures
$\left\vert \mathcal{S}^{\ast}\left(  \cdot\right)  \right\vert _{\mu}%
^{2}d\omega\left(  \cdot\right)  $ and $\mu$:
\[
\sum_{\beta\geq\alpha}\left\vert \mathcal{S}^{\ast}\left(  \beta\right)
\right\vert _{\mu}^{2}\omega\left(  \beta\right)  \leq C_{0}^{2}\left\vert
\mathcal{S}^{\ast}\left(  \alpha\right)  \right\vert _{\mu},\;\;\;\;\;\alpha
\in\mathcal{T}.
\]
In fact, if one permits $\nu$ in (\ref{SstarHardy}) to live in all of the
closure $\mathcal{T}^{\ast}$, then we can characterize (\ref{SstarHardy}) by a
simple condition, and if one permits $\sigma$ to live in all of $\mathcal{T}%
^{\ast}$ as well, then the corresponding maximal inequality is characterized
by a simple condition. The following theorem will be used later to
characterize Carleson measures for the Drury-Arveson space $B_{2}^{1/2}\left(
\mathbb{B}_{n}\right)  $. The proof can be used to simplify some of the
arguments in \cite{ArRoSa} and \cite{ArRoSa2}.

\begin{theorem}
\label{unity}Inequality (\ref{SstarHardy}) holds if and only if
\begin{equation}
\left\vert S\left(  \alpha\right)  \right\vert _{\sigma}\leq C_{0}%
^{2}\left\vert \mathcal{S}^{\ast}\left(  \alpha\right)  \right\vert _{\nu
},\;\;\;\;\;\alpha\in\mathcal{T}. \label{gensimp}%
\end{equation}
More generally, if both $\sigma$ and $\nu$ live in $\mathcal{T}^{\ast}$, then
the maximal inequality
\begin{equation}
\int_{\mathcal{T}^{\ast}}\mathcal{M}f\left(  \zeta\right)  ^{2}d\sigma\left(
\zeta\right)  \leq C^{2}\int_{\mathcal{T}^{\ast}}\left\vert f\right\vert
^{2}d\nu,\;\;\;\;\;\text{for all }f\text{ on }\mathcal{T}^{\ast},
\label{maxineq}%
\end{equation}
where
\[
\mathcal{M}f\left(  \zeta\right)  =\mathcal{M}\left(  fd\nu\right)  \left(
\zeta\right)  =\sup_{\alpha\in\mathcal{T}:\alpha\leq\zeta}\frac{1}{\left\vert
\mathcal{S}^{\ast}\left(  \alpha\right)  \right\vert _{\nu}}\int
_{\mathcal{S}^{\ast}\left(  \alpha\right)  }\left\vert f\right\vert d\nu,
\]
holds if and only if
\begin{equation}
\left\vert \mathcal{S}^{\ast}\left(  \alpha\right)  \right\vert _{\sigma}\leq
C_{0}^{2}\left\vert \mathcal{S}^{\ast}\left(  \alpha\right)  \right\vert
_{\nu},\;\;\;\;\;\alpha\in\mathcal{T}. \label{gensimp'}%
\end{equation}

\end{theorem}%

\proof
The necessity of (\ref{gensimp}) for (\ref{SstarHardy}), and also
(\ref{gensimp'}) for (\ref{maxineq}), follows upon setting $f=\chi
_{\mathcal{S}^{*}\left(  \alpha\right)  }$ in the respective inequality. To
see that (\ref{gensimp'}) is sufficient for (\ref{maxineq}), which includes
the assertion that (\ref{gensimp}) is sufficient for (\ref{SstarHardy}), note
that the sublinear map $\mathcal{M}$ is bounded with norm $1$ from $L^{\infty
}\left(  \mathcal{T}^{*};\nu\right)  $ to $L^{\infty}\left(  \mathcal{T}%
^{*};\sigma\right)  $, and is weak type $1-1$ with constant $C_{0}$ by
(\ref{gensimp'}). Indeed,
\[
\left\{  \zeta\in\mathcal{T}^{*}:\mathcal{M}f\left(  \zeta\right)
>\lambda\right\}  \subset\cup\left\{  \mathcal{S}^{*}\left(  \alpha\right)
:\alpha\in\mathcal{T}\text{ and }\mathcal{M}f\left(  \alpha\right)
>\lambda\right\}  ,
\]
and if we let $\lambda>0$ and denote by $\Gamma$ the minimal elements in
$\left\{  \alpha\in\mathcal{T}:\mathcal{M}f\left(  \alpha\right)
>\lambda\right\}  $, then
\begin{align*}
\left|  \left\{  \zeta\in\mathcal{T}^{*}:\mathcal{M}f\left(  \zeta\right)
>\lambda\right\}  \right|  _{\sigma}  &  \leq\sum_{\alpha\in\Gamma}\left|
\mathcal{S}^{*}\left(  \alpha\right)  \right|  _{\sigma}\leq C_{0}^{2}%
\sum_{\alpha\in\Gamma}\left|  \mathcal{S}^{*}\left(  \alpha\right)  \right|
_{\nu}\\
&  \leq C_{0}^{2}\sum_{\alpha\in\Gamma}\lambda^{-1}\int_{\mathcal{S}%
^{*}\left(  \alpha\right)  }\left|  f\right|  d\nu\leq C_{0}^{p}\lambda
^{-1}\int_{\mathcal{T}^{*}}\left|  f\right|  d\nu.
\end{align*}
Marcinkiewicz interpolation now completes the proof.

The proof actually yields the following more general inequality.

\begin{theorem}%
\[
\int_{\mathcal{T}^{*}}\mathcal{M}f\left(  \zeta\right)  ^{2}d\sigma\left(
\zeta\right)  \leq C^{2}\int_{\mathcal{T}^{*}}\left|  f\right|  ^{2}%
\mathcal{M}\left(  d\sigma\right)  d\nu,\;\;\;\;\;\text{for all }f\text{ on
}\mathcal{T}^{*}.
\]

\end{theorem}%

\proof
Following the above proof we use instead the estimate
\begin{align*}
\left\vert \left\{  \zeta\in\mathcal{T}^{\ast}:\mathcal{M}f\left(
\zeta\right)  >\lambda\right\}  \right\vert _{\sigma}  &  \leq\sum_{\alpha
\in\Gamma}\left\vert \mathcal{S}^{\ast}\left(  \alpha\right)  \right\vert
_{\sigma}=\sum_{\alpha\in\Gamma}\frac{\left\vert \mathcal{S}^{\ast}\left(
\alpha\right)  \right\vert _{\sigma}}{\left\vert \mathcal{S}^{\ast}\left(
\alpha\right)  \right\vert _{\nu}}\left\vert \mathcal{S}^{\ast}\left(
\alpha\right)  \right\vert _{\nu}\\
&  \leq\sum_{\alpha\in\Gamma}\frac{\left\vert \mathcal{S}^{\ast}\left(
\alpha\right)  \right\vert _{\sigma}}{\left\vert \mathcal{S}^{\ast}\left(
\alpha\right)  \right\vert _{\nu}}\lambda^{-1}\int_{\mathcal{S}^{\ast}\left(
\alpha\right)  }\left\vert f\right\vert d\nu\\
&  \leq\sum_{\alpha\in\Gamma}\lambda^{-1}\int_{\mathcal{S}^{\ast}\left(
\alpha\right)  }\left\vert f\left(  \zeta\right)  \right\vert \mathcal{M}%
\left(  d\sigma\right)  \left(  \zeta\right)  d\nu\left(  \zeta\right)  ,
\end{align*}
which shows that $\mathcal{M}$ is weak type $1-1$ with respect to the measures
$\sigma$ and $\mathcal{M}\left(  d\sigma\right)  d\nu.$

\section{Carleson measures for the Hardy-Sobolev spaces}

\subsection{The case $\sigma\geq0$}

Given a positive measure $\mu$ on the ball, we denote by $\widehat{\mu}$ the
associated measure on the Bergman tree $\mathcal{T}_{n}$ given by
$\widehat{\mu}\left(  \alpha\right)  =\int_{K_{\alpha}}d\mu$ for $\alpha
\in\mathcal{T}_{n}$. We will often write $\mu\left(  \alpha\right)  $ for
$\widehat{\mu}\left(  \alpha\right)  $ when no confusion should arise. Let
$\sigma\geq0$. Recall that $\mu$ is a $B_{2}^{\sigma}$-Carleson measure on
$\mathbb{B}_{n}$ if there is a positive constant $C$ such that
\begin{equation}
\left(  \int_{\mathbb{B}_{n}}\left\vert f\left(  z\right)  \right\vert
^{2}d\mu\left(  z\right)  \right)  ^{\frac{1}{2}}\leq C\left\Vert f\right\Vert
_{B_{2}^{\sigma}}, \label{Carleson''}%
\end{equation}
for all $f\in B_{2}^{\sigma}$. In this section we show (Theorem \ref{Besov'})
that $\mu$ is a $B_{2}^{\sigma}$-Carleson measure on $\mathbb{B}_{n}$ if
$\widehat{\mu}$ is a $B_{2}^{\sigma}(\mathcal{T}_{n})$-Carleson measure, i.e.
if it satisfies
\begin{equation}
\left(  \sum_{\alpha\in\mathcal{T}_{n}}If\left(  \alpha\right)  ^{2}\mu\left(
\alpha\right)  \right)  ^{1/2}\leq C\left(  \sum_{\alpha\in\mathcal{T}_{n}%
}\left[  2^{-\sigma d\left(  \alpha\right)  }f\left(  \alpha\right)  \right]
^{2}\right)  ^{1/2},\;\;\;\;\;f\geq0, \label{treecond'}%
\end{equation}
which is (\ref{star}) with $w\left(  \alpha\right)  =\mu\left(  \alpha\right)
$ and $v\left(  \alpha\right)  =2^{-2\sigma d\left(  \alpha\right)  }$. The
dual of (\ref{treecond'}) is
\begin{equation}
\left(  \sum_{\alpha\in\mathcal{T}_{n}}\left[  2^{\sigma d\left(
\alpha\right)  }I^{\ast}g\mu\left(  \alpha\right)  \right]  ^{2}\right)
^{1/2}\leq C\left(  \sum_{\alpha\in\mathcal{T}_{n}}g\left(  \alpha\right)
^{2}\mu\left(  \alpha\right)  \right)  ^{1/2},\;\;\;\;\;g\geq0.
\label{treecond'dual}%
\end{equation}
Theorem \ref{Tars} shows that (\ref{treecond'}) is equivalent to the tree
condition
\begin{equation}
\sum_{\beta\geq\alpha}\left[  2^{\sigma d\left(  \beta\right)  }I^{\ast}%
\mu\left(  \beta\right)  \right]  ^{2}\leq CI^{\ast}\mu\left(  \alpha\right)
<\infty,\;\;\;\;\;\alpha\in\mathcal{T}_{n}. \label{treecond}%
\end{equation}
Conversely, in the range $0\leq\sigma<1/2$, we show that $\mu$ is
$B_{2}^{\sigma}(\mathcal{T}_{n})$-Carleson if $\mu$ is a $B_{2}^{\sigma
}\left(  \mathbb{B}_{n}\right)  $-Carleson measure on $\mathbb{B}_{n}$.

\begin{theorem}
\label{Besov'}Suppose $\sigma\geq0$ and that the structural constants
$\lambda,\theta$ in the construction of $\mathcal{T}_{n}$ (subsection 2.2. of
\cite{ArRoSa2}) satisfy $\lambda=1$ and $\theta=\frac{\ln2}{2}$. Let $\mu$ be
a positive measure on the unit ball $\mathbb{B}_{n}$. Then with constants
depending only on $\sigma$ and $n$, conditions \emph{2} and \emph{3} below are
equivalent, condition \emph{3} is sufficient for condition \emph{1}, and
provided that $0\leq\sigma<1/2$, condition \emph{3} is necessary for condition
\emph{1}:

\begin{enumerate}
\item $\mu$ is a $B_{2}^{\sigma}\left(  \mathbb{B}_{n}\right)  $-Carleson
measure on $\mathbb{B}_{n}$, i.e. (\ref{Carleson''}) holds.

\item $\widehat{\mu}=\left\{  \mu\left(  \alpha\right)  \right\}  _{\alpha
\in\mathcal{T}_{n}}$ is a $B_{2}^{\sigma}\left(  \mathcal{T}_{n}\right)
$-Carleson measure, i.e. (\ref{treecond'}) holds with $\mu\left(
\alpha\right)  =\int_{K_{\alpha}}d\mu,$ where $\mathcal{T}_{n}$\ ranges over
all unitary rotations of a fixed Bergman tree.

\item There is $C<\infty$ such that
\[
\sum_{\beta\geq\alpha}\left[  2^{\sigma d\left(  \beta\right)  }I^{\ast}%
\mu\left(  \beta\right)  \right]  ^{2}\leq CI^{\ast}\mu\left(  \alpha\right)
<\infty,\;\;\;\;\;\alpha\in\mathcal{T}_{n},
\]
where $\mathcal{T}_{n}$\ ranges over all unitary rotations of a fixed Bergman tree.
\end{enumerate}
\end{theorem}

%

\proof
The case $\sigma=0$ is Theorem \ref{Besov} above, and was proved in
\cite{ArRoSa2}. Theorem \ref{Tars} yields the equivalence of conditions $2$
and $3 $ in Theorem \ref{Besov'}.

We use the presentation of $B_{2}^{\sigma}\left(  \mathbb{B}_{n}\right)  $
given by $\mathcal{H}_{k}$ with kernel function $k\left(  w,z\right)  =\left(
\frac{1}{1-\overline{w}\cdot z}\right)  ^{2\sigma}$ on $\mathbb{B}_{n}$ as
given in Subsubsection \ref{IS}. To begin we must verify that this kernel is
positive definite, i.e.,
\[
\sum_{i,j=1}^{m}a_{i}\overline{a_{j}}k\left(  z_{i},z_{j}\right)  \geq0\text{
with equality }\Leftrightarrow\text{ all }a_{i}=0.
\]
Now for $0<\sigma\leq1/2$, this follows by expanding $\left(  1-\overline
{w}\cdot z\right)  ^{-2\sigma}$ in a power series, using that the coefficients
in the expansion are positive, and that the matrices $\left[  \left(
\overline{z_{i}}\cdot z_{j}\right)  ^{\ell}\right]  _{i,j=1}^{N}$ are
nonnegative semidefinite by Schur's theorem for $\ell,N\geq1$. There is
however another approach that not only works for all $\sigma>0$, but also
yields the equivalence of the norms in $\mathcal{H}_{k}$ and $B_{2}^{\sigma
}\left(  \mathbb{B}_{n}\right)  $. For this we recall the invertible
\textquotedblleft radial\textquotedblright\ differentiation operators
$R^{\gamma,t}:H\left(  \mathbb{B}_{n}\right)  \rightarrow H\left(
\mathbb{B}_{n}\right)  $ given in \cite{Zhu} by
\[
R^{\gamma,t}f\left(  z\right)  =\sum_{k=0}^{\infty}\frac{\Gamma\left(
n+1+\gamma\right)  \Gamma\left(  n+1+k+\gamma+t\right)  }{\Gamma\left(
n+1+\gamma+t\right)  \Gamma\left(  n+1+k+\gamma\right)  }f_{k}\left(
z\right)  ,
\]
provided neither $n+\gamma$ nor $n+\gamma+t$ is a negative integer, and where
$f\left(  z\right)  =\sum_{k=0}^{\infty}f_{k}\left(  z\right)  $ is the
homogeneous expansion of $f$. If the inverse of $R^{\gamma,t}$ is denoted
$R_{\gamma,t}$, then Proposition 1.14 of \cite{Zhu} yields
\begin{align}
R^{\gamma,t}\left(  \frac{1}{\left(  1-\overline{w}\cdot z\right)
^{n+1+\gamma}}\right)   &  =\frac{1}{\left(  1-\overline{w}\cdot z\right)
^{n+1+\gamma+t}},\label{Zhuidentity}\\
R_{\gamma,t}\left(  \frac{1}{\left(  1-\overline{w}\cdot z\right)
^{n+1+\gamma+t}}\right)   &  =\frac{1}{\left(  1-\overline{w}\cdot z\right)
^{n+1+\gamma}},\nonumber
\end{align}
for all $w\in\mathbb{B}_{n}$. Thus for any $\gamma$, $R^{\gamma,t}$ is
approximately differentiation of order $t$. From Theorem 6.1 and Theorem 6.4
of \cite{Zhu} we have that the derivatives $R^{\gamma,m}f\left(  z\right)  $
are \textquotedblleft$L^{2}$ norm equivalent\textquotedblright\ to $\sum
_{k=0}^{m-1}\left\vert f^{\left(  k\right)  }\left(  0\right)  \right\vert
+f^{\left(  m\right)  }\left(  z\right)  $ for $m$ large enough and $f\in
H\left(  \mathbb{B}_{n}\right)  $. We will also use that the proof of
Corollary 6.5 of \cite{Zhu} shows that $R^{\gamma,\frac{n+1+\alpha}{2}-\sigma
}$ is a bounded invertible operator from $B_{2}^{\sigma}$ onto the weighted
Bergman space $A_{\alpha}^{2}$, provided that neither $n+\gamma$ nor
$n+\gamma+\frac{n+1+\alpha}{2}-\sigma$ is a negative integer.

Let $\ell_{z}^{\eta}\left(  \zeta\right)  =\left(  \frac{1}{1-\overline
{z}\cdot\zeta}\right)  ^{\eta}$ and set $d\nu_{\alpha}\left(  \zeta\right)
=(1-\left\vert z\right\vert ^{2})^{\alpha}d\lambda(\zeta).$ Note from
(\ref{Zhuidentity}) that
\[
R^{\gamma,t}\ell_{z}^{\eta}\left(  \zeta\right)  =\frac{1}{\left(
1-\overline{z}\cdot\zeta\right)  ^{n+1+\alpha}}%
\]
provided $t=n+1+\alpha-\eta$ and $\gamma=\eta-n-1.$The reproducing formula in
Theorem 2.7 of \cite{Zhu} yields
\begin{align*}
\ell_{w}^{\eta}\left(  z\right)   &  =\int_{\mathbb{B}_{n}}\ell_{w}^{\eta
}\left(  \zeta\right)  \overline{\left(  \frac{1}{1-\overline{z}\cdot\zeta
}\right)  }^{n+1+\alpha}d\nu_{\alpha}\left(  \zeta\right) \\
&  =\int_{\mathbb{B}_{n}}\ell_{w}^{\eta}\left(  \zeta\right)  \overline
{R^{\gamma,t}\ell_{z}^{\eta}\left(  \zeta\right)  }d\nu_{\alpha}\left(
\zeta\right)  .
\end{align*}
Now let $S^{\gamma,t}$ be the square root of $R^{\gamma,t}$ defined by
\[
S^{\gamma,t}f\left(  z\right)  =\sum_{k=0}^{\infty}\sqrt{\frac{\Gamma\left(
n+1+\gamma\right)  \Gamma\left(  n+1+k+\gamma+t\right)  }{\Gamma\left(
n+1+\gamma+t\right)  \Gamma\left(  n+1+k+\gamma\right)  }}f_{k}\left(
z\right)  .
\]
Since $R^{\gamma,t}=\left(  S^{\gamma,t}\right)  ^{\ast}S^{\gamma,t}$ we have
with $\eta=2\sigma$ that
\begin{align*}
\sum_{i,j=1}^{m}a_{i}\overline{a_{j}}k\left(  x_{i},x_{j}\right)   &
=\sum_{i,j=1}^{m}a_{i}\overline{a_{j}}\int_{\mathbb{B}_{n}}S^{\gamma,t}%
\ell_{z_{i}}^{\eta}\left(  \zeta\right)  \overline{S^{\gamma,t}\ell_{z_{j}%
}^{\eta}\left(  \zeta\right)  }d\nu_{\alpha}\left(  \zeta\right) \\
&  =\int_{\mathbb{B}_{n}}\left\vert \sum_{i=1}^{m}a_{i}S^{\gamma,t}\ell
_{z_{i}}^{\eta}\left(  \zeta\right)  \right\vert ^{2}d\nu_{\alpha}\left(
\zeta\right)
\end{align*}
is positive definite. Note that $S^{\gamma,t}$ is a radial differentiation
operator of order $\frac{t}{2}$ so that $S^{\gamma,t}R_{\gamma,\frac{t}{2}} $
is bounded and invertible on the weighted Bergman space $A_{\alpha}^{2}$ (e.g.
by inspecting coefficients in homogeneous expansions). Thus with $\eta
=2\sigma$, we also have the equivalence of norms:
\[
\left\Vert f\right\Vert _{\mathcal{H}_{k}}^{2}=\int_{\mathbb{B}_{n}}\left\vert
S^{\gamma,t}f\left(  \zeta\right)  \right\vert ^{2}d\nu_{\alpha}\left(
\zeta\right)  \approx\int_{\mathbb{B}_{n}}\left\vert R^{\gamma,\frac
{n+1+\alpha-2\sigma}{2}}f\left(  \zeta\right)  \right\vert ^{2}d\nu_{\alpha
}\left(  \zeta\right)  \approx\left\Vert f\right\Vert _{B_{2}^{\sigma}\left(
\mathbb{B}_{n}\right)  }^{2}.
\]
For the remainder of this proof we will use the $\mathcal{H}_{k}$ norm on the
space $B_{2}^{\sigma}\left(  \mathbb{B}_{n}\right)  $.

The next part of the argument holds for general Hilbert spaces with
reproducing kernel hence we isolate it as a separate lemma. Let $\mathcal{J}$
be a Hilbert space of functions on $X$ with reproducing kernel functions
\{$j_{x}(\cdot)\}_{x\in X}.$ A measure $\mu$ on $X$ is a $\mathcal{J}%
$-Carleson measure exactly if the inclusion map $T$ is bounded from
$\mathcal{J}$ to $L^{2}\left(  X,\mu\right)  $.

\begin{lemma}
\label{carllem}A measure $\mu$ is a $\mathcal{J}$-Carleson measure if and only
if the linear map
\[
f\left(  \cdot\right)  \rightarrow Sf\left(  \cdot\right)  =\int
_{X}\operatorname{Re}j_{x}\left(  \cdot\right)  \,f(x)d\mu(x)
\]
is bounded on $L^{2}\left(  X,\mu\right)  .$
\end{lemma}

%

\proof
$T$ is bounded if and only if the adjoint $T^{\ast}$ is bounded from
$L^{2}\left(  X,\mu\right)  $ to $\mathcal{J}$, i.e.
\begin{equation}
\left\Vert T^{\ast}f\right\Vert _{\mathcal{J}}^{2}=\left\langle T^{\ast
}f,T^{\ast}f\right\rangle _{\mathcal{J}}\leq C\left\Vert f\right\Vert
_{L^{2}\left(  \mu\right)  }^{2},\;\;\;\;\;f\in L^{2}\left(  \mu\right)  .
\label{TT*}%
\end{equation}
For $x\in X$ we have
\begin{align*}
T^{\ast}f\left(  x\right)   &  =\left\langle T^{\ast}f,j_{x}\right\rangle
_{\mathcal{J}}=\left\langle f,Tj_{x}\right\rangle _{L^{2}\left(  \mu\right)
}\\
&  =\int f\left(  w\right)  \overline{j_{x}\left(  w\right)  }d\mu\left(
w\right) \\
&  =\int j_{w}\left(  x\right)  f\left(  w\right)  d\mu\left(  w\right)  ,
\end{align*}
and thus we obtain
\begin{align*}
\left\Vert T^{\ast}f\right\Vert _{\mathcal{J}}^{2}  &  =\left\langle T^{\ast
}f,T^{\ast}f\right\rangle _{\mathcal{J}}\\
&  =\left\langle \int j_{w}f\left(  w\right)  d\mu\left(  w\right)  ,\int
j_{w^{\prime}}f\left(  w^{\prime}\right)  d\mu\left(  w^{\prime}\right)
\right\rangle _{\mathcal{J}}\\
&  =\int\int\left\langle j_{w},j_{w^{\prime}}\right\rangle _{\mathcal{J}%
}f\left(  w\right)  d\mu\left(  w\right)  \overline{f\left(  w^{\prime
}\right)  }d\mu\left(  w^{\prime}\right) \\
&  =\int\int j_{w}\left(  w^{\prime}\right)  f\left(  w\right)  d\mu\left(
w\right)  \overline{f\left(  w^{\prime}\right)  }d\mu\left(  w^{\prime
}\right)  .
\end{align*}
Having (\ref{TT*}) for general $f$ is equivalent to having it for real $f$ and
we now suppose $f$ is real. In that case we continue with
\begin{align*}
\left\Vert T^{\ast}f\right\Vert _{\mathcal{J}}^{2}  &  =\int\int
\operatorname{Re}j_{w}\left(  w^{\prime}\right)  f\left(  w\right)  f\left(
w^{\prime}\right)  d\mu\left(  w\right)  d\mu\left(  w^{\prime}\right) \\
&  =\left\langle Sf,f\right\rangle _{L^{2}\left(  \mu\right)  }.
\end{align*}
The last quantity satisfies the required estimates exactly if $S$ is bounded;
the proof is complete.

The first of the two following corollaries is immediate from the lemma.

\begin{corollary}
\label{same}Suppose $\mathcal{J}$ and $\mathcal{J}^{\prime}$ are two
reproducing kernel Hilbert spaces on $X$ with kernel functions $\left\{
j\right\}  $ and $\left\{  j^{\prime}\right\}  $ respectively. If
$\operatorname{Re}j_{x}(y)\leq c\operatorname{Re}j_{x}^{\prime}(y)$ then every
$\mathcal{J}^{\prime}$-Carleson measure is a $\mathcal{J}$-Carleson measure.
If $\operatorname{Re}j_{x}(y)\sim c\operatorname{Re}j_{x}^{\prime}(y)$ then
the two sets of Carleson measures coincide.
\end{corollary}

\begin{corollary}
\label{same2}Suppose $X$ is a bounded open set in some $\mathbb{R}^{k}$ and
$\partial\bar{X}$ is smooth. Suppose $\mathcal{J}$ and $\mathcal{J}^{\prime} $
are two reproducing kernel Hilbert spaces on $X$ with kernel functions
$\left\{  j\right\}  $ and $\left\{  j^{\prime}\right\}  $ and that there is a
smooth function $h(x,y)$ on $\bar{X}\times\bar{X}$ which is bounded and
bounded away from zero so that
\[
j_{x}^{\prime}(y)=j_{x}(y)h(x,y).
\]
Then the set of $\mathcal{J}^{\prime}$-Carleson measures and $\mathcal{J}%
$-Carleson measures coincide.
\end{corollary}

%

\proof
In the proof of the lemma we saw that $\mu$ was a $\mathcal{J}$-Carleson
measure if and only if%

\[
f\left(  \cdot\right)  \rightarrow Rf\left(  \cdot\right)  =\int_{X}%
j_{x}\left(  \cdot\right)  \,f(x)d\mu(x)
\]
was a bounded operator on $L^{2}\left(  X,\mu\right)  ;$ and similarly for
$j^{\prime}.$ Thus we need to show that $R$ is bounded if and only if
$R^{\prime}$ is where $R^{\prime}$ is given by
\begin{align}
f\left(  \cdot\right)  \rightarrow R^{\prime}f\left(  \cdot\right)   &
=\int_{X}j_{x}^{\prime}\left(  \cdot\right)  \,f(x)d\mu(x)\nonumber\\
&  =\int_{X}j_{x}(\cdot)h(x,\cdot)f(x)d\mu(x). \label{R'}%
\end{align}
However this follows from standard facts about bounded operators given by
integral kernels. For instance we could extend $h$ to be a smooth compactly
supported function in a box in $\mathbb{R}^{2k}$ which contains $X\times X.$
Then expand $h\left(  x,y\right)  $ in a multiple Fourier series $\sum
_{\alpha=(\alpha_{1},\alpha_{2})}c_{\alpha}e^{-i\alpha_{1}\cdot x}%
e^{-i\alpha_{2}\cdot y}$ and substitute into (\ref{R'}). This yields the
operator equation
\[
R^{\prime}=\sum_{\alpha=(\alpha_{1},\alpha_{2})}c_{\alpha}M_{e(\alpha_{2}%
)}RM_{e(\alpha_{1})}%
\]
where the $c_{a}$ are Fourier coefficients and the $e$'s are unimodular
characters and the $M_{e}$'s are the corresponding multiplication operators
$M_{e(\alpha)}g\left(  z\right)  =e^{-i\alpha\cdot z}g\left(  z\right)  $. The
$M_{e}$'s are unitary and the smoothness of $h$ insures that $\left\{
c_{a}\right\}  $ is an absolutely convergent sequence.\ Hence if $R$ is
bounded so is $R^{\prime}$. Because $h$ is bounded away from zero we can work
with $1/h(x,y)$ to reverse the argument and complete the proof.\medskip

In the case of current interest the lemma gives that $\mu$ is a $B_{2}%
^{\sigma}\left(  \mathbb{B}_{n}\right)  $-Carleson measure exactly if we have
estimates for
\[
\left\langle T^{\ast}f,T^{\ast}f\right\rangle _{B_{2}^{\sigma}\left(
\mathbb{B}_{n}\right)  }=\int\int\operatorname{Re}\left(  \frac{1}%
{1-\overline{w}\cdot w^{\prime}}\right)  ^{2\sigma}f\left(  w\right)
d\mu\left(  w\right)  f\left(  w^{\prime}\right)  d\mu\left(  w^{\prime
}\right)
\]
for $f\geq0.$

\bigskip Now we use that%

\begin{equation}
\operatorname{Re}\left(  \frac{1}{1-\overline{w}\cdot w^{\prime}}\right)
^{2\sigma}\approx\left\vert \frac{1}{1-\overline{w}\cdot w^{\prime}%
}\right\vert ^{2\sigma} \label{realbelow}%
\end{equation}
for $0\leq\sigma<1/2$, to obtain that $\mu$ is $B_{2}^{\sigma}\left(
\mathbb{B}_{n}\right)  $-Carleson if and only if
\[
\int\int\left\vert \frac{1}{1-\overline{w}\cdot w^{\prime}}\right\vert
^{2\sigma}f\left(  w\right)  d\mu\left(  w\right)  f\left(  w^{\prime}\right)
d\mu\left(  w^{\prime}\right)  \leq C\left\Vert f\right\Vert _{L^{2}\left(
\mu\right)  }^{2},\;\;\;\;\;f\geq0.
\]

This inequality is easily discretized using that
\begin{equation}
c2^{d\left(  \alpha\wedge\alpha^{\prime}\right)  }\leq\left\vert \frac
{1}{1-\overline{w}\cdot w^{\prime}}\right\vert \leq C\int_{\mathcal{U}_{n}%
}2^{d\left(  \alpha\left(  Uw\right)  \wedge\alpha\left(  Uw^{\prime}\right)
\right)  }dU, \label{disckernel}%
\end{equation}
for $w\in K_{\alpha}$ and $w^{\prime}\in K_{\alpha^{\prime}}$ where
$\alpha\left(  Uw\right)  $ denotes the unique kube $K_{\alpha\left(
Uw\right)  } $ containing $Uw.$ The second inequality above is analogous to
similar inequalities in Euclidean space used to control an operator by
translations of its dyadic version, and the proof is similar (e.g. use
(\ref{lowbound}) below and integrate over $\mathcal{U}_{n}\mathcal{)}$. Using
this and decomposing the ball $\mathbb{B}_{n}$ as $\cup_{\alpha\in
\mathcal{T}_{n}}K_{\alpha}$, we obtain that $\mu$ is $B_{2}^{\sigma}\left(
\mathbb{B}_{n}\right)  $-Carleson if and only if
\[
\sum_{\alpha,\alpha^{\prime}\in\mathcal{T}_{n}}2^{2\sigma d\left(
\alpha\wedge\alpha^{\prime}\right)  }f\left(  \alpha\right)  \mu\left(
\alpha\right)  f\left(  \alpha^{\prime}\right)  \mu\left(  \alpha^{\prime
}\right)  \leq C\sum_{\alpha\in\mathcal{T}_{n}}f\left(  \alpha\right)  ^{2}%
\mu\left(  \alpha\right)  ,\;\;\;\;\;f\geq0,
\]
where $\mathcal{T}_{n}$ ranges over all unitary rotations of a fixed Bergman
tree. Now for $\sigma>0$,
\[
2^{2\sigma d\left(  \alpha\wedge\alpha^{\prime}\right)  }\approx\sum
_{\gamma\leq\alpha\wedge\alpha^{\prime}}2^{2\sigma d\left(  \gamma\right)  },
\]
and so the left side above is approximately
\[
\sum_{\alpha,\alpha^{\prime}\in\mathcal{T}_{n}}\sum_{\gamma\leq\alpha
\wedge\alpha^{\prime}}2^{2\sigma d\left(  \gamma\right)  }f\left(
\alpha\right)  \mu\left(  \alpha\right)  f\left(  \alpha^{\prime}\right)
\mu\left(  \alpha^{\prime}\right)  =\sum_{\gamma\in\mathcal{T}_{n}}2^{2\sigma
d\left(  \gamma\right)  }I^{\ast}f\left(  \gamma\right)  ^{2}.
\]
Thus for $0<\sigma<1/2$, $\mu$ is $B_{2}^{\sigma}\left(  \mathbb{B}%
_{n}\right)  $-Carleson if and only if (\ref{treecond'dual}) holds where
$\mathcal{T}_{n}$ ranges over all unitary rotations of a fixed Bergman tree.
By Theorem \ref{Tars}, this is equivalent to the tree condition
(\ref{treecond}) where $\mathcal{T}_{n}$ ranges over all unitary rotations of
a fixed Bergman tree. However, we need only consider a fixed Bergman tree
$\mathcal{T}_{n}$ since if $\mu$ is a positive measure on the ball whose
discretization $\mu_{\mathcal{T}_{n}}$ on $\mathcal{T}_{n}$ satisfies the tree
condition, then its discretization $\mu_{U\mathcal{T}_{n}}$ to any unitary
rotation $U\mathcal{T}_{n}$ also satisfies the tree condition (with a possibly
larger, but controlled constant). Indeed, Theorem \ref{Tars} shows that
$\mu_{\mathcal{T}_{n}}$ is $B_{2}^{\sigma}\left(  \mathcal{T}_{n}\right)
$-Carleson, and hence so is the fattened measure defined by
\[
\mu_{\mathcal{T}_{n}}^{\natural}\left(  \alpha\right)  =\sum_{d\left(
\alpha,\beta\right)  \leq N}\mu_{\mathcal{T}_{n}}\left(  \beta\right)
,\;\;\;\;\;\alpha\in\mathcal{T}_{n}.
\]
Since $\mu_{U\mathcal{T}_{n}}$ is pointwise dominated by $\mu_{\mathcal{T}%
_{n}}^{\natural}$ for $N$ sufficiently large, $\mu_{U\mathcal{T}_{n}}$ is
$B_{2}^{\sigma}\left(  \mathcal{T}_{n}\right)  $-Carleson as well, hence
satisfies the tree condition (\ref{treecond}) with $U\mathcal{T}_{n}$ in place
of $\mathcal{T}_{n}$.

Finally, we note that in the case $\sigma\geq1/2$, the above argument,
together with the inequality
\[
\left\vert \operatorname{Re}\left(  \frac{1}{1-\overline{w}\cdot w^{\prime}%
}\right)  ^{2\sigma}\right\vert \leq\left\vert \frac{1}{1-\overline{w}\cdot
w^{\prime}}\right\vert ^{2\sigma},
\]
shows that the tree condition (\ref{treecond}) is sufficient for $\mu$ to be a
$B_{2}^{\sigma}\left(  \mathbb{B}_{n}\right)  $-Carleson measure. This
completes the proof of Theorem \ref{Besov'}.

\subsection{The case $\sigma=1/2$: The Drury-Arveson Hardy space $H_{n}^{2}$}

The above theorem just misses capturing the Drury-Arveson Hardy space
$H_{n}^{2}=B_{2}^{1/2}\left(  \mathbb{B}_{n}\right)  $. If we take
$\sigma=1/2$ in the above proof, then (\ref{realbelow}) combined with the
first inequality in (\ref{disckernel}) is weakened to the inequality
\begin{equation}
\operatorname{Re}\frac{1}{1-\overline{z}\cdot z^{\prime}}=\frac
{\operatorname{Re}\left(  1-\overline{z}\cdot z^{\prime}\right)  }{\left\vert
1-\overline{z}\cdot z^{\prime}\right\vert ^{2}}\geq c+c2^{2d\left(
\alpha\wedge\alpha^{\prime}\right)  -d^{\ast}\left(  \left[  \alpha\right]
\wedge\left[  \alpha^{\prime}\right]  \right)  },\;\;\;\;\;z\in K_{\alpha
},z^{\prime}\in K_{\alpha^{\prime}}, \label{below'}%
\end{equation}
(see below for the definition of $d^{\ast}\left(  \left[  \alpha\right]
\wedge\left[  \alpha^{\prime}\right]  \right)  $ related to a quotient tree
$\mathcal{R}_{n}$ of the Bergman tree $\mathcal{T}_{n}$) which does not lead
to the tree condition (\ref{treecond}). We will however modify the proof so as
to give a characterization in Theorem \ref{Arvcom} below of the Carleson
measures for $H_{n}^{2}=B_{2}^{1/2}\left(  \mathbb{B}_{n}\right)  $ in terms
of the simple condition (\ref{Arvsimple}) and the \textquotedblleft
split\textquotedblright\ tree condition (\ref{fullsplittreecondition}) given
below. We will proceed by three propositions, the first reducing the Carleson
measure embedding for $H_{n}^{2}$ to a positive bilinear inequality on the ball.

\begin{proposition}
\label{realreduction}Let $\mu$ be a positive measure on the ball
$\mathbb{B}_{n}$. Then $\mu$ is $H_{n}^{2}$-Carleson if and only if the
bilinear inequality
\begin{equation}
\int_{\mathbb{B}_{n}}\int_{\mathbb{B}_{n}}\left(  \operatorname{Re}\frac
{1}{1-\overline{z}\cdot z^{\prime}}\right)  f\left(  z^{\prime}\right)
d\mu\left(  z^{\prime}\right)  g\left(  z\right)  d\mu\left(  z\right)  \leq
C\left\Vert f\right\Vert _{L^{2}\left(  \mu\right)  }\left\Vert g\right\Vert
_{L^{2}\left(  \mu\right)  }, \label{Rebil}%
\end{equation}
holds for all $f,g\geq0$. Moreover, provided we use the $H_{n}^{2}$ norm for
Carleson measures (but not the $B_{2}^{1/2}\left(  \mathbb{B}_{n}\right)  $
norm) the constants implicit in the above statement are independent of
dimension $n$.
\end{proposition}

%

\proof
This is immediate from Lemma \ref{carllem}.

\medskip

We will proceed from the continuous bilinear inequality (\ref{Rebil}) in two
steps. First we obtain Proposition \ref{discreduction} which states that
(\ref{Rebil}) is equivalent to a family of discrete inequalities involving
positive quantities. In the section following that we give necessary and
sufficient conditions for the discrete inequalities to hold.

However before doing those things we introduce two additional objects
associated to the tree $\mathcal{T}_{n}.$ The first is a decomposition of
$\mathcal{T}_{n}$ into a set of equivalence classes called rings. The rings
will help provide a language for a precise description of the local size of
the integration kernel in (\ref{Rebil}). Second, we introduce a notion of a
unitary rotation of $\mathcal{T}_{n}.$ As is often the case, when we pass from
a discrete inequality to a continuous one technical problems arise associated
with edge effects. We will deal with those by averaging over unitary rotations
of $\mathcal{T}_{n}.$

\subsubsection{A modified Bergman tree $\mathcal{T}_{n}$ and its quotient tree
$\mathcal{R}_{n}$\label{ring}}

We begin by recalling the main features of the construction of $\mathcal{T}%
_{n}$ given in \cite{ArRoSa2}, and describe the modification we need. Recall
that $\beta$ is the Bergman metric on the unit ball $\mathbb{B}_{n}$ in
$\mathbb{C}^{n}$. Note that for each $r>0$%
\[
\mathcal{S}_{r}=\partial B_{\beta}\left(  0,r\right)  =\left\{  z\in
\mathbb{B}_{n}:\beta\left(  0,z\right)  =r\right\}
\]
is a Euclidean sphere centered at the origin. In fact, by (1.40) in \cite{Zhu}
we have $\beta\left(  0,z\right)  =\tanh^{-1}\left\vert z\right\vert $, and
so
\begin{align}
1-\left\vert z\right\vert ^{2}  &  =1-\tanh^{2}\beta\left(  0,z\right)
\label{connection}\\
&  =\frac{4}{e^{2\beta\left(  0,z\right)  }+2+e^{-2\beta\left(  0,z\right)  }%
}\nonumber\\
&  \approx4e^{-2\beta\left(  0,z\right)  }\nonumber
\end{align}
for $\beta\left(  0,z\right)  $ large. We recall the following elementary
abstract construction from \cite{ArRoSa2} (Lemma 7 on page 18).

\begin{lemma}
\label{spherelemma}Let $\left(  X,d\right)  $ be a separable metric space and
$\lambda>0$. There is a denumerable set of points $E=\left\{  x_{j}\right\}
_{j=1}^{\infty\text{ or }J}$ and a corresponding set of Borel subsets $Q_{j} $
of $X$ satisfying
\begin{align}
X  &  =\cup_{j=1}^{\infty\text{ or }J}Q_{j},\label{tile}\\
Q_{i}\cap Q_{j}  &  =\phi,\;\;\;\;\;i\neq j,\nonumber\\
B\left(  x_{j},\lambda\right)   &  \subset Q_{j}\subset B\left(
x_{j},2\lambda\right)  ,\;\;\;\;\;j\geq1.\nonumber
\end{align}

\end{lemma}

We refer to the sets $Q_{j}$ as \emph{qubes} centered at $x_{j}$. In
\cite{ArRoSa2}, we applied Lemma \ref{spherelemma} to the spheres
$\mathcal{S}_{r} $ for $r>0$ as follows. Fix \emph{structural constants}
$\theta,\lambda>0$. For $N\in\mathbb{N}$, apply the lemma to the metric space
$\left(  \mathcal{S}_{N\theta},\beta\right)  $ to obtain points $\left\{
z_{j}^{N}\right\}  _{j=1}^{J}$ and qubes $\left\{  Q_{j}^{N}\right\}
_{j=1}^{J} $ in $\mathcal{S}_{N\theta}$ satisfying (\ref{tile}). For the
remainder of this subsection we assume $\theta=\frac{\ln2}{2}$ and $\lambda=1$.

However, we now wish to facilitate the definition of an equivalence relation
that identifies qubes \textquotedblleft lying in the same complex line
intersected with the sphere\textquotedblright. To achieve this we recall the
projective space $\mathbb{C}P(n-1)$ can be realized as the set of all complex
circles $\left[  \zeta\right]  =\left\{  e^{is}\zeta:e^{is}\in\mathbb{T}%
\right\}  $, $\zeta\in\partial\mathbb{B}_{n}$, in the unit sphere (for $n=2$
these circles give the Hopf fibration of the real $3-$sphere). In \cite{AhCo}
an induced Koranyi metric was defined on $\mathbb{C}P(n-1)$ by
\[
d\left(  \left[  \eta\right]  ,\left[  \zeta\right]  \right)  =\inf\left\{
d\left(  e^{is}\eta,e^{it}\zeta\right)  :e^{is},e^{it}\in\mathbb{T}\right\}
\]
where $d\left(  \eta,\zeta\right)  =\left\vert 1-\overline{\eta}\cdot
\zeta\right\vert ^{\frac{1}{2}}$. We scale this construction to the sphere
$\mathcal{S}_{r}$ by defining $\mathbb{P}_{r}$ to be the projective space of
complex circles $\left[  \zeta\right]  =\left\{  e^{is}\zeta:e^{is}%
\in\mathbb{T}\right\}  $, $\zeta\in\mathcal{S}_{r}$, in the sphere
$\mathcal{S}_{r}$ with induced Bergman metric
\[
\beta\left(  \left[  \eta\right]  ,\left[  \zeta\right]  \right)
=\inf\left\{  \beta\left(  e^{is}\eta,e^{it}\zeta\right)  :e^{is},e^{it}%
\in\mathbb{T}\right\}  .
\]
For $N\in\mathbb{N}$, we now apply Lemma \ref{spherelemma} to the projective
metric space $\left(  \mathbb{P}_{N\theta},\beta\right)  $ to obtain
projective points (complex circles) $\left\{  \mathsf{w}_{j}^{N}\right\}
_{j=1}^{J},$ $J$ depending on $N,$ in $\mathbb{P}_{N\theta}$ and unit
projective qubes $\left\{  \mathsf{Q}_{j}^{N}\right\}  _{j=1}^{J}$ contained
in $\mathbb{P}_{N\theta}$ satisfying (\ref{tile}). For each $N$ and $j$ we
define points $\left\{  z_{j,i}^{N}\right\}  _{i=1}^{M}$ on the complex circle
$\mathsf{w}_{j}^{N}$ that are approximately distance $1$ from their neighbours
in the Bergman metric: $\beta\left(  z_{j,i}^{N},z_{j,i+1}^{N}\right)
\approx1$ for $1\leq i\leq M$ ($z_{j,M+1}^{N}=z_{j,1}^{N}$). We then define
corresponding \emph{qubes} $\left\{  Q_{j,i}^{N}\right\}  _{i}$ so that
$\mathsf{Q}_{j}^{N}=\cup_{i}Q_{j,i}^{N}$, and so that (\ref{tile}) holds in
the metric space $\left(  \mathcal{S}_{N\theta},\beta\right)  $ for the
collection $\left\{  Q_{j,i}^{N}\right\}  _{j,i}$.

For $z\in\mathbb{B}_{n}$, let $P_{r}z$ denote the radial projection of $z$
onto the sphere $\mathcal{S}_{r}$. We now define subsets $K_{j,i}^{N}$ of
$\mathbb{B}_{n}$ by $K_{1}^{0}=\left\{  z\in\mathbb{B}_{n}:\beta\left(
0,z\right)  <\theta\right\}  $ and
\[
K_{j,i}^{N}=\left\{  z\in\mathbb{B}_{n}:N\theta\leq d\left(  0,z\right)
<\left(  N+1\right)  \theta,\;P_{N\theta}z\in Q_{j,i}^{N}\right\}
,\;\;\;\;\;N\geq1\text{ and }j,i\geq1.
\]
We define corresponding points $c_{j,i}^{N}\in K_{j,i}^{N}$ by
\[
c_{j,i}^{N}=P_{\left(  N+\frac{1}{2}\right)  \theta}\left(  z_{j,i}%
^{N}\right)  .
\]
We will refer to the subset $K_{j,i}^{N}$ of $\mathbb{B}_{n}$ as a \emph{kube}
centered at $c_{j,i}^{N}$ (while $K_{1}^{0}$ is centered at $0$). Similarly we
define \emph{projective kubes} \textsf{K}$_{j}^{N}=\cup_{i}K_{j,i}^{N}$ with
centre \textsf{c}$_{j}^{N}=P_{\left(  N+\frac{1}{2}\right)  \theta}\left(
\text{\textsf{w}}_{j}^{N}\right)  $.

Define a tree structure on the collection of all projective kubes
\[
\mathcal{R}_{n}=\left\{  \mathsf{K}_{j}^{N}\right\}  _{N\geq0,j\geq1}%
\]
by declaring that \textsf{K}$_{i}^{N+1}$ is a child of \textsf{K}$_{j}^{N}$,
written \textsf{K}$_{i}^{N+1}\in\mathcal{C}\left(  \mathsf{K}_{j}^{N}\right)
$, if the projection $P_{N\theta}\left(  \mathsf{w}_{i}^{N+1}\right)  $ of the
circle \textsf{w}$_{i}^{N+1}$ onto the sphere $\mathcal{S}_{N\theta}$ lies in
the projective qube \textsf{Q}$_{j}^{N}$. In the case $N=0$, we declare every
kube \textsf{K}$_{j}^{1}$ to be a child of the root kube \textsf{K}$_{1}^{0}$.
An element $\mathsf{K}_{j}^{N}$ is, roughly, the orbit of a single kube under
a circle action; thus we often refer to them as rings and to $\mathcal{R}_{n}$
as the ring tree. One can think of the ring tree $\mathcal{R}_{n}$ as a
\textquotedblleft quotient tree\textquotedblright\ of the Bergman tree
$\mathcal{T}_{n}$ by the one-parameter family of slice rotations $z\rightarrow
e^{is}z$, $e^{is}\in\mathbb{T}$.

We will now define a tree structure on the collection of kubes
\[
\mathcal{T}_{n}=\left\{  K_{j,i}^{N}\right\}  _{N\geq0\text{ and }j,i\geq1}%
\]
that is compatible with the above tree structure on the collection of
projective kubes $\mathcal{R}_{n}$. To this end, we reindex the kubes
$\left\{  K_{j,i}^{N}\right\}  _{N\geq0\text{ and }j,i\geq1}$ as $\left\{
K_{j}^{N}\right\}  _{N\geq0,j\geq1}$ and define an equivalence relation
$\thicksim$ on the reindexed collection $\left\{  K_{j}^{N}\right\}  _{j}$ by
declaring kubes equivalent that lie in the same projective kube: $K_{i}%
^{N}\thicksim K_{k}^{N}$ if and only if there is a projective kube
$\mathsf{K}_{j}^{N}$ such that $K_{i}^{N},K_{k}^{N}\in\mathsf{K}_{j}^{N}$.
Given $K_{i}^{N}\in\mathcal{T}_{n}$, we denote by $\left[  K_{i}^{N}\right]  $
the equivalence class of $K_{i}^{N}$, which can of course be identified with a
projective kube in $\mathcal{R}_{n}$. Define the tree structure on
$\mathcal{T}_{n}$ by declaring that $K_{i}^{N+1}$ is a child of $K_{j}^{N}$,
written $K_{i}^{N+1}\in\mathcal{C}\left(  K_{j}^{N}\right)  $, if the
projection $P_{N\theta}\left(  z_{i}^{N+1}\right)  $ of $z_{i}^{N+1}$ onto the
sphere $\mathcal{S}_{N\theta}$ lies in the qube $Q_{j}^{N}$. Note that by
construction, it follows that $\left[  K_{i}^{N+1}\right]  $ is then also a
child of $\left[  K_{j}^{N}\right]  $ in $\mathcal{R}_{n}$. In the case $N=0$,
we declare every kube $K_{j}^{1}$ to be a child of the root kube $K_{1}^{0}$.

We will typically write $\alpha,\beta,\gamma$ etc. to denote elements
$K_{j}^{N}$ of the tree $\mathcal{T}_{n}$ when the correspondence with the
unit ball $\mathbb{B}_{n}$ is immaterial. We will write $K_{\alpha}$ for the
kube $K_{j}^{N}$ and $c_{\alpha}$ for its center $c_{j}^{N}$ when the
correspondence matters. Sometimes we will further abuse notation by using
$\alpha$ to denote the center $c_{\alpha}=c_{j}^{N}$ of the kube $K_{\alpha
}=K_{j}^{N}$. Similarly, we will typically write $A,B,C$ etc. to denote
elements $\mathsf{K}_{j}^{N}$ of the ring tree $\mathcal{R}_{n}$ when the
correspondence with the unit ball $\mathbb{B}_{n}$ is immaterial, and we will
write $\mathsf{K}_{A}$ for the projective kube $\mathsf{K}_{j}^{N}$
corresponding to $A$ when the correspondence matters. Finally, for $\alpha
\in\mathcal{T}_{n}$, we denote by $\left[  \alpha\right]  $ the ring in
$\mathcal{R}_{n}$ that corresponds to the equivalence class of $\alpha$. The
following compatibility relations hold for $\alpha,\beta\in\mathcal{T}_{n} $
and $A,B\in\mathcal{R}_{n}$:
\begin{align}
\beta &  \leq\alpha\Longrightarrow\left[  \beta\right]  \leq\left[
\alpha\right]  ,\label{cr}\\
B  &  \leq A\Longleftrightarrow\text{ for every }\alpha\in A\text{ there is
}\beta\in B\text{ with }\beta\leq\alpha,\nonumber
\end{align}

We will also need the notion of a unitary rotation of $\mathcal{T}_{n}$. For
each $w\in\mathbb{B}_{n}$ define $\left\langle w\right\rangle \in
\mathcal{T}_{n}$ to be the unique tree element such that $w\in K_{\left\langle
w\right\rangle }$, and define $\left[  w\right]  \in\mathcal{R}_{n}$ to be the
unique ring tree element such that $w\in\mathsf{K}_{\left[  w\right]  }$ (here
we are viewing the projective kube $\mathsf{K}_{\left[  w\right]  }$ as a
subset of the ball $\mathbb{B}_{n}$). The notation is coherent; the ring
containing $w$ is the equivalence class in $\mathcal{T}$ containing the kube
$K_{\left\langle w\right\rangle };$ $\left[  w\right]  =\left[  \left\langle
w\right\rangle \right]  $. Let $\mathcal{U}_{n}$ be the unitary group with
Haar measure $dU$. Recall that we may identify $\alpha$ with the center
$c_{\alpha}$ of the Bergman kube $K_{\alpha}$ (subsubsection 5.2.1 of
\cite{ArRoSa2}). If we define $K_{U^{-1}\alpha}=U^{-1}K_{\alpha}$, then
$\left\{  K_{U^{-1}\alpha}\right\}  _{\alpha\in\mathcal{T}_{n}}\equiv\left\{
U^{-1}K_{\alpha}\right\}  _{\alpha\in\mathcal{T}_{n}}$ is the Bergman grid
rotated by $U^{-1}$, and
\begin{equation}
\alpha=\left\langle Uz\right\rangle \Leftrightarrow Uz\in K_{\alpha
}\Leftrightarrow z\in U^{-1}K_{\alpha}\Leftrightarrow z\in K_{U^{-1}\alpha}.
\label{rotate}%
\end{equation}
We denote by $U^{-1}\mathcal{T}_{n}$ the tree corresponding to the rotated
grid $\left\{  K_{U^{-1}\alpha}\right\}  _{\alpha\in\mathcal{T}_{n}}$. The
same construction applies to obtain the rotated ring tree $U^{-1}%
\mathcal{R}_{n}$, and the compatibility relation (\ref{cr}) persists between
$U^{-1}\mathcal{T}_{n}$ and $U^{-1}\mathcal{R}_{n}$ since $\left[
U^{-1}\alpha\right]  =U^{-1}\left[  \alpha\right]  $. We also define
$\left\langle w\right\rangle _{U}\in U^{-1}\mathcal{T}_{n}$ and $\left[
w\right]  _{U}\in U^{-1}\mathcal{R}_{n}$ by $w\in K_{\left\langle
w\right\rangle _{U}}$ and $w\in\mathsf{K}_{\left[  w\right]  _{U}}$
respectively. Then from (\ref{rotate}) we have $\alpha=\left\langle
Uz\right\rangle \Leftrightarrow U^{-1}\alpha=\left\langle z\right\rangle _{U}$.

We will also want distance functions with controlled behavior under unitary
rotations. We now extend the definition of the tree distance $d_{U^{-1}%
\mathcal{T}_{n}}$ and the ring distance $d_{U^{-1}\mathcal{R}_{n}}$ on the
rotations $U^{-1}\mathcal{T}_{n}$ and $U^{-1}\mathcal{R}_{n}$ to
$\mathbb{B}_{n}\times\mathbb{B}_{n}$ by
\begin{align*}
d_{U^{-1}\mathcal{T}_{n}}\left(  z,w\right)   &  =d_{U^{-1}\mathcal{T}_{n}%
}\left(  \left\langle z\right\rangle _{U},\left\langle w\right\rangle
_{U}\right)  ,\;\;\;\;\;z,w\in\mathbb{B}_{n},\\
d_{U^{-1}\mathcal{R}_{n}}\left(  z,w\right)   &  =d_{U^{-1}\mathcal{R}_{n}%
}\left(  \left[  z\right]  _{U},\left[  w\right]  _{U}\right)
\;\;\;\;\;z,w\in\mathbb{B}_{n}.
\end{align*}
We have the following identities:
\begin{align*}
d_{U^{-1}\mathcal{T}_{n}}\left(  z,w\right)   &  =d_{\mathcal{T}_{n}}\left(
Uz,Uw\right)  ,\\
d_{U^{-1}\mathcal{R}_{n}}\left(  z,w\right)   &  =d_{\mathcal{R}_{n}}\left(
Uz,Uw\right)  .
\end{align*}
We often write simply $d$ when the underlying tree is evident, especially when
it is $\mathcal{T}_{n}$ or $\mathcal{R}_{n}$, and provided this will cause no
confusion; e.g. $d\left(  z,w\right)  =d_{\mathcal{T}_{n}}\left(  \left\langle
z\right\rangle ,\left\langle w\right\rangle \right)  $.

Finally, we introduce yet another structure on the trees $\mathcal{T}_{n}$ and
$\mathcal{R}_{n}$, namely the $\emph{unitary}$ tree distance $d^{\ast}$ given
by
\begin{align*}
d^{\ast}\left(  \alpha,\beta\right)   &  =\inf_{U\in\mathcal{U}_{n}%
}d_{\mathcal{T}_{n}}\left(  Uc_{\alpha},Uc_{\beta}\right)  =\inf
_{U\in\mathcal{U}_{n}}d_{U^{-1}\mathcal{T}_{n}}\left(  c_{\alpha},c_{\beta
}\right)  ,\\
d^{\ast}\left(  \left[  \alpha\right]  ,\left[  \beta\right]  \right)   &
=\inf_{U\in\mathcal{U}_{n}}d_{\mathcal{R}_{n}}\left(  Uc_{\alpha},Uc_{\beta
}\right)  =\inf_{U\in\mathcal{U}_{n}}d_{U^{-1}\mathcal{R}_{n}}\left(
c_{\alpha},c_{\beta}\right)  .
\end{align*}
Note that the analogous definitions of $d^{\ast}$ on the rotated trees
$U^{-1}\mathcal{T}_{n}$ and $U^{-1}\mathcal{R}_{n}$ coincide with the above
definitions, so that we can write simply $d^{\ast}$ for $d_{U^{-1}%
\mathcal{T}_{n}}^{\ast}$ or $d_{U^{-1}\mathcal{R}_{n}}^{\ast}$ without
ambiguity. We now \textbf{define} $d^{\ast}\left(  \alpha\wedge\beta\right)  $
and $d^{\ast}\left(  A\wedge B\right)  $ in analogy with the corresponding
formulas for $d$; namely
\begin{align*}
2d^{\ast}\left(  \alpha\wedge\beta\right)   &  =d^{\ast}\left(  \alpha\right)
+d^{\ast}\left(  \beta\right)  -d^{\ast}\left(  \alpha,\beta\right)
,\;\;\;\;\;\alpha,\beta\in U^{-1}\mathcal{T}_{n},\\
2d^{\ast}\left(  A\wedge B\right)   &  =d^{\ast}\left(  A\right)  +d^{\ast
}\left(  B\right)  -d^{\ast}\left(  A,B\right)  ,\;\;\;\;\;A,B\in
U^{-1}\mathcal{R}_{n},
\end{align*}
so that
\begin{align*}
d^{\ast}\left(  \alpha\wedge\beta\right)   &  =\sup_{U\in\mathcal{U}_{n}%
}d_{U^{-1}\mathcal{T}_{n}}\left(  \left\langle c_{\alpha}\right\rangle
_{U}\wedge\left\langle c_{\beta}\right\rangle _{U}\right)  ,\;\;\;\;\;\alpha
,\beta\in\mathcal{T}_{n},\\
d^{\ast}\left(  \left[  \alpha\right]  \wedge\left[  \beta\right]  \right)
&  =\sup_{U\in\mathcal{U}_{n}}d_{U^{-1}\mathcal{R}_{n}}\left(  \left[
c_{\alpha}\right]  _{U}\wedge\left[  c_{\beta}\right]  _{U}\right)
\;\;\;\;\alpha,\beta\in\mathcal{T}_{n},
\end{align*}
The unitary distance $d^{\ast}$ on the ring tree $\mathcal{R}_{n}$ will play a
crucial role in discretizing the bilinear inequality (\ref{Rebil}) in the next
section. (Actually $d^{\ast}\left(  \left[  \alpha\right]  \wedge\left[
\beta\right]  \right)  $ is a function of the pair $\left(  \left[
\alpha\right]  ,\left[  \beta\right]  \right)  $ not of the ring tree element
$\left[  \alpha\right]  \wedge\left[  \beta\right]  $. We indulge in this
slight abuse of notation because below $d^{\ast}\left(  \left[  \alpha\right]
\wedge\left[  \beta\right]  \right)  $ will have the role of a substitute for
$d\left(  \left[  \alpha\right]  \wedge\left[  \beta\right]  \right)  .)$

\subsubsection{The discrete inequality}

We can now state the discretization inequality.

\begin{proposition}
\label{discreduction}Let $\mu$ be a positive measure on $\mathbb{B}_{n}$. Then
the bilinear inequality (\ref{Rebil}) is equivalent to having, for all unitary
rotations of a fixed Bergman tree $\mathcal{T}_{n}$ together with the
corresponding rotations of the associated ring tree $\mathcal{R}_{n}$, and
with constants independent of the rotation, the discrete inequality,
\begin{equation}
\sum_{\alpha\in\mathcal{T}_{n}}\left\vert T_{\mu}g\left(  \alpha\right)
\right\vert ^{2}\mu\left(  \alpha\right)  \leq C\sum_{\alpha\in\mathcal{T}%
_{n}}\left\vert g\left(  \alpha\right)  \right\vert ^{2}\mu\left(
\alpha\right)  ,\;\;\;\;\;g\geq0, \label{discnew}%
\end{equation}
where $T_{\mu}$ is the positive linear operator on the tree $\mathcal{T}_{n} $
given by,
\begin{equation}
T_{\mu}g\left(  \alpha\right)  =\sum_{\beta\in\mathcal{T}_{n}}2^{2d\left(
\alpha\wedge\beta\right)  -d^{\ast}\left(  \left[  \alpha\right]
\wedge\left[  \beta\right]  \right)  }g\left(  \beta\right)  \mu\left(
\beta\right)  ,\;\;\;\;\;\alpha\in\mathcal{T}_{n}. \label{deffracnew}%
\end{equation}
Equivalently, (\ref{discnew}) can be replaced by the bilinear estimate
\begin{align}
&  \sum_{\alpha,\alpha^{\prime}\in\mathcal{T}_{n}}2^{2d\left(  \alpha
\wedge\alpha^{\prime}\right)  -d^{\ast}\left(  \left[  \alpha\right]
\wedge\left[  \alpha^{\prime}\right]  \right)  }f\left(  \alpha\right)
\mu\left(  \alpha\right)  g\left(  \alpha^{\prime}\right)  \mu\left(
\alpha^{\prime}\right) \label{donedeal}\\
\;\;\;\;\;  &  \leq C\left\{  \sum_{\alpha\in\mathcal{T}_{n}}f\left(
\alpha\right)  ^{2}\mu\left(  \alpha\right)  \right\}  ^{\frac{1}{2}}\left\{
\sum_{\alpha^{\prime}\in\mathcal{T}_{n}}g\left(  \alpha^{\prime}\right)
^{2}\mu\left(  \alpha^{\prime}\right)  \right\}  ^{\frac{1}{2}},\nonumber
\end{align}
where $\mathcal{T}_{n}$ ranges over all unitary rotations of a fixed Bergman tree.
\end{proposition}

%

\proof
We first establish (\ref{donedeal}), i.e. we discretize the bilinear
inequality (\ref{Rebil}) to the following discrete bilinear inequality valid
for all unitary rotations $U^{-1}\mathcal{T}_{n}$ of the Bergman tree
$\mathcal{T}_{n}$:
\begin{align}
&  \sum_{\alpha,\alpha^{\prime}\in U^{-1}\mathcal{T}_{n}}2^{2d\left(
\alpha\wedge\alpha^{\prime}\right)  -d^{\ast}\left(  \left[  \alpha\right]
\wedge\left[  \alpha^{\prime}\right]  \right)  }f\left(  \alpha\right)
\mu\left(  \alpha\right)  g\left(  \alpha^{\prime}\right)  \mu\left(
\alpha^{\prime}\right) \label{done}\\
&  \;\;\;\;\;\leq C\left\{  \sum_{\alpha\in U^{-1}\mathcal{T}_{n}}f\left(
\alpha\right)  ^{2}\mu\left(  \alpha\right)  \right\}  ^{\frac{1}{2}}\left\{
\sum_{\alpha^{\prime}\in U^{-1}\mathcal{T}_{n}}g\left(  \alpha^{\prime
}\right)  ^{2}\mu\left(  \alpha^{\prime}\right)  \right\}  ^{\frac{1}{2}%
},\nonumber
\end{align}
for all $U\in\mathcal{U}_{n}$, $f,g\geq0$ on $U^{-1}\mathcal{T}_{n}$ and where
the constant $C$ is independent of $U,f,g$. At a crucial point in the argument
below, we need to estimate the distance $1-\left\vert \overline{z}\cdot
z^{\prime}\right\vert ^{2}$ in terms of the tree structure, and this is what
leads to the associated ring tree $\mathcal{R}_{n}$ and the quantities
$d\left(  \left[  \alpha\right]  \wedge\left[  \alpha^{\prime}\right]
\right)  $ and $d^{\ast}\left(  \left[  \alpha\right]  \wedge\left[
\alpha^{\prime}\right]  \right)  $. Recall that a \emph{slice} of the ball
$\mathbb{B}_{n}$ is the intersection of the ball with a complex line through
the origin. In particular, every point $z\in\mathbb{B}_{n}\setminus\left\{
0\right\}  $ lies in a unique slice
\[
S_{z}=\left\{  \left(  e^{i\theta}z_{1},...,e^{i\theta}z_{n}\right)
:\theta\in\left[  0,2\pi\right)  \right\}  .
\]
We define two elements $\alpha$ and $\alpha^{\prime}$ of the Bergman tree
$\mathcal{T}_{n}$ to be \emph{slice-related} if $\alpha\thicksim\alpha
^{\prime}$ where, recall, $\thicksim$ denotes that the two elements lie in the
same projective kube. Now given $\alpha,\alpha^{\prime}\in\mathcal{T}_{n},$
let
\[
\left[  o,\alpha\right]  =\left\{  o,\alpha_{1},...,\alpha_{m}=\alpha\right\}
\text{ and }\left[  o,\alpha^{\prime}\right]  =\left\{  o,\alpha_{1}^{\prime
},...,\alpha_{m^{\prime}}^{\prime}=\alpha^{\prime}\right\}
\]
be the geodesics from the root $o$ to $\alpha,\alpha^{\prime}$ respectively.
We then have from (\ref{cr}) that $\alpha_{k}$ and $\alpha_{k}^{\prime}$ are
slice-related if and only if $k\leq d\left(  \left[  \alpha\right]
\wedge\left[  \alpha^{\prime}\right]  \right)  $.

It may help the reader to visualize $d\left(  \left[  \alpha\right]
\wedge\left[  \alpha^{\prime}\right]  \right)  $ in the following way. Imagine
that each slice $S$ is thickened to a \emph{slab} $\mathcal{S}$ of width one
in the Bergman metric. Thus in the Euclidean metric, a slab $\mathcal{S}$ is a
lens whose \textquotedblleft thickness\textquotedblright\ at any point is
roughly the square root of the distance to the boundary of the ball
$\partial\mathbb{B}_{n}$. Moreover, given $z\in\mathbb{B}_{n}$, we denote by
$\mathcal{S}_{z}$ the slab corresponding to the slice $S_{z}$, but truncated
by intersecting it with $B\left(  0,\left\vert z\right\vert \right)  $. The
slabs $\mathcal{S}_{c_{\alpha}}$ and $\mathcal{S}_{c_{\alpha^{\prime}}}$
associated with the unique slices $S_{c_{\alpha}}$ and $S_{c_{\alpha^{\prime}%
}}$ through $c_{\alpha}$ and $c_{\alpha^{\prime}}$ will intersect in a
\textquotedblleft disc\textquotedblright\ of radius roughly $d\left(  \left[
\alpha\right]  \wedge\left[  \alpha^{\prime}\right]  \right)  $ in the Bergman
metric - at least this will be the case for a \textquotedblleft fixed
proportion\textquotedblright\ of pairs $\left(  \alpha,\alpha^{\prime}\right)
$, and will be literally true for all pairs with the unitary quantity
$d^{\ast}\left(  \left[  \alpha\right]  \wedge\left[  \alpha^{\prime}\right]
\right)  $ in place of $d\left(  \left[  \alpha\right]  \wedge\left[
\alpha^{\prime}\right]  \right)  $. Note from this picture that $\alpha
_{d\left(  \left[  \alpha\right]  \wedge\left[  \alpha^{\prime}\right]
\right)  }$ is the \emph{exit point} $E_{\alpha^{\prime}}\alpha$ of the
geodesic $\left[  o,\alpha\right]  $ from the slab $\mathcal{S}_{\alpha
^{\prime}}$ associated to the slice $S_{\alpha^{\prime}}$ through
$c_{\alpha^{\prime}}$, and similarly, $\alpha_{d\left(  \left[  \alpha\right]
\wedge\left[  \alpha^{\prime}\right]  \right)  }^{\prime}$ is the exit point
$E_{\alpha}\alpha^{\prime}$ of the geodesic $\left[  o,\alpha^{\prime}\right]
$ from the slab $\mathcal{S}_{\alpha}$. Both points have the same distance
from the root. Note that we can also define $E_{\alpha^{\prime}}\alpha$ as the
intersection of the geodesic $\left[  o,\alpha\right]  $ with the ring
$\left[  \alpha\right]  \wedge\left[  \alpha^{\prime}\right]  $, which we will
denote by $E_{\left[  \alpha\right]  \wedge\left[  \alpha^{\prime}\right]
}\alpha$. Finally, note that since $d\left(  \left[  \alpha\right]
\wedge\left[  \alpha^{\prime}\right]  \right)  =d\left(  E_{\alpha^{\prime}%
}\alpha\right)  =d\left(  E_{\alpha}\alpha^{\prime}\right)  $ and
$\alpha\wedge\alpha^{\prime}=\alpha_{\ell}$ where $\ell=\max\left\{
k:\alpha_{k}=\alpha_{k}^{\prime}\right\}  $, we have that $d\left(  \left[
\alpha\right]  \wedge\left[  \alpha^{\prime}\right]  \right)  $ satisfies
\begin{equation}
d\left(  \alpha\wedge\alpha^{\prime}\right)  \leq d\left(  \left[
\alpha\right]  \wedge\left[  \alpha^{\prime}\right]  \right)  \leq\min\left\{
d\left(  \alpha\right)  ,d\left(  \alpha^{\prime}\right)  \right\}  .
\label{extremes}%
\end{equation}

The key feature of the quantity $d\left(  \left[  \alpha\right]  \wedge\left[
\alpha^{\prime}\right]  \right)  $ is that $2^{-d\left(  \left[
\alpha\right]  \wedge\left[  \alpha^{\prime}\right]  \right)  }$ is
essentially $1-\left|  \overline{z}\cdot z^{\prime}\right|  ^{2}$ for $z\in
K_{\alpha}$, $z^{\prime}\in K_{\alpha^{\prime}}$. More precisely, for each
$z,z^{\prime}\in\mathbb{B}_{n}$, there is a subset $\Sigma$ of the unitary
group $\mathcal{U}_{n}$ with Haar measure $\left|  \Sigma\right|  \geq c>0$
and satisfying
\begin{align}
c2^{-d\left(  \left[  Uz\right]  \wedge\left[  Uz^{\prime}\right]  \right)  }
&  \leq1-\left|  \overline{z}\cdot z^{\prime}\right|  ^{2},\;\;\;\;\;U\in
\Sigma,\label{numest}\\
1-\left|  \overline{z}\cdot z^{\prime}\right|  ^{2}  &  \leq C2^{-d\left(
\left[  Uz\right]  \wedge\left[  Uz^{\prime}\right]  \right)  },\;\;\;\;\;U\in
\mathcal{U}_{n}.\nonumber
\end{align}
In particular, in terms of the unitary ring distance $d^{*}$, we have the
equivalence
\begin{equation}
1-\left|  \overline{z}\cdot z^{\prime}\right|  ^{2}\approx2^{-d^{*}\left(
\left[  z\right]  \wedge\left[  z^{\prime}\right]  \right)  }. \label{inpart}%
\end{equation}
The full force of the first inequality in (\ref{numest}) will not be used
until the next subsubsection when we prove the sufficiency of the simple
condition and split tree condition for (\ref{donedeal}). To prove
(\ref{numest}), let $S=S_{z}$ be the slice through $z$, $\mathcal{S}%
=\mathcal{S}_{z}$ the corresponding slab, denote by $P$ projection from the
ball onto $S$, and by $Q$ its orthogonal projection, so that
\[
Pw=\frac{\overline{z}\cdot w}{\left|  z\right|  ^{2}}z,\;\;\;\;\;Qw=w-Pw.
\]
If $d=d\left(  \left[  z\right]  \wedge\left[  z^{\prime}\right]  \right)  $,
then $\left\langle z^{\prime}\right\rangle _{d}$ is the exit point
$E_{\left\langle z\right\rangle }\left\langle z^{\prime}\right\rangle $ of
$\left[  o,\left\langle z^{\prime}\right\rangle \right]  $ from the slab
$\mathcal{S}$. Since $\mathcal{S}$ is a lens whose Euclidean ``thickness'' at
any point is roughly the square root of the distance from the boundary, we
have
\[
\left|  Q\left(  c_{\left\langle z^{\prime}\right\rangle _{d}}\right)
\right|  \leq C2^{-\frac{1}{2}d\left(  \left\langle z^{\prime}\right\rangle
_{d}\right)  }=C2^{-\frac{1}{2}d}.
\]
Since $z^{\prime}\in K_{\left\langle z^{\prime}\right\rangle }$ where
$\left\langle z^{\prime}\right\rangle \geq\left\langle z^{\prime}\right\rangle
_{d+1}$, we also have
\[
\left|  Q\left(  z^{\prime}\right)  \right|  \leq C2^{-\frac{1}{2}d}.
\]
It now follows that
\begin{align*}
1-\left|  \overline{z}\cdot z^{\prime}\right|  ^{2}  &  =1-\left|  z\right|
^{2}\left|  Pz^{\prime}\right|  ^{2}=1-\left|  z\right|  ^{2}\left(  1-\left|
Qz^{\prime}\right|  ^{2}\right) \\
&  \leq1-\left|  z\right|  ^{2}+C^{2}2^{-d}\\
&  \leq C^{2}2^{-d\left(  \left[  z\right]  \wedge\left[  z^{\prime}\right]
\right)  },
\end{align*}
and since this argument works for any Bergman tree $U^{-1}\mathcal{T}_{n}$,
this yields the second inequality in (\ref{numest}).

To obtain the first inequality in (\ref{numest}), we use a standard averaging
argument as follows. Given $U\in\mathcal{U}_{n}$, if $d=d\left(  \left[
Uz\right]  \wedge\left[  Uz^{\prime}\right]  \right)  $, then $\left(
\left\langle z^{\prime}\right\rangle _{U}\right)  _{d}$ is the exit point
$E_{\left\langle z\right\rangle _{U}}\left\langle z^{\prime}\right\rangle _{U}
$ of $\left[  o,\left\langle z^{\prime}\right\rangle _{U}\right]  $ from the
slab $\mathcal{S}$. Since $c_{\left(  \left\langle z^{\prime}\right\rangle
_{U}\right)  _{d+1}}$ lies outside $\mathcal{S}$, and since $\mathcal{S}$ is a
lens whose Euclidean \textquotedblleft thickness\textquotedblright\ at any
point is roughly the square root of the distance from the boundary, we have
\[
\left\vert Q\left(  c_{\left(  \left\langle z^{\prime}\right\rangle
_{U}\right)  _{d+1}}\right)  \right\vert \geq c2^{-\frac{1}{2}d\left(  \left(
\left\langle z^{\prime}\right\rangle _{U}\right)  _{d}\right)  }=c2^{-\frac
{1}{2}d}.
\]
Since $z^{\prime}\in K_{\left\langle z^{\prime}\right\rangle _{U}}$ where
$\left\langle z^{\prime}\right\rangle _{U}\geq\left(  \left\langle z^{\prime
}\right\rangle _{U}\right)  _{d+1}$, we thus also have
\[
\left\vert Q\left(  z^{\prime}\right)  \right\vert \geq c2^{-\frac{1}{2}d},
\]
for $U$ in a subset $\Sigma$ of the unitary group $\mathcal{U}_{n}$ such that
$\left\vert \Sigma\right\vert \geq c>0$ (the third line of (\ref{tile}) is
used here). It now follows that
\begin{align*}
1-\left\vert \overline{z}\cdot z^{\prime}\right\vert ^{2}  &  =1-\left\vert
z\right\vert ^{2}\left\vert Pz^{\prime}\right\vert ^{2}=1-\left\vert
z\right\vert ^{2}\left(  1-\left\vert Qz^{\prime}\right\vert ^{2}\right) \\
&  \geq1-\left\vert z\right\vert ^{2}\left(  1-c^{2}2^{-d}\right) \\
&  \geq c^{2}2^{-d}=c^{2}2^{-d\left(  \left[  Uz\right]  \wedge\left[
Uz^{\prime}\right]  \right)  },
\end{align*}
for all $U\in\Sigma$, which yields the first inequality in (\ref{numest}).

The main inequalities used in establishing the equivalence of (\ref{Rebil})
and (\ref{done}) are (\ref{below'}), i.e.
\begin{equation}
\operatorname{Re}\frac{1}{1-\overline{z}\cdot z^{\prime}}\geq c+c2^{2d\left(
\alpha\wedge\alpha^{\prime}\right)  -d^{*}\left(  \left[  \alpha\right]
\wedge\left[  \alpha^{\prime}\right]  \right)  },\;\;\;\;\;z\in K_{\alpha
},z^{\prime}\in K_{\alpha^{\prime}}, \label{below''}%
\end{equation}
for all $\alpha,\alpha^{\prime}\in U^{-1}\mathcal{T}_{n}$, $U\in
\mathcal{U}_{n}$, together with a converse obtained by averaging over all
unitary rotations $U^{-1}\mathcal{T}_{n}$ of the Bergman tree $\mathcal{T}%
_{n}$,
\begin{equation}
\operatorname{Re}\frac{1}{1-\overline{z}\cdot z^{\prime}}\leq C+C\int
_{\mathcal{U}_{n}}2^{2d\left(  \left\langle Uz\right\rangle \wedge\left\langle
Uz^{\prime}\right\rangle \right)  -d^{*}\left(  \left[  z\right]
\wedge\left[  z^{\prime}\right]  \right)  }dU. \label{unittree}%
\end{equation}
This latter inequality is analogous to similar inequalities in Euclidean space
used to control an operator by translations of its dyadic version, and the
proof given below is similar.

To prove (\ref{below''}) and (\ref{unittree}), we will use the identity (Lemma
1.3 of \cite{Zhu})
\begin{equation}
1-\overline{\varphi_{a}\left(  w\right)  }\cdot\varphi_{a}\left(  z\right)
=\frac{\left(  1-\overline{a}\cdot a\right)  \left(  1-\overline{w}\cdot
z\right)  }{\left(  1-\overline{w}\cdot a\right)  \left(  1-\overline{a}\cdot
z\right)  },\;\;\;\;\;z,w\in\overline{\mathbb{B}_{n}},a\in\mathbb{B}_{n},
\label{polar}%
\end{equation}
the fact that the Bergman balls $B_{\beta}\left(  a,r\right)  $ are the
ellipsoids (\cite{Rud}, page 29)
\begin{equation}
B_{\beta}\left(  a,r\right)  =\left\{  z\in\mathbb{B}_{n}:\frac{\left|
P_{a}z-c_{a}\right|  ^{2}}{t^{2}\rho_{a}^{2}}+\frac{\left|  Q_{a}z\right|
^{2}}{t^{2}\rho_{a}}<1\right\}  , \label{ellip}%
\end{equation}
where
\[
c_{a}=\frac{\left(  1-t^{2}\right)  a}{1-t^{2}\left|  a\right|  ^{2}}%
,\;\rho_{a}=\frac{1-\left|  a\right|  ^{2}}{1-t^{2}\left|  a\right|  ^{2}},
\]
and $t>0$ satisfies $B_{\beta}\left(  0,r\right)  =B\left(  0,t\right)  $, and
the fact that the projection of $B_{\beta}\left(  a,1\right)  $ onto the
sphere $\partial\mathbb{B}_{n}$ is essentially the nonisotropic Koranyi ball
$Q\left(  \frac{a}{\left|  a\right|  },\sqrt{1-\left|  a\right|  ^{2}}\right)
$ given in (4.11) of \cite{Zhu} by
\begin{equation}
Q\left(  \zeta,\delta\right)  =\left\{  \eta\in\partial\mathbb{B}_{n}:\left|
1-\overline{\eta}\cdot\zeta\right|  ^{\frac{1}{2}}\leq\delta\right\}
,\;\;\;\;\;\zeta\in\partial\mathbb{B}_{n}. \label{ellip'}%
\end{equation}
Indeed, if $c_{\alpha}$ is the center of the Bergman kube $K_{\alpha}$, then
the successor set $S\left(  \alpha\right)  =\cup_{\beta\geq\alpha}K_{\beta}$
consists essentially of all points $z$ lying between $K_{\alpha}$ and its
projection onto the sphere, and from (\ref{ellip}) and (\ref{ellip'}) we then
have
\[
S\left(  \alpha\right)  \approx\left\{  z\in\mathbb{B}_{n}:\left|
1-\overline{c_{\alpha}}\cdot z\right|  \leq1-\left|  c_{\alpha}\right|
^{2}\approx2^{-d\left(  \alpha\right)  }\right\}
\]
in the sense that if
\[
\mathcal{S}_{C}\left(  w\right)  =\left\{  z\in\mathbb{B}_{n}:\left|
1-\overline{w}\cdot z\right|  \leq C\left(  1-\left|  w\right|  ^{2}\right)
\right\}  ,
\]
then there are positive constants $c$ and $C$ such that
\[
\mathcal{S}_{c}\left(  c_{\alpha}\right)  \subset S\left(  \alpha\right)
\subset\mathcal{S}_{C}\left(  c_{\alpha}\right)  ,
\]
where
\[
\mathcal{S}_{C}\left(  c_{\alpha}\right)  \approx\left\{  z\in\mathbb{B}%
_{n}:\left|  1-\overline{c_{\alpha}}\cdot z\right|  \leq C2^{-d\left(
\alpha\right)  }\right\}  .
\]

Using (\ref{polar}) with $a=c_{\alpha}$, $\omega=\varphi_{a}\left(  w\right)
$ and $\zeta=\varphi_{a}\left(  z\right)  $, we see that
\[
\left\vert 1-\overline{\omega}\cdot\zeta\right\vert \leq C,\;\;\;\;\;\omega
,\zeta\in K_{\alpha},
\]
since $\left\vert w\right\vert ,\left\vert z\right\vert \leq\rho<1$ for
$\omega,\zeta\in K_{\alpha}$, and it now follows easily that
\begin{align}
\left\vert 1-\overline{\omega}\cdot\zeta\right\vert  &  \leq C2^{-d\left(
\alpha\right)  },\;\;\;\;\;\omega,\zeta\in S\left(  \alpha\right)
,\label{succset}\\
\left\vert 1-\overline{\omega}\cdot\zeta\right\vert  &  \geq c2^{-d\left(
\alpha\right)  },\;\;\;\;\;\omega\in S\left(  \alpha\right)  ,\zeta
\notin\mathcal{S}_{2C}\left(  \alpha\right)  .\nonumber
\end{align}
Now fix $U\in\mathcal{U}_{n}$, $\alpha\in U^{-1}\mathcal{T}_{n}$,
$\alpha^{\prime}\in U^{-1}\mathcal{T}_{n}$, $z\in K_{\alpha}$, $z^{\prime}\in
K_{\alpha^{\prime}}$ and let $\beta=\alpha\wedge\alpha^{\prime}$ be the
minimum of $\alpha$ and $\alpha^{\prime}$ in the Bergman tree. From the first
inequality in (\ref{succset}), we obtain
\begin{equation}
\left\vert 1-\overline{z^{\prime}}\cdot z\right\vert \leq C2^{-d\left(
\beta\right)  }=C2^{-d\left(  \alpha\wedge\alpha^{\prime}\right)  }.
\label{denest}%
\end{equation}
We now write $\frac{\overline{z}\cdot z^{\prime}}{\left\vert \overline{z}\cdot
z^{\prime}\right\vert }=e^{i\theta}$, where by localizing $z$ and $z^{\prime}$
to lie close together near the boundary of the ball, we may assume that both
$\left\vert \theta\right\vert $ and $1-\left\vert \overline{z^{\prime}}\cdot
z\right\vert ^{2}$ are small, say less than $\varepsilon>0$. We then have
\begin{align*}
\operatorname{Re}\left(  1-\overline{z}\cdot z^{\prime}\right)   &  =\left(
1-\left\vert \overline{z}\cdot z^{\prime}\right\vert \right)  +\left\vert
\overline{z}\cdot z^{\prime}\right\vert \left(  1-\cos\theta\right) \\
&  \approx\left(  1-\left\vert \overline{z}\cdot z^{\prime}\right\vert
^{2}\right)  +\left(  1-\cos^{2}\theta\right) \\
&  =\left(  1-\left\vert \overline{z}\cdot z^{\prime}\right\vert ^{2}\right)
+\sin^{2}\theta\\
&  \approx\left(  1-\left\vert \overline{z}\cdot z^{\prime}\right\vert
^{2}\right)  +\left\vert \operatorname{Im}\left(  1-\overline{z}\cdot
z^{\prime}\right)  \right\vert ^{2}\\
&  =1-\left\vert \overline{z}\cdot z^{\prime}\right\vert ^{2}+\left\vert
1-\overline{z}\cdot z^{\prime}\right\vert ^{2}-\left\vert \operatorname{Re}%
\left(  1-\overline{z}\cdot z^{\prime}\right)  \right\vert ^{2}.
\end{align*}
However, for $\varepsilon>0$ sufficiently small, we may absorb the last term
$\left\vert \operatorname{Re}\left(  1-\overline{z}\cdot z^{\prime}\right)
\right\vert ^{2}$ on the right side into the left side, to obtain
\begin{equation}
\operatorname{Re}\frac{1}{1-\overline{z}\cdot z^{\prime}}=\frac
{\operatorname{Re}\left(  1-\overline{z}\cdot z^{\prime}\right)  }{\left\vert
1-\overline{z}\cdot z^{\prime}\right\vert ^{2}}\approx\frac{1-\left\vert
\overline{z}\cdot z^{\prime}\right\vert ^{2}}{\left\vert 1-\overline{z}\cdot
z^{\prime}\right\vert ^{2}}+1. \label{realcal}%
\end{equation}
Note that (\ref{realcal}) persists for all $z,z^{\prime}\in\mathbb{B}_{n}$
since if $z$ and $z^{\prime}$ do not lie close together near the boundary of
the ball, then $\left\vert 1-\overline{z}\cdot z^{\prime}\right\vert \geq c>0$.

Using (\ref{denest}), (\ref{realcal}) and (\ref{inpart}), we immediately have
the lower bound
\[
\operatorname{Re}\frac{1}{1-\overline{z}\cdot z^{\prime}}\geq c+c2^{2d\left(
\alpha\wedge\alpha^{\prime}\right)  -d^{\ast}\left(  \left[  \alpha\right]
\wedge\left[  \alpha^{\prime}\right]  \right)  },\;\;\;\;\;z\in K_{\alpha
},z^{\prime}\in K_{\alpha^{\prime}},
\]
which is (\ref{below''}). To obtain the converse (\ref{unittree}), we use the
third line in (\ref{tile}) to note that for fixed $z,z^{\prime}\in
\mathbb{B}_{n}$, there is a subset $\Sigma$ of the unitary group
$\mathcal{U}_{n}$ having Haar measure bounded below by a positive constant
$c$, and such that for each $U\in\Sigma$, if $\alpha\in U^{-1}\mathcal{T}_{n}%
$, $\alpha^{\prime}\in U^{-1}\mathcal{T}_{n}$, $z\in K_{\alpha}$, $z^{\prime
}\in K_{\alpha^{\prime}}$, and $\beta=\alpha\wedge\alpha^{\prime}\in
U^{-1}\mathcal{T}_{n}$, then $z$ and $z^{\prime}$ do not lie in a common child
$\gamma\in U^{-1}\mathcal{T}_{n}$ of $\beta$ (we may of course replace
\textquotedblleft child\textquotedblright\ by an \textquotedblleft$\ell$-fold
grandchild\textquotedblright\ with $\ell$ sufficiently large and fixed). From
the second inequality in (\ref{succset}), we then obtain
\begin{equation}
\left\vert 1-\overline{z^{\prime}}\cdot z\right\vert \geq c2^{-d\left(
\alpha\wedge\alpha^{\prime}\right)  },\;\;\;\;\;U\in\Sigma, \label{lowbound}%
\end{equation}
and combined with the second inequality in (\ref{numest}), (\ref{realcal}) now
yields (\ref{unittree}) upon integrating over Haar measure and using
$\left\vert \Sigma\right\vert \geq c>0$.

Now (\ref{Rebil}) is invariant under unitary transformations, and so
(\ref{below''}) for the tree $U^{-1}\mathcal{T}_{n}$ immediately shows that
(\ref{Rebil}) implies (\ref{done}) (note that we are throwing away the
constant lower bound of $c$ in (\ref{below''})).

Conversely, for $U\in\mathcal{U}_{n}$ let $f\left(  U^{-1}\alpha\right)
=\int_{U^{-1}K_{\alpha}}fd\lambda_{n}$ and $\nu\left(  U^{-1}\alpha\right)
=\int_{U^{-1}K_{\alpha}}d\nu$ be the function and measure discretizations of
$f$ and $\nu$ respectively on the rotated Bergman grid $\left\{
K_{U^{-1}\alpha}\right\}  _{\alpha\in\mathcal{T}_{n}}$. From (\ref{unittree})
and (\ref{rotate}) the left side of (\ref{Rebil}) with $f=g$ satisfies
\begin{align*}
&  \int_{\mathbb{B}_{n}}\int_{\mathbb{B}_{n}}\left(  \operatorname{Re}\frac
{1}{1-\overline{z}\cdot z^{\prime}}\right)  f\left(  z^{\prime}\right)
d\mu\left(  z^{\prime}\right)  f\left(  z\right)  d\mu\left(  z\right) \\
&  \;\;\;\;\;\leq C\int_{\mathcal{U}_{n}}\int_{\mathbb{B}_{n}}\int
_{\mathbb{B}_{n}}f\left(  z^{\prime}\right)  d\mu\left(  z^{\prime}\right)
f\left(  z\right)  d\mu\left(  z\right)  dU\\
&  \;\;\;\;\;\;\;\;\;\;+C\int_{\mathcal{U}_{n}}\int_{\mathbb{B}_{n}}%
\int_{\mathbb{B}_{n}}\frac{2^{2d\left(  \left\langle Uz\right\rangle
\wedge\left\langle Uz^{\prime}\right\rangle \right)  }}{2^{d^{*}\left(
\left[  z\right]  \wedge\left[  z^{\prime}\right]  \right)  }}f\left(
z^{\prime}\right)  d\mu\left(  z^{\prime}\right)  f\left(  z\right)
d\mu\left(  z\right)  dU\\
&  \;\;\;\;\;=I+II.
\end{align*}

Now $\mu$ is a finite measure and from Cauchy's inequality, we obtain that
\begin{equation}
I\leq C\left\|  f\right\|  _{L^{2}\left(  \mu\right)  }^{2}. \label{termI}%
\end{equation}
For each $U\in\mathcal{U}_{n}$, we decompose the ball $\mathbb{B}_{n}$ by the
rotated Bergman tree $U^{-1}\mathcal{T}_{n}$ to obtain
\[
II=C\int_{\mathcal{U}_{n}}\sum_{\alpha,\alpha^{\prime}\in\mathcal{T}_{n}}%
\int_{z\in K_{U^{-1}\alpha}}\int_{z^{\prime}\in K_{U^{-1}\alpha^{\prime}}%
}\frac{2^{2d\left(  \alpha\wedge\alpha^{\prime}\right)  }}{2^{d^{*}\left(
\left[  \alpha\right]  \wedge\left[  \alpha^{\prime}\right]  \right)  }%
}f\left(  z^{\prime}\right)  d\mu\left(  z^{\prime}\right)  f\left(  z\right)
d\mu\left(  z\right)  dU.
\]
Now let $\left(  fd\mu\right)  _{U}=\left(  fd\mu\right)  \circ U^{-1}$ for
each $U\in\mathcal{U}_{n}$, so that $f\left(  z^{\prime}\right)  d\mu\left(
z^{\prime}\right)  =\left(  fd\mu\right)  _{U}\left(  Uz^{\prime}\right)  $.
Then if we make the change of variable $w^{\prime}=Uz^{\prime}$ and $w=Uz$ in
the inner integrals above, $II$ becomes
\begin{align*}
&  C\int_{\mathcal{U}_{n}}\left\{  \sum_{\alpha,\alpha^{\prime}\in
\mathcal{T}_{n}}\int_{w\in K_{\alpha}}\int_{w^{\prime}\in K_{\alpha^{\prime}}%
}\frac{2^{2d\left(  \alpha\wedge\alpha^{\prime}\right)  }}{2^{d^{*}\left(
\left[  \alpha\right]  \wedge\left[  \alpha^{\prime}\right]  \right)  }%
}\left(  fd\mu\right)  _{U}\left(  w^{\prime}\right)  \left(  fd\mu\right)
_{U}\left(  w\right)  \right\}  dU\\
&  \;\;\;\;\;=C\int_{\mathcal{U}_{n}}\left\{  \sum_{\alpha,\alpha^{\prime}%
\in\mathcal{T}_{n}}\frac{2^{2d\left(  \alpha\wedge\alpha^{\prime}\right)  }%
}{2^{d^{*}\left(  \left[  \alpha\right]  \wedge\left[  \alpha^{\prime}\right]
\right)  }}\left(  fd\mu\right)  _{U}\left(  \alpha^{\prime}\right)  \left(
fd\mu\right)  _{U}\left(  \alpha\right)  \right\}  dU.
\end{align*}
Now we write
\[
\left(  fd\mu\right)  _{U}\left(  \alpha\right)  =\int_{U^{-1}K_{\alpha}}%
fd\mu=\left(  \frac{1}{\left|  U^{-1}K_{\alpha}\right|  _{\mu}}\int
_{U^{-1}K_{\alpha}}fd\mu\right)  \mu\left(  U^{-1}\alpha\right)  ,
\]
so that we obtain an estimate for $II$ from (\ref{done}) as follows:
\begin{align}
II  &  \leq C\int_{\mathcal{U}_{n}}\left\{  \sum_{\alpha,\alpha^{\prime}%
\in\mathcal{T}_{n}}\frac{2^{2d\left(  \alpha\wedge\alpha^{\prime}\right)  }%
}{2^{d^{*}\left(  \left[  \alpha\right]  \wedge\left[  \alpha^{\prime}\right]
\right)  }}\left(  fd\mu\right)  _{U}\left(  \alpha^{\prime}\right)  \left(
fd\mu\right)  _{U}\left(  \alpha\right)  \right\}  dU\label{termII}\\
&  \leq C\int_{\mathcal{U}_{n}}\left\{  \sum_{\alpha,\alpha^{\prime}%
\in\mathcal{T}_{n}}\frac{2^{2d\left(  \alpha\wedge\alpha^{\prime}\right)  }%
}{2^{d^{*}\left(  \left[  \alpha\right]  \wedge\left[  \alpha^{\prime}\right]
\right)  }}\left(  \frac{1}{\left|  U^{-1}K_{\alpha^{\prime}}\right|  _{\mu}%
}\int_{U^{-1}K_{\alpha^{\prime}}}fd\mu\right)  \mu\left(  U^{-1}\alpha
^{\prime}\right)  \right. \nonumber\\
&  \times\left.  \left(  \frac{1}{\left|  U^{-1}K_{\alpha}\right|  _{\mu}}%
\int_{U^{-1}K_{\alpha}}fd\mu\right)  \mu\left(  U^{-1}\alpha\right)  \right\}
dU\nonumber\\
&  \leq C\int_{\mathcal{U}_{n}}\left\{  \sum_{\alpha\in\mathcal{T}_{n}}\left(
\frac{1}{\left|  U^{-1}K_{\alpha}\right|  _{\mu}}\int_{U^{-1}K_{\alpha}}%
fd\mu\right)  ^{2}\mu\left(  U^{-1}\alpha\right)  \right\}  dU\nonumber\\
&  \leq C\int_{\mathcal{U}_{n}}\left\{  \sum_{\alpha\in\mathcal{T}_{n}}%
\int_{U^{-1}K_{\alpha}}f^{2}d\mu\right\}  dU=C\left\|  f\right\|
_{L^{2}\left(  \widetilde{\mu}\right)  }^{2}.\nonumber
\end{align}

Combining the estimates (\ref{termI}) and (\ref{termII}) for terms $I$ and
$II$, we thus obtain the bilinear inequality (\ref{Rebil}) when $f=g$, and
this suffices for the general inequality. This completes the proof of the
equivalence of (\ref{Rebil}) and (\ref{done}).

Now (\ref{done}) can be rewritten as
\begin{equation}
\sum_{\alpha\in\mathcal{T}_{n}}f\left(  \alpha\right)  \left\{  T_{\mu
}g\left(  \alpha\right)  \right\}  \mu\left(  \alpha\right)  \leq C\left\|
f\right\|  _{\ell^{2}\left(  \mu\right)  }\left\|  g\right\|  _{\ell
^{2}\left(  \mu\right)  }, \label{done''}%
\end{equation}
for all $f,g\geq0$ on $\mathcal{T}_{n}$, and where $T_{\mu}$ is given in
(\ref{deffracnew}):
\[
T_{\mu}g\left(  \alpha\right)  =\sum_{\alpha^{\prime}\in\mathcal{T}_{n}%
}2^{2d\left(  \alpha\wedge\alpha^{\prime}\right)  -d^{*}\left(  \left[
\alpha\right]  \wedge\left[  \alpha^{\prime}\right]  \right)  }g\left(
\alpha^{\prime}\right)  \mu\left(  \alpha^{\prime}\right)  .
\]
Upon using the Cauchy-Schwartz inequality and taking the supremum over all $f
$ with $\left\|  f\right\|  _{\ell^{2}\left(  \mu\right)  }=1$ in
(\ref{done''}), we obtain the equivalence of (\ref{done''}) and the discrete
inequality (\ref{discnew}), where $\mathcal{T}_{n}$ ranges over all unitary
rotations of a fixed Bergman tree.

\subsubsection{Carleson measures for $H_{n}^{2}$ and inequalities for positive
quantities}

Using Propositions \ref{realreduction} and \ref{discreduction}, we can
characterize Carleson measures for the Drury-Arveson Hardy space $H_{n}^{2}$
by either (\ref{Rebil}) or (\ref{discnew}). Recall that $\widehat{\mu}\left(
\alpha\right)  =\mu\left(  \alpha\right)  =\int_{K_{\alpha}}d\mu$ for
$\alpha\in\mathcal{T}_{n}$.

\begin{theorem}
\label{Achar}Let $\mu$ be a positive measure on the ball $\mathbb{B}_{n}$ with
$n$ finite. Then the following conditions are equivalent:

\begin{enumerate}
\item $\mu$ is a Carleson measure on the Drury-Arveson space $H_{n}^{2}$,

\item $\mu$ satisfies (\ref{Rebil}),

\item $\widehat{\mu}$ satisfies (\ref{discnew}) for all unitary rotations of a
fixed Bergman tree.
\end{enumerate}
\end{theorem}

In Proposition \ref{treereduction} of the next subsubsection, we will complete
the characterization of Carleson measures for the Drury-Arveson space by
giving necessary and sufficient conditions for (\ref{discnew}) taken over all
unitary rotations of a fixed Bergman tree, namely the split tree condition
(\ref{fullsplittreecondition}) and the simple condition (\ref{Arvsimple}),
both given below, taken over all unitary rotations of a fixed Bergman tree..
We record here the necessity of the simple condition.

\begin{lemma}
\label{simpnecc}If $\mu$ is a Carleson measure on the Drury-Arveson space
$H_{n}^{2}$, then $\mu$ satisfies the simple condition,
\begin{equation}
2^{d\left(  \alpha\right)  }I^{*}\mu\left(  \alpha\right)  \leq
C,\;\;\;\;\;\alpha\in\mathcal{T}_{n}. \label{Arvsimple}%
\end{equation}

\end{lemma}

Recall that $p\sigma=1$ and that $\theta=\frac{\ln2}{2}$ so that $1-\left|
w\right|  ^{2}\approx2^{-d\left(  \alpha\right)  }=2^{-p\sigma d\left(
\alpha\right)  }$ for $w\in K_{\alpha}$.%

\proof
(of the lemma) \ In fact the analogous statement holds for all $\sigma>0.$
Recall from Subsubsection \ref{IS} that $B_{2}^{\sigma}\left(  \mathbb{B}%
_{n}\right)  $ can be realized as $\mathcal{H}_{k}$ with kernel function
$k\left(  w,z\right)  =\left(  \frac{1}{1-\overline{w}\cdot z}\right)
^{2\sigma}$ on $\mathbb{B}_{n}.$ This function satisfies
\[
\left\Vert \left(  \frac{1}{1-\overline{w}\cdot z}\right)  ^{2\sigma
}\right\Vert _{B_{2}^{\sigma}\left(  \mathbb{B}_{n}\right)  }^{2}=\left(
\frac{1}{1-\left\vert w\right\vert ^{2}}\right)  ^{2\sigma}.
\]
Testing the Carleson embedding on these functions quickly leads to the desired estimates.

\medskip

Later we will use the fact that the condition SC(1/2) is sufficient to insure
the tree condition with $\sigma<1/2$. Rather than prove that in isolation we
take the opportunity to record two strengthenings of the condition
SC($\sigma)$ each of which is sufficient to imply the corresponding tree
condition (\ref{treecondo}). Either of the two suffices to establish that,
given any $\varepsilon>0,$ the condition SC($\sigma+\varepsilon)$ implies
(\ref{treecondo}).

For $\sigma>0.$ We will say that a measure $\mu$ satisfies the strengthened
simple condition if there is a \emph{summable} function $h(\cdot)$ such that
\begin{equation}
2^{2\sigma d\left(  \alpha\right)  }I^{\ast}\mu\left(  \alpha\right)  \leq
Ch\left(  d\left(  \alpha\right)  \right)  ,\;\;\;\;\;\alpha\in\mathcal{T}%
_{n}, \label{ssimp}%
\end{equation}
For $0<p<1$ we say that $\mu$ satisfies the $\ell^{p}$-simple condition if
\begin{equation}
2^{2\sigma d\left(  \alpha\right)  }\left(  \sum_{\beta\geq\alpha}\mu\left(
\beta\right)  ^{p}\right)  ^{\frac{1}{p}}\leq C,\;\;\;\;\;\alpha\in
\mathcal{T}_{n}. \label{psigma}%
\end{equation}
Note that the choices $h\equiv1$ and $p=1$ recapture the simple condition
SC$(\sigma)$: $2^{2\sigma d\left(  \alpha\right)  }I^{\ast}\mu\left(
\alpha\right)  \leq C$.

\begin{lemma}
\label{simptree}Let $\sigma>0$. If $\mu$ satisfies $\emph{either}$ the
strengthened simple condition (\ref{ssimp}), \emph{or} the $\ell^{p}$-simple
condition (\ref{psigma}) for some $0<p<1$, then $\mu$ satisfies the tree
condition (\ref{treecondo}).
\end{lemma}

For the particular case when $\mu$ is the interpolation measure associated
with a separated sequence of points in the unit disk the result about the
$\ell^{p}$-simple condition is Theorem 4 on page 38 of \cite{Seip}.%

\proof
The left side of (\ref{treecondo}) satisfies
\begin{align*}
\sum_{\gamma\geq\alpha}2^{2\sigma d\left(  \gamma\right)  }I^{\ast}\mu\left(
\gamma\right)  ^{2}  &  =\sum_{\delta,\delta^{\prime}\geq\alpha}\sum
_{\alpha\leq\gamma\leq\delta\wedge\delta^{\prime}}2^{2\sigma d\left(
\gamma\right)  }\mu\left(  \delta\right)  \mu\left(  \delta^{\prime}\right) \\
&  \leq C\sum_{\delta,\delta^{\prime}\geq\alpha}2^{2\sigma d\left(
\delta\wedge\delta^{\prime}\right)  }\mu\left(  \delta\right)  \mu\left(
\delta^{\prime}\right)  .
\end{align*}
If (\ref{ssimp}) holds, then we continue with
\begin{align*}
\sum_{\delta,\delta^{\prime}\geq\alpha}2^{2\sigma d\left(  \delta\wedge
\delta^{\prime}\right)  }\mu\left(  \delta\right)  \mu\left(  \delta^{\prime
}\right)   &  \leq\sum_{\delta\geq\alpha}\mu\left(  \delta\right)
\sum_{\alpha\leq\beta\leq\delta}2^{2\sigma d\left(  \beta\right)  }I^{\ast}%
\mu\left(  \beta\right) \\
&  \leq C\sum_{\delta\geq\alpha}\mu\left(  \delta\right)  \sum_{\alpha
\leq\beta\leq\delta}h\left(  d\left(  \beta\right)  \right) \\
&  \leq C\sum_{\delta\geq\alpha}\mu\left(  \delta\right)  =CI^{\ast}\mu\left(
\alpha\right)  ,
\end{align*}
which yields (\ref{treecondo}). On the other hand, if (\ref{psigma}) holds
with $0<p<1$, then we use
\[
\mu\left(  \delta\right)  \mu\left(  \delta^{\prime}\right)  \leq\mu\left(
\delta\right)  ^{2-p}\mu\left(  \delta^{\prime}\right)  ^{p}+\mu\left(
\delta^{\prime}\right)  ^{2-p}\mu\left(  \delta\right)  ^{p},
\]
together with the symmetry in $\delta$ and $\delta^{\prime}$, to continue
with
\begin{align*}
\sum_{\delta,\delta^{\prime}\geq\alpha}2^{2\sigma d\left(  \delta\wedge
\delta^{\prime}\right)  }\mu\left(  \delta\right)  ^{2-p}\mu\left(
\delta^{\prime}\right)  ^{p}  &  =\sum_{\delta\geq\alpha}\mu\left(
\delta\right)  ^{2-p}\sum_{\alpha\leq\beta\leq\delta}2^{2\sigma d\left(
\beta\right)  }\left(  \sum_{\delta^{\prime}\geq\beta}\mu\left(
\delta^{\prime}\right)  ^{p}\right) \\
&  \leq C\sum_{\delta\geq\alpha}\mu\left(  \delta\right)  ^{2-p}\sum
_{\alpha\leq\beta\leq\delta}2^{2\sigma\left(  1-p\right)  d\left(
\beta\right)  }\\
&  \leq C_{p}\sum_{\delta\geq\alpha}\mu\left(  \delta\right)  ^{2-p}%
2^{2\sigma\left(  1-p\right)  d\left(  \delta\right)  }\\
&  \leq C_{p}\sum_{\delta\geq\alpha}\mu\left(  \delta\right)  =C_{p}I^{\ast
}\mu\left(  \alpha\right)  ,
\end{align*}
which again yields (\ref{treecondo}). The final inequality here follows since
\[
\mu\left(  \delta\right)  ^{1-p}2^{2\sigma\left(  1-p\right)  d\left(
\delta\right)  }=\left(  \mu\left(  \delta\right)  2^{2\sigma d\left(
\delta\right)  }\right)  ^{1-p}\leq\left(  I^{\ast}\mu\left(  \delta\right)
2^{2\sigma d\left(  \delta\right)  }\right)  ^{1-p}\leq C
\]
by the usual simple condition, an obvious consequence of (\ref{psigma}) when
$0<p<1$.

The two conditions in the lemma are independent of each other. We offer
ingredients for the examples that show this but omit the details of the
verification. Suppose $\sigma=1/2$ and let $\mathcal{T}_{0}$ be the linear
tree. The measure $\mu\left(  \alpha\right)  =2^{-d\left(  \alpha\right)  }$
satisfies (\ref{psigma}) for any $p>0$ but fails (\ref{ssimp}) for any
summable $h.$ Now consider the binary tree $\mathcal{T}_{1}$. Set $\mu
_{N}\left(  \alpha\right)  =2^{-N}N^{-1}\left(  \log N\right)  ^{-2}.$ With
the choice $h\left(  n\right)  =n^{-1}(\log n)^{-2},n\geq2,\ $the measures
$\mu_{N}$ satisfy (\ref{ssimp}) uniformly in $N.$ However with the choice of
$\alpha=o$ the left side of (\ref{psigma}) is $2^{-N+N/p}N^{-1}(\log N)^{-2}$
which is unbounded in $N$ for any fixed $p<1$.

\subsubsection{The split tree condition\label{split}}

The bilinear inequality associated with (\ref{discnew}) is
\begin{align}
\sum_{\alpha\in\mathcal{T}_{n}}f\left(  \alpha\right)  T_{\mu}g\left(
\alpha\right)  \mu\left(  \alpha\right)   &  =\sum_{\alpha,\beta\in
\mathcal{T}_{n}}2^{2d\left(  \alpha\wedge\beta\right)  -d^{\ast}\left(
\left[  \alpha\right]  \wedge\left[  \beta\right]  \right)  }f\mu\left(
\alpha\right)  g\mu\left(  \beta\right) \label{biin}\\
&  \leq C\left(  \sum_{\alpha\in\mathcal{T}_{n}}f\left(  \alpha\right)
^{2}\mu\left(  \alpha\right)  \right)  ^{\frac{1}{2}}\left(  \sum_{\alpha
\in\mathcal{T}_{n}}g\left(  \alpha\right)  ^{2}\mu\left(  \alpha\right)
\right)  ^{\frac{1}{2}}.\nonumber
\end{align}
By Theorem \ref{Achar} and Lemma \ref{simpnecc}, the simple condition
(\ref{simp}) is necessary for (\ref{biin}). We now derive another necessary
condition for (\ref{biin}) to hold, namely the split tree condition
(\ref{splittree}). First we set $f=g=\chi_{S\left(  \eta\right)  }$ in
(\ref{biin}) to obtain
\[
\sum_{\alpha,\beta\geq\eta}2^{2d\left(  \alpha\wedge\beta\right)  -d^{\ast
}\left(  \left[  \alpha\right]  \wedge\left[  \beta\right]  \right)  }%
\mu\left(  \alpha\right)  \mu\left(  \beta\right)  \leq CI^{\ast}\mu\left(
\eta\right)  ,\;\;\;\;\;\eta\in\mathcal{T}_{n}.
\]
If we organize the sum on the left hand side by summing first over rings, we
obtain
\[
\sum_{A,B\in\mathcal{R}_{n}}\sum_{\substack{\alpha,\beta\geq\eta\\\alpha\in
A,\beta\in B}}\frac{2^{2d\left(  \alpha\wedge\beta\right)  }}{2^{d^{\ast
}\left(  A\wedge B\right)  }}\mu\left(  \alpha\right)  \mu\left(
\beta\right)  =\sum_{C\in\mathcal{R}_{n}}\sum_{\substack{A,B\in\mathcal{R}%
_{n}\\A\wedge B=C}}\sum_{\substack{\alpha,\beta\geq\eta\\\alpha\in A,\beta\in
B}}\frac{2^{2d\left(  \alpha\wedge\beta\right)  }}{2^{d^{\ast}\left(  A\wedge
B\right)  }}\mu\left(  \alpha\right)  \mu\left(  \beta\right)  .
\]
Now define $A\curlywedge B=C$ to mean the more restrictive condition that both
$A\wedge B=C$ and $d^{\ast}\left(  A\wedge B\right)  -d\left(  C\right)  $ is
bounded (thus requiring that $A\wedge B$ is not \textquotedblleft
artificially\textquotedblright\ too much closer to the root than it ought to
be due to the vagaries of the particular tree structure). We can then restrict
the sum over $A$ and $B$ above to $A\curlywedge B=C$ which permits
$2^{d^{\ast}\left(  A\wedge B\right)  }$ to be replaced by $2^{d\left(
C\right)  }$. The result is
\[
\sum_{C\in\mathcal{R}_{n}}\sum_{\substack{A,B\in\mathcal{R}_{n}\\A\curlywedge
B=C}}\sum_{\substack{\alpha,\beta\geq\eta\\\alpha\in A,\beta\in B}%
}\frac{2^{2d\left(  \alpha\wedge\beta\right)  }}{2^{d\left(  C\right)  }}%
\mu\left(  \alpha\right)  \mu\left(  \beta\right)  \leq CI^{\ast}\mu\left(
\eta\right)  ,\;\;\;\;\;\eta\in\mathcal{T}_{n}.
\]
This is the split tree condition on the Bergman tree $\mathcal{T}_{n}$, which
dominates the more transparent form
\[
\sum_{k\geq0}\sum_{\gamma\geq\eta}2^{d\left(  \gamma\right)  -k}%
\sum_{\substack{\left(  \delta,\delta^{\prime}\right)  \in\mathcal{G}^{\left(
k\right)  }\left(  \gamma\right)  }}I^{\ast}\mu\left(  \delta\right)  I^{\ast
}\mu\left(  \delta^{\prime}\right)  \leq CI^{\ast}\mu\left(  \eta\right)
,\;\;\;\;\;\eta\in\mathcal{T}_{n},
\]
with $\gamma=\alpha\wedge\beta$ and $k=d(c)-d(\gamma),$ where as in Definition
\ref{grand}, the set $\mathcal{G}^{\left(  k\right)  }\left(  \gamma\right)  $
consists of pairs $\left(  \delta,\delta^{\prime}\right)  $ of grand$^{k}%
$-children of $\gamma$ in $\mathcal{G}^{\left(  k\right)  }\left(
\gamma\right)  \times\mathcal{G}^{\left(  k\right)  }\left(  \gamma\right)  $
which satisfy $\delta\wedge\delta^{\prime}=\gamma$, $\left[  A^{2}%
\delta\right]  =\left[  A^{2}\delta^{\prime}\right]  $ (which implies
$d\left(  \left[  \delta\right]  ,\left[  \delta^{\prime}\right]  \right)
\leq4$) and $d^{\ast}\left(  \left[  \delta\right]  ,\left[  \delta^{\prime
}\right]  \right)  =4$. Note that $\mathcal{G}^{\left(  0\right)  }\left(
\gamma\right)  =\mathcal{G}\left(  \gamma\right)  $ is the set of
grandchildren of $\gamma$.

To show the sufficiency of the simple condition (\ref{simp}) and the split
tree condition (\ref{splittree}) taken over all unitary rotations of a fixed
Bergman tree, we begin by claiming that the left hand side of (\ref{biin})
satisfies
\begin{align}
&  \sum_{A,B\in\mathcal{R}_{n}}\sum_{\substack{\alpha\in A \\\beta\in B
}}\frac{2^{2d\left(  \alpha\wedge\beta\right)  }}{2^{d^{\ast}\left(  A\wedge
B\right)  }}f\mu\left(  \alpha\right)  g\mu\left(  \beta\right) \nonumber\\
&  \qquad\qquad\qquad\leq C\int\limits_{\mathcal{U}_{n}}\sum_{C\in
U^{-1}\mathcal{R}_{n}}\sum_{\substack{A,B\in U^{-1}\mathcal{R}_{n}
\\A\curlywedge B=C}}\sum_{\substack{\alpha\in A \\\beta\in B}}\frac
{2^{2d\left(  \alpha\wedge\beta\right)  }}{2^{d\left(  C\right)  }}f\mu\left(
\alpha\right)  g\mu\left(  \beta\right)  \text{ }dU,
\end{align}
To see this we note that from (\ref{numest}) and (\ref{inpart}), we have
\[
d^{\ast}\left(  \left[  z\right]  \wedge\left[  z^{\prime}\right]  \right)
\leq d\left(  \left[  Uz\right]  \wedge\left[  Uz^{\prime}\right]  \right)
+C,
\]
for $U\in\Sigma$ where $\left\vert \Sigma\right\vert \geq c>0$. Moreover, this
inequality persists in the following somewhat stronger form: for any fixed
rings $A,B$ associated to the tree $\mathcal{R}_{n}$, there is $\Sigma$ with
$\left\vert \Sigma\right\vert \geq c>0$ such that for any $U\in\Sigma$, if
$A^{\prime},B^{\prime}\in U^{-1}\mathcal{R}_{n}$ satisfy $A\cap A^{\prime}%
\neq\phi,B\cap B^{\prime}\neq\phi$, then $d\left(  A^{\prime}\wedge B^{\prime
}\right)  -d^{\ast}\left(  A\wedge B\right)  $ is bounded and hence
$A^{\prime}\curlywedge B^{\prime}=A^{\prime}\wedge B^{\prime}$. Thus
\begin{align*}
&  \sum_{\substack{\alpha\in A \\\beta\in B}}\frac{2^{2d\left(  \alpha
\wedge\beta\right)  }}{2^{d^{\ast}\left(  A\wedge B\right)  }}f\mu\left(
\alpha\right)  g\mu\left(  \beta\right) \\
&  \qquad\qquad\leq C\int_{\mathcal{U}_{n}}\sum_{\substack{C^{\prime}\in
U^{-1}\mathcal{R}_{n} \\A^{\prime}\curlywedge B^{\prime}=C^{\prime} \\A\cap
A^{\prime},B\cap B^{\prime}\neq\phi}}\sum_{\alpha\in A^{\prime},\beta\in
B^{\prime}}\frac{2^{2d\left(  \alpha\wedge\beta\right)  }}{2^{d\left(
C^{\prime}\right)  }}f\mu\left(  \alpha\right)  g\mu\left(  \beta\right)
\text{ }dU
\end{align*}
as required. Thus it will suffice to prove that (\ref{simp}) and the split
tree condition (\ref{splittree}) for the tree $U^{-1}\mathcal{T}_{n}$ imply
\begin{align}
&  \sum_{C\in U^{-1}\mathcal{R}_{n}}\sum_{\substack{A,B\in U^{-1}%
\mathcal{R}_{n} \\A\curlywedge B=C}}\sum_{\substack{\alpha\in A \\\beta\in
B}}\frac{2^{2d\left(  \alpha\wedge\beta\right)  }}{2^{d\left(  C\right)  }%
}f\mu\left(  \alpha\right)  g\mu\left(  \beta\right) \label{suffprove}\\
&  \qquad\qquad\qquad\qquad\leq C\left(  \sum_{\alpha\in\mathcal{T}_{n}%
}f\left(  \alpha\right)  ^{2}\mu\left(  \alpha\right)  \right)  ^{\frac{1}{2}%
}\left(  \sum_{\alpha\in\mathcal{T}_{n}}g\left(  \alpha\right)  ^{2}\mu\left(
\alpha\right)  \right)  ^{\frac{1}{2}},\nonumber
\end{align}
with a constant $C$ independent of $U\in\mathcal{U}_{n}$. Without loss of
generality we prove (\ref{suffprove}) when $U$ is the identity.

Define the projection $P_{C}$ from functions $h=\left\{  h\left(
\alpha\right)  \right\}  _{\alpha\in A}$ on the ring $A$ to functions $P_{C}h$
on the ring $C$ (provided $C\leq A$) by
\[
P_{C}h=\left\{  \sum_{\substack{\alpha\in A \\\alpha\geq\gamma}}h\left(
\alpha\right)  \right\}  _{\gamma\in C}.
\]
We also define the \textquotedblleft Poisson kernel\textquotedblright%
\ $\mathbb{P}_{C}$ at scale $C$ to be the mapping taking functions $h=\left\{
h\left(  \gamma^{\prime}\right)  \right\}  _{\gamma^{\prime}\in C}$ on $C$ to
functions $\mathbb{P}_{C}h=\left\{  \mathbb{P}_{C}h\left(  \gamma\right)
\right\}  _{\gamma\in C}$ on $C$ given by
\[
\mathbb{P}_{C}h=\left\{  \sum_{\gamma^{\prime}\in A}\frac{2^{2d\left(
\gamma\wedge\gamma^{\prime}\right)  }}{2^{d\left(  C\right)  }}h\left(
\gamma^{\prime}\right)  \right\}  _{\gamma\in C}.
\]
Now if $f_{A}$ denotes the restriction $\chi_{A}f$ of $f$ to the ring $A$, we
can write the left side of (\ref{suffprove}) as approximately
\begin{align*}
&  \sum_{C\in\mathcal{R}_{n}}\sum_{\substack{A,B\in\mathcal{R}_{n}
\\A\curlywedge B=C}}\sum_{\gamma\in C}\mathbb{P}_{C}\left(  P_{C}\left(
f_{A}\mu\right)  \right)  \left(  \gamma\right)  P_{C}\left(  g_{B}\mu\right)
\left(  \gamma\right) \\
&  \;\;\;\;\;=\sum_{C\in\mathcal{R}_{n}}\sum_{\substack{A,B\in\mathcal{R}_{n}
\\A\curlywedge B=C}}\left\langle \mathbb{P}_{C}\left(  P_{C}\left(  f_{A}%
\mu\right)  \right)  ,P_{C}\left(  g_{B}\mu\right)  \right\rangle _{C},
\end{align*}
where the inner product $\left\langle F,G\right\rangle _{C}$ is given by
$\sum_{\gamma\in C}F\left(  \gamma\right)  G\left(  \gamma\right)  $. At this
point we notice that the Poisson kernel
\[
\mathbb{P}_{C}\left(  \gamma,\gamma^{\prime}\right)  =\frac{2^{2d\left(
\gamma\wedge\gamma^{\prime}\right)  }}{2^{d\left(  C\right)  }}%
\]
is a geometric sum of averaging operators $\mathbb{A}_{C}^{k}$ with kernel
\[
\mathbb{A}_{C}^{k}\left(  \gamma,\gamma^{\prime}\right)  =2^{d\left(
C\right)  -k}\chi_{\left\{  d\left(  \gamma\wedge\gamma^{\prime}\right)
=d\left(  C\right)  -k\right\}  },
\]
namely
\begin{equation}
\mathbb{P}_{C}\left(  \gamma,\gamma^{\prime}\right)  =\sum_{k=0}^{d\left(
C\right)  }2^{-k}\mathbb{A}_{C}^{k}\left(  \gamma,\gamma^{\prime}\right)  .
\label{geosum}%
\end{equation}

We now consider the bilinear inequality with $\mathbb{P}_{C}$ replaced by
$\mathbb{A}_{C}^{0}$:
\begin{equation}
\sum_{C\in\mathcal{R}_{n}}\sum_{\substack{A,B\in\mathcal{R}_{n} \\A\curlywedge
B=C}}\left\langle \mathbb{A}_{C}^{0}\left(  P_{C}\left(  f_{A}\mu\right)
\right)  ,P_{C}\left(  g_{B}\mu\right)  \right\rangle _{C}\leq C\left\Vert
f\right\Vert _{\ell^{2}\left(  \mu\right)  }\left\Vert g\right\Vert _{\ell
^{2}\left(  \mu\right)  }. \label{kineq}%
\end{equation}
The left side of (\ref{kineq}) is
\begin{align*}
&  \sum_{C\in\mathcal{R}_{n}}2^{d\left(  C\right)  }\sum_{\substack{A,B\in
\mathcal{R}_{n} \\A\curlywedge B=C}}\left\langle P_{C}\left(  f_{A}\mu\right)
,P_{C}\left(  g_{B}\mu\right)  \right\rangle _{C}\\
&  \;\;\;\;\;=\sum_{C\in\mathcal{R}_{n}}2^{d\left(  C\right)  }\sum_{\gamma\in
C}\left\{  \sum_{\substack{A,B\in\mathcal{R}_{n} \\A\curlywedge B=C}}I^{\ast
}\left(  f_{A}\mu\right)  \left(  \gamma\right)  I^{\ast}\left(  g_{B}%
\mu\right)  \left(  \gamma\right)  \right\}  .
\end{align*}
For fixed $\gamma\in C$, we dominate the sum $\sum_{A,B\in\mathcal{R}%
_{n}:A\curlywedge B=C}$ in braces above by
\begin{align*}
&  \sum_{\substack{A,B\in\mathcal{R}_{n} \\A\curlywedge B=C}}I^{\ast}\left(
f_{A}\mu\right)  \left(  \gamma\right)  I^{\ast}\left(  g_{B}\mu\right)
\left(  \gamma\right)  \leq I^{\ast}\left(  f\mu\right)  \left(
\gamma\right)  \left(  g\mu\right)  \left(  \gamma\right) \\
&  \quad\quad\quad+\left(  f\mu\right)  \left(  \gamma\right)  I^{\ast}\left(
g\mu\right)  \left(  \gamma\right)  +\sum_{\substack{\delta,\delta^{\prime}%
\in\mathcal{G}\left(  \gamma\right)  \\d^{\ast}\left(  \left[  \delta\right]
,\left[  \delta^{\prime}\right]  \right)  =4}}I^{\ast}\left(  f\mu\right)
\left(  \delta\right)  I^{\ast}\left(  g\mu\right)  \left(  \delta^{\prime
}\right)  .
\end{align*}
The first two terms easily satisfy the bilinear inequality using only the
simple condition (\ref{Arvsimple}). Indeed,
\begin{align*}
\sum_{C\in\mathcal{R}_{n}}2^{d\left(  C\right)  }\sum_{\gamma\in C}I^{\ast
}\left(  f\mu\right)  \left(  \gamma\right)  \left(  g\mu\right)  \left(
\gamma\right)   &  =\sum_{\gamma\in\mathcal{T}_{n}}2^{d\left(  \gamma\right)
}I^{\ast}\left(  f\mu\right)  \left(  \gamma\right)  \left(  g\mu\right)
\left(  \gamma\right) \\
&  =\sum_{\gamma\in\mathcal{T}_{n}}I\left(  2^{d}f\mu\right)  \left(
\gamma\right)  g\left(  \gamma\right)  \mu\left(  \gamma\right) \\
&  \leq\left\Vert I\left(  2^{d}f\mu\right)  \right\Vert _{\ell^{2}\left(
\mu\right)  }\left\Vert g\right\Vert _{\ell^{2}\left(  \mu\right)  }.
\end{align*}
Now the inequality $\left\Vert I\left(  2^{d}f\mu\right)  \right\Vert
_{\ell^{2}\left(  \mu\right)  }\leq C\left\Vert f\right\Vert _{\ell^{2}\left(
\mu\right)  }$ can be rewritten
\[
\sum_{\gamma\in\mathcal{T}_{n}}Ih\left(  \gamma\right)  ^{2}\mu\left(
\gamma\right)  \leq C\sum_{\gamma\in\mathcal{T}_{n}}h\left(  \gamma\right)
^{2}2^{-2d\left(  \gamma\right)  }\mu\left(  \gamma\right)  ,
\]
which by Theorem \ref{Tars}, is equivalent to the condition
\[
\sum_{\gamma\geq\alpha}I^{\ast}\mu\left(  \gamma\right)  ^{2}2^{2d\left(
\gamma\right)  }\mu\left(  \gamma\right)  \leq CI^{\ast}\mu\left(
\alpha\right)  ,\;\;\;\;\;\alpha\in\mathcal{T}_{n},
\]
which in turn is trivially implied by the simple condition $I^{\ast}\mu\left(
\gamma\right)  ^{2}2^{2d\left(  \gamma\right)  }\leq C^{2}$.

It remains then to consider the \textquotedblleft split\textquotedblright%
\ bilinear inequality
\begin{equation}
\sum_{\gamma\in\mathcal{T}_{n}}2^{d\left(  \gamma\right)  }\sum
_{\substack{\delta,\delta^{\prime}\in\mathcal{G}\left(  \gamma\right)
\\d^{\ast}\left(  \left[  \delta\right]  ,\left[  \delta^{\prime}\right]
\right)  =4}}I^{\ast}\left(  f\mu\right)  \left(  \delta\right)  I^{\ast
}\left(  g\mu\right)  \left(  \delta^{\prime}\right)  \leq C\left\Vert
f\right\Vert _{\ell^{2}\left(  \mu\right)  }\left\Vert g\right\Vert _{\ell
^{2}\left(  \mu\right)  }, \label{finalineq}%
\end{equation}
or equivalently the corresponding quadratic inequality obtained by setting
$f=g$:
\begin{equation}
\sum_{\gamma\in\mathcal{T}_{n}}2^{d\left(  \gamma\right)  }\sum
_{\substack{\delta,\delta^{\prime}\in\mathcal{G}\left(  \gamma\right)
\\d^{\ast}\left(  \left[  \delta\right]  ,\left[  \delta^{\prime}\right]
\right)  =4}}I^{\ast}\left(  f\mu\right)  \left(  \delta\right)  I^{\ast
}\left(  f\mu\right)  \left(  \delta^{\prime}\right)  \leq C\sum_{\alpha
\in\mathcal{T}_{n}}f\left(  \alpha\right)  ^{2}\mu\left(  \alpha\right)  .
\label{finalineq'}%
\end{equation}
Note that the restriction to $k=0$ in the split tree condition
(\ref{splittree}) yields the following necessary condition for
(\ref{finalineq'}):
\begin{equation}
\sum_{\gamma\geq\alpha}2^{d\left(  \gamma\right)  }\sum_{\substack{\delta
,\delta^{\prime}\in\mathcal{G}\left(  \gamma\right)  \\d^{\ast}\left(  \left[
\delta\right]  ,\left[  \delta^{\prime}\right]  \right)  =4}}I^{\ast}%
\mu\left(  \delta\right)  I^{\ast}\mu\left(  \delta^{\prime}\right)  \leq
CI^{\ast}\mu\left(  \alpha\right)  ,\;\;\;\;\;\alpha\in\mathcal{T}_{n}.
\label{ArvCar}%
\end{equation}
We now show that (\ref{ArvCar}) and (\ref{Arvsimple}) together imply
(\ref{finalineq'}). To see this write the left side of (\ref{finalineq'}) as
\[
\sum_{\gamma\in\mathcal{T}_{n}}2^{d\left(  \gamma\right)  }\sum
_{\substack{\delta,\delta^{\prime}\in\mathcal{G}\left(  \gamma\right)
\\d^{\ast}\left(  \left[  \delta\right]  ,\left[  \delta^{\prime}\right]
\right)  =4}}I^{\ast}\mu\left(  \delta\right)  I^{\ast}\mu\left(
\delta^{\prime}\right)  \frac{I^{\ast}\left(  f\mu\right)  \left(
\delta\right)  I^{\ast}\left(  f\mu\right)  \left(  \delta^{\prime}\right)
}{I^{\ast}\mu\left(  \delta\right)  I^{\ast}\mu\left(  \delta^{\prime}\right)
},
\]
and using the symmetry in $\delta,\delta^{\prime}$ we bound it by
\begin{align*}
&  \sum_{\gamma\in\mathcal{T}_{n}}2^{d\left(  \gamma\right)  }\sum
_{\substack{\delta,\delta^{\prime}\in\mathcal{G}\left(  \gamma\right)
\\d^{\ast}\left(  \left[  \delta\right]  ,\left[  \delta^{\prime}\right]
\right)  =4}}I^{\ast}\mu\left(  \delta\right)  I^{\ast}\mu\left(
\delta^{\prime}\right)  \left(  \frac{I^{\ast}\left(  f\mu\right)  \left(
\delta\right)  }{I^{\ast}\mu\left(  \delta\right)  }\right)  ^{2}\\
&  \;\;\;\;\;=\sum_{\delta\in\mathcal{T}_{n}}\left(  \frac{I^{\ast}\left(
f\mu\right)  \left(  \delta\right)  }{I^{\ast}\mu\left(  \delta\right)
}\right)  ^{2}\sum_{\substack{\delta^{\prime}\in\mathcal{G}\left(  A^{2}%
\delta\right)  \\d^{\ast}\left(  \left[  \delta\right]  ,\left[
\delta^{\prime}\right]  \right)  =4}}2^{d\left(  A^{2}\delta\right)  }I^{\ast
}\mu\left(  \delta\right)  I^{\ast}\mu\left(  \delta^{\prime}\right)  .
\end{align*}
By Theorem \ref{unity}, this last term is dominated by the right side of
(\ref{finalineq'}) provided $I^{\ast}\sigma\left(  \alpha\right)  \leq
CI^{\ast}\mu\left(  \alpha\right)  $ for all $\alpha\in\mathcal{T}_{n}$ where
$\sigma\left(  \delta\right)  $ is given by
\[
\sigma\left(  \delta\right)  =\sum_{\substack{\delta^{\prime}\in
\mathcal{G}\left(  A^{2}\delta\right)  \\d^{\ast}\left(  \left[
\delta\right]  ,\left[  \delta^{\prime}\right]  \right)  =4}}2^{d\left(
A^{2}\delta\right)  }I^{\ast}\mu\left(  \delta\right)  I^{\ast}\mu\left(
\delta^{\prime}\right)  .
\]
This latter condition can be expressed as
\begin{align}
&  \sum_{\gamma\geq\alpha}2^{d\left(  \gamma\right)  }\sum_{\substack{\delta
,\delta^{\prime}\in\mathcal{G}\left(  \gamma\right)  \\d^{\ast}\left(  \left[
\delta\right]  ,\left[  \delta^{\prime}\right]  \right)  =4}}I^{\ast}%
\mu\left(  \delta\right)  I^{\ast}\mu\left(  \delta^{\prime}\right)
+2^{d\left(  A\alpha\right)  }\sum_{\substack{\delta^{\prime}\in
\mathcal{G}\left(  A\delta\right)  \\d^{\ast}\left(  \left[  \delta\right]
,\left[  \delta^{\prime}\right]  \right)  =4}}I^{\ast}\mu\left(
A\alpha\right)  I^{\ast}\mu\left(  \delta^{\prime}\right) \label{needy}\\
&  \qquad+2^{d\left(  A^{2}\alpha\right)  }\sum_{\substack{\delta^{\prime}%
\in\mathcal{G}\left(  A^{2}\delta\right)  \\d^{\ast}\left(  \left[
\delta\right]  ,\left[  \delta^{\prime}\right]  \right)  =4}}I^{\ast}%
\mu\left(  A^{2}\alpha\right)  I^{\ast}\mu\left(  \delta^{\prime}\right)  \leq
CI^{\ast}\mu\left(  \alpha\right)  ,\;\;\;\;\;\alpha\in\mathcal{T}%
_{n}.\nonumber
\end{align}
Now the necessary condition (\ref{ArvCar}) shows that the first sum in
(\ref{needy}) is at most $CI^{\ast}\mu\left(  \alpha\right)  $, while the
simple condition (\ref{Arvsimple}) yields $2^{d\left(  A\alpha\right)
}I^{\ast}\mu\left(  \delta^{\prime}\right)  \leq C$, which shows that the
second sum in (\ref{needy}) is at most $CI^{\ast}\mu\left(  \alpha\right)  $.
The third sum is handled similarly and this completes the proof that
(\ref{kineq}) holds when both (\ref{ArvCar}) and (\ref{Arvsimple}) hold.

To handle the averaging operators $\mathbb{A}_{C}^{k}$ for $k>0$, we compute
that for $D\in\mathcal{R}_{n}$,
\begin{align*}
&  \sum_{C\in\mathcal{C}^{\left(  k-1\right)  }\left(  D\right)  }%
\sum_{\substack{A,B\in\mathcal{R}_{n}\\A\curlywedge B=C}}\left\langle
\mathbb{A}_{C}^{k}\left(  P_{C}\left(  f_{A}\mu\right)  \right)  ,P_{C}\left(
g_{B}\mu\right)  \right\rangle _{C}\\
&  \quad\quad\quad=\sum_{C\in\mathcal{C}^{\left(  k-1\right)  }\left(
D\right)  }\sum_{\substack{A,B\in\mathcal{R}_{n}\\A\curlywedge B=C}%
}2^{d\left(  C\right)  -k}\sum_{\substack{\gamma,\gamma^{\prime}\in
C\\d\left(  \gamma\wedge\gamma^{\prime}\right)  =d\left(  D\right)  }%
}P_{C}\left(  f_{A}\mu\right)  \left(  \gamma\right)  P_{C}\left(  g_{B}%
\mu\right)  \left(  \gamma^{\prime}\right) \\
&  \quad\quad\quad=\sum_{C\in\mathcal{C}^{\left(  k-1\right)  }\left(
D\right)  }2^{d\left(  C\right)  -k}\left(  \sum_{\substack{\delta
,\delta^{\prime}\in D\\d\left(  \delta,\delta^{\prime}\right)  =2}%
}\sum_{\substack{\gamma,\gamma^{\prime}\in C\\\gamma\geq\delta\\\gamma
^{\prime}\geq\delta^{\prime}}}\sum_{\substack{A,B\in\mathcal{R}_{n}%
\\A\curlywedge B=C}}P_{C}\left(  f_{A}\mu\right)  \left(  \gamma\right)
P_{C}\left(  g_{B}\mu\right)  \left(  \gamma^{\prime}\right)  \right)  .
\end{align*}
Summing this over all rings $D\in\mathcal{R}_{n}$, and then summing in $k>0$,
we obtain
\begin{align*}
&  \sum_{C\in\mathcal{R}_{n}}\sum_{A,B\in\mathcal{R}_{n}:A\curlywedge
B=C}\left\langle \sum_{k>0}2^{-k}\mathbb{A}_{C}^{k}\left(  P_{C}\left(
f_{A}\mu\right)  \right)  ,P_{C}\left(  g_{B}\mu\right)  \right\rangle _{C}\\
&  \qquad\qquad=\sum_{k>0}2^{-k}\sum_{D\in\mathcal{R}_{n}}\sum_{C\in
\mathcal{C}^{\left(  k-1\right)  }\left(  D\right)  }\sum_{\substack{A,B\in
\mathcal{R}_{n}\\A\curlywedge B=C}}\left\langle \mathbb{A}_{C}^{k}\left(
P_{C}\left(  f_{A}\mu\right)  \right)  ,P_{C}\left(  g_{B}\mu\right)
\right\rangle _{C}\\
&  \qquad\qquad=\sum_{k>0}\sum_{D\in\mathcal{R}_{n}}\sum_{C\in\mathcal{C}%
^{\left(  k-1\right)  }\left(  D\right)  }2^{d\left(  C\right)  -2k}\left(
\ast\right)  .
\end{align*}
The term $(\ast)$ is
\[
(\ast)=\sum_{\substack{\delta,\delta^{\prime}\in D\\d\left(  \delta
,\delta^{\prime}\right)  =2}}\sum_{\substack{\gamma,\gamma^{\prime}\in
C\\\gamma\geq\delta,\gamma^{\prime}\geq\delta^{\prime}}}\sum_{\substack{A,B\in
\mathcal{R}_{n}\\A\curlywedge B=C}}P_{C}\left(  f_{A}\mu\right)  \left(
\gamma\right)  P_{C}\left(  g_{B}\mu\right)  \left(  \gamma^{\prime}\right)
\]
and it satisfies $\left(  \ast\right)  \leq I+II+III$ with
\[
I=\sum_{\substack{\delta,\delta^{\prime}\in D\\d\left(  \delta,\delta^{\prime
}\right)  =2}}\sum_{\substack{\gamma,\gamma^{\prime}\in C\\\gamma\geq
\delta,\gamma^{\prime}\geq\delta^{\prime}}}\left(  f\mu\right)  \left(
\gamma\right)  \left(  \sum_{B\geq C}P_{C}\left(  g_{B}\mu\right)  \left(
\gamma^{\prime}\right)  \right)
\]%
\[
II=\sum_{\substack{\delta,\delta^{\prime}\in D\\d\left(  \delta,\delta
^{\prime}\right)  =2}}\sum_{\substack{\gamma,\gamma^{\prime}\in C\\\gamma
\geq\delta,\gamma^{\prime}\geq\delta^{\prime}}}\left(  \sum_{A\geq C}%
P_{C}\left(  f_{A}\mu\right)  \left(  \gamma\right)  \right)  \left(
g\mu\right)  \left(  \gamma^{\prime}\right)
\]%
\[
III=\sum_{\substack{\delta,\delta^{\prime}\in D\\d\left(  \delta
,\delta^{\prime}\right)  =2}}\sum_{\substack{\gamma,\gamma^{\prime}\in
C\\\gamma\geq\delta,\gamma^{\prime}\geq\delta^{\prime}}}\sum_{\substack{A,B\in
\mathcal{R}_{n}\\A\curlywedge B=C}}P_{C}\left(  f_{A}\mu\right)  \left(
\gamma\right)  P_{C}\left(  g_{B}\mu\right)  \left(  \gamma^{\prime}\right)
.
\]
We now analyze these sums in terms of the operator $I^{\ast}$. The first two,
$I$ and $II$, are are similar to each other and can be controlled by the
simple condition (\ref{Arvsimple}) alone. Indeed, $\sum_{B\geq C}P_{C}\left(
g_{B}\mu\right)  \left(  \gamma^{\prime}\right)  =I^{\ast}\mu\left(
\gamma^{\prime}\right)  $ and
\[
\sum_{k>0}\sum_{D\in\mathcal{R}_{n}}\left\{  \sum_{C\in\mathcal{C}^{\left(
k-1\right)  }\left(  D\right)  }2^{d\left(  C\right)  -2k}\left(
\sum_{\substack{\delta,\delta^{\prime}\in D\\d\left(  \delta,\delta^{\prime
}\right)  =2}}\sum_{\substack{\gamma,\gamma^{\prime}\in C\\\gamma\geq
\delta,\gamma^{\prime}\geq\delta^{\prime}}}\left(  f\mu\right)  \left(
\gamma\right)  I^{\ast}\mu\left(  \gamma^{\prime}\right)  \right)  \right\}
\]
has bilinear kernel function $K\left(  \gamma,\beta\right)  =2^{d\left(
\gamma\right)  -2k}$ where $k=d\left(  \gamma\right)  -d\left(  \gamma
\wedge\gamma^{\prime}\right)  $ and $\gamma^{\prime}$ is the unique element of
the ring $C=\left[  \gamma\right]  $ with $\beta\geq\gamma^{\prime}$. Since
$d\left(  \gamma\wedge\beta\right)  =d\left(  \gamma\wedge\gamma^{\prime
}\right)  $, we thus have
\[
K\left(  \gamma,\beta\right)  =2^{d\left(  \gamma\right)  -2\left(  d\left(
\gamma\right)  -d\left(  \gamma\wedge\gamma^{\prime}\right)  \right)
}=2^{2d\left(  \gamma\wedge\beta\right)  -\min\left\{  d\left(  \gamma\right)
,d\left(  \beta\right)  \right\}  },
\]
and the case $r=1$ of Theorem \ref{simplesuffices} below shows that this
kernel is controlled by the simple condition. We now turn to $III$. Using
$\left(  +\right)  $ to denote a set of summation indices:%
\[
\left(  +\right)  =\left\{  \eta,\eta^{\prime}:A^{2}\eta,A^{2}\eta^{\prime}\in
C,d\left(  \eta\wedge\eta^{\prime}\right)  =d\left(  D\right)  +1,d^{\ast
}\left(  \left[  \eta\right]  \wedge\left[  \eta^{\prime}\right]  \right)
=4\right\}
\]
we can rewrite $III$ as
\[
III=\sum_{\left(  +\right)  }I^{\ast}\left(  f\mu\right)  \left(  \eta\right)
I^{\ast}\left(  g\mu\right)  \left(  \eta^{\prime}\right)  .
\]
Setting $f=g$ we see we must show that
\begin{equation}
\sum_{k>0}\sum_{D\in\mathcal{R}_{n}}\sum_{C\in\mathcal{C}^{\left(  k-1\right)
}\left(  D\right)  }2^{d\left(  C\right)  -2k}\sum_{\left(  +\right)  }%
I^{\ast}\left(  f\mu\right)  \left(  \eta\right)  I^{\ast}\left(  f\mu\right)
\left(  \eta^{\prime}\right)  \leq C\left\Vert f\right\Vert _{\ell^{2}\left(
\mu\right)  }^{2}. \label{mustshow'}%
\end{equation}
Just as in handling the bilinear inequality for $\mathbb{A}_{C}^{0}$ above, we
exploit the symmetry in $\eta,\eta^{\prime}$ to obtain
\begin{align*}
&  \sum_{k>0}\sum_{D\in\mathcal{R}_{n}}\sum_{C\in\mathcal{C}^{\left(
k-1\right)  }\left(  D\right)  }2^{d\left(  C\right)  -2k}\sum_{\left(
+\right)  }I^{\ast}\left(  f\mu\right)  \left(  \eta\right)  I^{\ast}\left(
f\mu\right)  \left(  \eta^{\prime}\right) \\
&  \;\;\;\;\;=\sum_{k>0}\sum_{D\in\mathcal{R}_{n}}\sum_{C\in\mathcal{C}%
^{\left(  k-1\right)  }\left(  D\right)  }2^{d\left(  C\right)  -2k}%
\sum_{\left(  +\right)  }\left[  I^{\ast}\mu\left(  \eta\right)  I^{\ast}%
\mu\left(  \eta^{\prime}\right)  \right]  \frac{I^{\ast}\left(  f\mu\right)
\left(  \eta\right)  I^{\ast}\left(  f\mu\right)  \left(  \eta^{\prime
}\right)  }{I^{\ast}\mu\left(  \eta\right)  I^{\ast}\mu\left(  \eta^{\prime
}\right)  }\\
&  \;\;\;\;\;\leq\sum_{k>0}\sum_{D\in\mathcal{R}_{n}}\sum_{C\in\mathcal{C}%
^{\left(  k-1\right)  }\left(  D\right)  }2^{d\left(  C\right)  -2k}%
\sum_{\left(  +\right)  }I^{\ast}\mu\left(  \eta\right)  I^{\ast}\mu\left(
\eta^{\prime}\right)  \left(  \frac{I^{\ast}\left(  f\mu\right)  \left(
\eta\right)  }{I^{\ast}\mu\left(  \eta\right)  }\right)  ^{2}.
\end{align*}
Now we apply Theorem \ref{unity} to obtain that the last expression above is
dominated by the right side of (\ref{mustshow'}) provided we have the
condition, for $\alpha\in\mathcal{T}_{n}.$
\begin{equation}
\sum_{\eta:\eta\geq\alpha}\left\{  \sum_{k>0}\sum_{D\in\mathcal{R}_{n}}%
\sum_{C\in\mathcal{C}^{\left(  k-1\right)  }\left(  D\right)  }2^{d\left(
C\right)  -2k}\sum_{\left(  +\right)  }I^{\ast}\mu\left(  \eta\right)
I^{\ast}\mu\left(  \eta^{\prime}\right)  \right\}  \leq CI^{\ast}\mu\left(
\alpha\right)  \label{needyfull}%
\end{equation}

As before, this condition is implied by the simple condition (\ref{Arvsimple})
together with the restriction to $k>0$ in the split tree condition
(\ref{splittree}):
\begin{equation}
\sum_{k>0}\sum_{\gamma\geq\alpha}2^{d\left(  \gamma\right)  -k}\sum
_{\substack{\eta,\eta^{\prime}\in\mathcal{G}^{\left(  k\right)  }\left(
\gamma\right)  }}I^{\ast}\mu\left(  \eta\right)  I^{\ast}\mu\left(
\eta^{\prime}\right)  \leq CI^{\ast}\mu\left(  \alpha\right)
,\;\;\;\;\;\alpha\in\mathcal{T}_{n}, \label{ArvCar''}%
\end{equation}
where we recall the notation from Definition \ref{grand},
\[
\mathcal{G}^{\left(  k\right)  }\left(  \gamma\right)  =\left\{  \left(
\eta,\eta^{\prime}\right)  \in\mathcal{G}^{\left(  k\right)  }\left(
\gamma\right)  \times\mathcal{G}^{\left(  k\right)  }\left(  \gamma\right)  :%
\genfrac{.}{.}{0pt}{}{\eta\wedge\eta^{\prime}=\gamma,\left[  A^{2}\eta\right]
=\left[  A^{2}\eta^{\prime}\right]  }{d^{\ast}\left(  \left[  \eta\right]
,\left[  \eta^{\prime}\right]  \right)  \geq2}%
\right\}  .
\]
We now show that (\ref{needyfull}) is implied by the simple condition
(\ref{Arvsimple}) together with (\ref{ArvCar''}). The proof is analogous to
the argument used to establish that (\ref{Arvsimple}) and (\ref{ArvCar}) imply
(\ref{needy}) above. We rewrite the left side of (\ref{needyfull}) as
\[
\sum_{k>0}\sum_{\gamma\geq\alpha}2^{d\left(  \gamma\right)  -k}\sum
_{\substack{\left(  \eta,\eta^{\prime}\right)  \in\mathcal{G}^{\left(
k\right)  }\left(  \gamma\right)  }}I^{\ast}\mu\left(  \eta\right)  I^{\ast
}\mu\left(  \eta^{\prime}\right)  +\mathsf{REST}.
\]
Now the terms in $\mathsf{REST}$ that have $\eta=\alpha$ are dominated by
\begin{align*}
&  \sum_{k>0}2^{d\left(  \alpha\right)  -2k}I^{\ast}\mu\left(  \alpha\right)
\sum_{\substack{\eta^{\prime}\in\left[  \eta\right]  \\d\left(  \eta\wedge
\eta^{\prime}\right)  =d\left(  \alpha\right)  -k}}I^{\ast}\mu\left(
\eta^{\prime}\right) \\
&  \qquad\qquad\qquad\leq\sum_{k>0}2^{d\left(  \alpha\right)  -2k}I^{\ast}%
\mu\left(  \alpha\right)  \sum_{\substack{\eta^{\prime}\in\left[  \eta\right]
\\d\left(  \eta\wedge\eta^{\prime}\right)  =d\left(  \alpha\right)
-k}}C2^{-d\left(  \alpha\right)  }\\
&  \qquad\qquad\qquad\leq C\sum_{k>0}2^{-k}I^{\ast}\mu\left(  \alpha\right)
\leq CI^{\ast}\mu\left(  \alpha\right)  ,
\end{align*}
as required. However, we must also sum over the terms having simultaneously
$\eta>\alpha$ and not $\eta^{\prime}>\alpha$. We organize this sum by summing
over the pairs $\left(  \eta,\eta^{\prime}\right)  \in\mathcal{G}^{\left(
k\right)  }\left(  \gamma\right)  $ for which $\eta\wedge\eta^{\prime}$ equals
a given $\gamma\in\left[  o,A\alpha\right]  $, and then splitting this sum
over those $\eta\in\mathcal{C}^{\left(  \ell\right)  }\left(  \alpha\right)
$, $\ell>0$, so that $d\left(  \alpha\right)  +\ell=d\left(  \eta\right)
=d\left(  \gamma\right)  +k$, and obtain the following:
\begin{align*}
&  \sum_{\gamma<\alpha}\sum_{\ell>0}\sum_{\eta\in\mathcal{C}^{\left(
\ell\right)  }\left(  \alpha\right)  }\sum_{\substack{\eta^{\prime}%
:d(\eta^{\prime})=d\left(  \eta\right)  \\\eta\wedge\eta^{\prime}=\gamma
}}2^{d\left(  \eta\right)  -2\left[  d\left(  \alpha\right)  +\ell-d\left(
\gamma\right)  \right]  }I^{\ast}\mu\left(  \eta\right)  I^{\ast}\mu\left(
\eta^{\prime}\right) \\
&  =\sum_{\gamma<\alpha}2^{2d\left(  \gamma\right)  -d\left(  \alpha\right)
}\sum_{\ell>0}2^{-\ell}\sum_{\eta\in\mathcal{C}^{\left(  \ell\right)  }\left(
\alpha\right)  }\sum_{\substack{\eta^{\prime}:d(\eta^{\prime})=d\left(
\eta\right)  \\\eta\wedge\eta^{\prime}=\gamma}}I^{\ast}\mu\left(  \eta\right)
I^{\ast}\mu\left(  \eta^{\prime}\right) \\
&  \leq\sum_{\gamma<\alpha}2^{2d\left(  \gamma\right)  -d\left(
\alpha\right)  }\sum_{\ell>0}2^{-\ell}I^{\ast}\mu\left(  \alpha\right)
I^{\ast}\mu\left(  \gamma\right) \\
&  \leq\left\{  C\sum_{\gamma<\alpha}2^{d\left(  \gamma\right)  -d\left(
\alpha\right)  }\right\}  I^{\ast}\mu\left(  \alpha\right)  =CI^{\ast}%
\mu\left(  \alpha\right)  .
\end{align*}

Combining the above with (\ref{geosum}) and (\ref{kineq}) we obtain
\[
\sum_{C\in\mathcal{R}_{n}}\sum_{\substack{A,B\in\mathcal{R}_{n}\\A\curlywedge
B=C}}\left\langle \mathbb{P}_{C}\left(  P_{C}\left(  f_{A}\mu\right)  \right)
,P_{C}\left(  g_{B}\mu\right)  \right\rangle _{C}\leq C\left\Vert f\right\Vert
_{\ell^{2}\left(  \mu\right)  }\left\Vert g\right\Vert _{\ell^{2}\left(
\mu\right)  },
\]
and hence (\ref{biin}), provided that (\ref{Arvsimple}), (\ref{ArvCar}) and
(\ref{ArvCar''}) all hold. We have thus obtained the following
characterization of (\ref{discnew}) taken over all unitary rotations of a
fixed Bergman tree.

\begin{proposition}
\label{treereduction}A positive measure $\mu$ on $\mathbb{B}_{n}$ satisfies
(\ref{discnew}), where $\mathcal{T}_{n}$ ranges over all unitary rotations of
a fixed Bergman tree, if and only if $\mu$ satisfies the simple condition
(\ref{Arvsimple}) and the following split tree condition,
\begin{equation}
\sum_{k\geq0}\sum_{\gamma\geq\alpha}2^{d\left(  \gamma\right)  -k}%
\sum_{\substack{\left(  \delta,\delta^{\prime}\right)  \in\mathcal{G}^{\left(
k\right)  }\left(  \gamma\right)  }}I^{\ast}\mu\left(  \delta\right)  I^{\ast
}\mu\left(  \delta^{\prime}\right)  \leq CI^{\ast}\mu\left(  \alpha\right)
,\;\;\;\;\;\alpha\in\mathcal{T}_{n}, \label{fullsplittreecondition}%
\end{equation}
taken over all unitary rotations of the same Bergman tree, and where
$\mathcal{G}^{\left(  k\right)  }$ is given by Definition \ref{grand}.
Moreover,
\begin{align*}
c_{n}\sup_{\mathcal{T}_{n}}\left\Vert \mu\right\Vert _{Carleson\left(
\mathcal{T}_{n}\right)  }  &  \leq\sup_{\alpha\in\mathcal{T}_{n}}%
\sqrt{2^{d\left(  \alpha\right)  }I^{\ast}\mu\left(  \alpha\right)  }\\
&  +\sup_{\substack{\alpha\in\mathcal{T}_{n}\\I^{\ast}\mu\left(
\alpha\right)  >0}}\sqrt{\frac{1}{I^{\ast}\mu\left(  \alpha\right)  }%
\sum_{k\geq0}\sum_{\gamma\geq\alpha}2^{d\left(  \gamma\right)  -k}%
\sum_{\substack{\left(  \delta,\delta^{\prime}\right)  \in\mathcal{G}^{\left(
k\right)  }\left(  \gamma\right)  }}I^{\ast}\mu\left(  \delta\right)  I^{\ast
}\mu\left(  \delta^{\prime}\right)  }\\
&  \leq C_{n}\sup_{\mathcal{T}_{n}}\left\Vert \mu\right\Vert _{Carleson\left(
\mathcal{T}_{n}\right)  },
\end{align*}
where the supremum is taken over all $\alpha\in\mathcal{T}_{n}$ and
$\mathcal{T}_{n}$ ranges over all unitary rotations of a fixed Bergman tree.
\end{proposition}

We note that Theorems \ref{unity} and \ref{simplesuffices} and Lemma
\ref{simpnecc} are independent of dimension, but the argument given above to
establish the equivalence of (\ref{discnew}) with
(\ref{fullsplittreecondition}) and (\ref{Arvsimple}) does depend on dimension.

Combining the three propositions above, we obtain the following
characterization of Carleson measures for the Drury-Arveson space.

\begin{theorem}
\label{Arvcom}A positive measure $\mu$ on the ball $\mathbb{B}_{n}$ is
$H_{n}^{2}$-Carleson if and only if $\mu$ satisfies the simple condition
(\ref{Arvsimple}) and the split tree condition (\ref{fullsplittreecondition})
taken over all unitary rotations of a fixed Bergman tree. Moreover, we have
\begin{align*}
c_{n}\left\Vert \mu\right\Vert _{Carleson}  &  \leq\sup_{\alpha\in
\mathcal{T}_{n}}\sqrt{2^{d\left(  \alpha\right)  }I^{\ast}\mu\left(
\alpha\right)  }\\
&  +\sup_{\substack{\alpha\in\mathcal{T}_{n} \\I^{\ast}\mu\left(
\alpha\right)  >0}}\sqrt{\frac{1}{I^{\ast}\mu\left(  \alpha\right)  }%
\sum_{k\geq0}\sum_{\gamma\geq\alpha}2^{d\left(  \gamma\right)  -k}%
\sum_{\substack{\left(  \delta,\delta^{\prime}\right)  \in\mathcal{G}^{\left(
k\right)  }\left(  \gamma\right)  }}I^{\ast}\mu\left(  \delta\right)  I^{\ast
}\mu\left(  \delta^{\prime}\right)  }\\
&  \leq C_{n}\left\Vert \mu\right\Vert _{Carleson},
\end{align*}
where the supremum is also taken over all unitary rotations $\mathcal{T}_{n}$
of a fixed Bergman tree.
\end{theorem}

\begin{remark}
We can recast the above characterization on the ball as follows. For
$w\in\mathbb{B}_{n}$ let $T\left(  w\right)  $ be the Carleson tent associated
to $w,$
\[
T\left(  w\right)  =\left\{  z\in\mathbb{B}_{n}:\left\vert 1-\overline{z}\cdot
Pw\right\vert \leq1-\left\vert w\right\vert \right\}  ,
\]
and $Pw$ denotes radial projection of $w$ onto the sphere $\partial
\mathbb{B}_{n}$.The $H_{n}^{2}$-Carleson norm of a positive measure $\mu$ on
$\mathbb{B}_{n}$satisfies
\begin{align*}
c_{n}\left\Vert \mu\right\Vert _{Carleson}  &  \leq\sup_{w\in\mathbb{B}_{n}%
}\sqrt{\left(  1-\left\vert w\right\vert ^{2}\right)  ^{-1}\mu\left(  T\left(
w\right)  \right)  }\\
&  +\sup_{\substack{w\in\mathbb{B}_{n} \\\mu\left(  T\left(  w\right)
\right)  >0}}\sqrt{\frac{1}{\mu\left(  T\left(  w\right)  \right)  }%
\int_{T\left(  w\right)  }\left\vert \int_{T\left(  w\right)  }\left(
\operatorname{Re}\frac{1}{1-\overline{z}\cdot z^{\prime}}\right)  d\mu\left(
z^{\prime}\right)  \right\vert ^{2}d\mu\left(  z\right)  }\\
&  \leq C_{n}\left\Vert \mu\right\Vert _{Carleson}.
\end{align*}

\end{remark}

The comparability constants $c_{n}$ and $C_{n}$ in Theorem \ref{Arvcom} depend
on the dimension $n$ because of Propositions \ref{discreduction} and
\ref{treereduction}, which both use an averaging process over all unitary
rotations of a fixed Bergman tree. Indeed, Proposition \ref{discreduction}
uses the lower bound (\ref{lowbound}), for a fixed proportion of rotations,
for the \emph{denominator} of the real part of the reproducing kernel in
(\ref{realcal}), while Proposition \ref{treereduction} uses the lower bound in
(\ref{numest}), for a \emph{different} fixed proportion of rotations, for the
\emph{numerator} in (\ref{realcal}). The subsequent averaging is essentially
equivalent to a covering lemma whose comparability constants depend on
dimension. On the other hand Proposition \ref{realreduction} gives a
characterization that is independent of dimension. It would be of interest,
especially of view of Theorem \ref{completeNP}, to find a more geometric
characterization which is independent of dimension. In particular, we do not
know if the constants in the geometric characterization in the previous remark
can be taken to be independent of dimension.

\subsection{Related inequalities\label{related}}

Inequality (\ref{extremes}) implies
\[
d\left(  \alpha\wedge\alpha^{\prime}\right)  \leq d^{\ast}\left(  \left[
\alpha\right]  \wedge\left[  \alpha^{\prime}\right]  \right)  \leq\min\left\{
d\left(  \alpha\right)  ,d\left(  \alpha^{\prime}\right)  \right\}  ,
\]
which has the following interpretation relative to the kernel
\[
K\left(  \alpha,\alpha^{\prime}\right)  =2^{2d\left(  \alpha\wedge
\alpha^{\prime}\right)  -d^{\ast}\left(  \left[  \alpha\right]  \wedge\left[
\alpha^{\prime}\right]  \right)  }%
\]
of the operator $T_{\mu}$ in (\ref{deffracnew}).

If we replace $d^{\ast}\left(  \left[  \alpha\right]  \wedge\left[
\alpha^{\prime}\right]  \right)  $ by the lower bound $d\left(  \alpha
\wedge\alpha^{\prime}\right)  $ in the kernel $K\left(  \alpha,\alpha^{\prime
}\right)  $, then $T_{\mu}$ becomes
\begin{equation}
T_{\mu}g\left(  \alpha\right)  =\sum_{\alpha^{\prime}\in\mathcal{T}_{n}%
}2^{d\left(  \alpha\wedge\alpha^{\prime}\right)  }g\left(  \alpha^{\prime
}\right)  \mu\left(  \alpha^{\prime}\right)  , \label{donebig}%
\end{equation}
whose boundedness on $\ell^{2}\left(  \mu\right)  $ is equivalent to $\mu$
being a Carleson measure for $B_{2}^{1/2}\left(  \mathcal{T}_{n}\right)  $,
which is in turn equivalent to the tree condition (\ref{treeArv}).
(Alternatively, the above kernel is the discretization of the continuous
kernel $\left\vert \frac{1}{1-\overline{z}\cdot z^{\prime}}\right\vert $,
whose Carleson measures are characterized by the tree condition). This
observation is at the heart of Proposition \ref{embedd'} given earlier that
shows the tree condition characterizes Carleson measures supported on a
$2$-manifold that meets the boundary transversely and in the complex
directions (so that $d^{\ast}\left(  \left[  \alpha\right]  \wedge\left[
\alpha^{\prime}\right]  \right)  \approx d\left(  \alpha\wedge\alpha^{\prime
}\right)  $ for $\alpha,\alpha^{\prime}$ in the support of the measure). In
addition, we can see from this observation that the simple condition
(\ref{Arvsimple}) is not sufficient for $\mu$ to be a $B_{2}^{1/2}\left(
\mathcal{T}_{n}\right)  $-Carleson measure. Indeed, let $\mathcal{Y}$ be any
dyadic subtree of $\mathcal{T}_{n}$ with the properties that the two children
$\alpha_{+}$ and $\alpha_{-}$ of each $\alpha\in\mathcal{Y}$ are also children
of $\alpha$ in $\mathcal{T}_{n}$, and such that no two tree elements in
$\mathcal{Y}$ are equivalent. Now let $\mu$ be any measure supported on
$\mathcal{Y}$ that satisfies the simple condition
\[
2^{d\left(  \alpha\right)  }I^{\ast}\mu\left(  \alpha\right)  \leq
C,\;\;\;\;\;\alpha\in\mathcal{Y},
\]
but not the tree condition
\[
\sum_{\beta\in\mathcal{Y}:\beta\geq\alpha}\left[  2^{d\left(  \beta\right)
/2}I^{\ast}\mu\left(  \beta\right)  \right]  ^{2}\leq CI^{\ast}\mu\left(
\alpha\right)  <\infty,\;\;\;\;\;\alpha\in\mathcal{Y}.
\]
For $\alpha,\alpha^{\prime}\in\mathcal{Y}$, we have $d\left(  \left[
\alpha\right]  \wedge\left[  \alpha^{\prime}\right]  \right)  =d\left(
\alpha\wedge\alpha^{\prime}\right)  $, and so $\mu$ is a $B_{2}^{1/2}\left(
\mathcal{T}_{n}\right)  $-Carleson measure if and only if the operator $T$ in
(\ref{donebig}) is bounded on $\ell^{2}\left(  \mu\right)  $, which is
equivalent to the above tree condition, which we have chosen to fail. Finally,
to transplant this example to the ball $\mathbb{B}_{n}$, we take $d\mu\left(
z\right)  =\sum_{\alpha\in\mathcal{Y}}\mu\left(  \alpha\right)  \delta
_{c_{\alpha}}\left(  z\right)  $ and show that the above tree condition fails
on a positive proportion of the rotated trees $U^{-1}\mathcal{T}_{n}$,
$U\in\mathcal{U}_{n}$.

If on the other hand, we replace $d^{\ast}\left(  \left[  \alpha\right]
\wedge\left[  \alpha^{\prime}\right]  \right)  $ in the kernel $K\left(
\alpha,\alpha^{\prime}\right)  $ by the upper bound $\min\left\{  d\left(
\alpha\right)  ,d\left(  \alpha^{\prime}\right)  \right\}  $, then $T_{\mu}$
becomes
\begin{equation}
T_{\mu}g\left(  \alpha\right)  =\sum_{\alpha^{\prime}\in\mathcal{T}_{n}%
}2^{2d\left(  \alpha\wedge\alpha^{\prime}\right)  -\min\left\{  d\left(
\alpha\right)  ,d\left(  \alpha^{\prime}\right)  \right\}  }g\left(
\alpha^{\prime}\right)  \mu\left(  \alpha^{\prime}\right)  , \label{donesmall}%
\end{equation}
whose boundedness on $\ell^{2}\left(  \mu\right)  $ is shown in Theorem
\ref{simplesuffices} below to be implied by the simple condition
(\ref{Arvsimple}). Thus we see that the simple condition (\ref{Arvsimple})
characterizes Carleson measures supported on a slice (when $d^{\ast}\left(
\left[  \alpha\right]  \wedge\left[  \alpha^{\prime}\right]  \right)
=\min\left\{  d\left(  \alpha\right)  ,d\left(  \alpha^{\prime}\right)
\right\}  $ for $\alpha,\alpha^{\prime}$ in the support of the measure). In
particular, this provides a new proof that the simple condition
(\ref{Arvsimple}) characterizes Carleson measures for the Hardy space
$H^{2}\left(  \mathbb{B}_{1}\right)  =B_{2}^{1/2}\left(  \mathbb{B}%
_{1}\right)  $ in the unit disc. A more general result based on this type of
estimate was given in Proposition \ref{embedd}.

For the sake of completeness, we note that the above inequalities
(\ref{donebig}) and (\ref{donesmall}) correspond to the two extreme estimates
in (\ref{extremes}) for the \emph{second} terms on the right sides of
(\ref{below''}) and (\ref{unittree}). The \emph{first} term $c$ on the right
side of (\ref{below''}) leads to the operator
\[
T_{\mu}g\left(  \alpha\right)  =\sum_{\alpha^{\prime}\in\mathcal{T}_{n}%
}g\left(  \alpha^{\prime}\right)  \mu\left(  \alpha^{\prime}\right)  ,
\]
whose boundedness on $\ell^{2}\left(  \mu\right)  $ is trivially characterized
by finiteness of the measure $\mu$.

As a final instance of the split tree condition simplifying when there is
additional geometric information we consider measures which are invariant
under the natural action of the circle on the ball. Here we extend the
language of \cite{AhCo} where measures on spheres were considered and say a
measure $\nu$ on $\mathbb{B}_{n}$ is \emph{invariant} if
\[
\int_{\mathbb{B}_{n}}f\left(  e^{i\theta_{1}}z_{1},...,e^{i\theta_{n}}%
z_{n}\right)  d\nu\left(  z\right)  =\int_{\mathbb{B}_{n}}f\left(  z\right)
d\nu\left(  z\right)
\]
for all continuous functions $f$ on the ball. We will also abuse the
terminology and use it for the discretization of such a measure.

We want to know when there is a Carleson embedding for such a measure. In
fact, when $\mu$ is invariant, the operator $T_{\mu}$ in (\ref{deffracnew}) is
bounded on $\ell^{2}\left(  \mu\right)  $ if and only if $\mu$ is finite. To
see this we need the \textquotedblleft Poisson kernel\textquotedblright%
\ estimate
\begin{equation}
\sum_{\beta\in B}2^{2d\left(  \alpha\wedge\beta\right)  }\approx2^{d\left(
B\right)  +d\left(  \left[  \alpha\right]  \wedge B\right)  },\;\;\;\;\;\alpha
\in\mathcal{T}_{n},B\in\mathcal{R}_{n}. \label{Pke}%
\end{equation}
With $A=\left[  \alpha\right]  $ and $\alpha^{\ast}=E_{A\wedge B}\alpha$,
(\ref{Pke}) follows from
\begin{align*}
\sum_{\beta\in B}2^{2d\left(  \alpha\wedge\beta\right)  }  &  =\sum_{\gamma\in
A\wedge B}\sum_{\substack{\beta\in B \\\beta\geq\gamma}}2^{2d\left(
\alpha^{\ast}\wedge\gamma\right)  }\\
&  =2^{d\left(  B\right)  -d\left(  A\wedge B\right)  }\sum_{j=0}^{d\left(
A\wedge B\right)  }2^{2j}2^{d\left(  A\wedge B\right)  -j}\\
&  \approx2^{d\left(  B\right)  -d\left(  A\wedge B\right)  }2^{2d\left(
A\wedge B\right)  }.
\end{align*}
Now with $\mu\left(  A\right)  =\sum_{\alpha\in A}\mu\left(  \alpha\right)  $,
and recalling that $\mu$ is invariant, we have for $\alpha\in A$,
\begin{align*}
T_{\mu}f\left(  \alpha\right)   &  \leq\sum_{B\in\mathcal{R}_{n}}\sum
_{\beta\in B}2^{2d\left(  \alpha\wedge\beta\right)  -d\left(  A\wedge
B\right)  }f\left(  \beta\right)  \mu\left(  \beta\right) \\
&  \approx\sum_{B\in\mathcal{R}_{n}}\mu\left(  B\right)  2^{-d\left(  A\wedge
B\right)  -d\left(  B\right)  }\sum_{\beta\in B}2^{2d\left(  \alpha\wedge
\beta\right)  }f\left(  \beta\right)  .
\end{align*}
Using (\ref{Pke}) we compute that $T_{\mu}1$ is bounded (and hence a Schur
function):
\[
T_{\mu}1\left(  \alpha\right)  \approx\sum_{B\in\mathcal{R}_{n}}\mu\left(
B\right)  2^{-d\left(  A\wedge B\right)  -d\left(  B\right)  }\sum_{\beta\in
B}2^{2d\left(  \alpha\wedge\beta\right)  }\approx\sum_{B\in\mathcal{R}_{n}}%
\mu\left(  B\right)  =\left\Vert \mu\right\Vert .
\]
Thus $T_{\mu}$ is bounded on $\ell^{\infty}\left(  \mu\right)  $ with norm at
most $\left\Vert \mu\right\Vert $, and by duality also on $\ell^{1}\left(
\mu\right)  $. Interpolation now yields that $T_{\mu}$ is bounded on $\ell
^{2}\left(  \mu\right)  $ with norm at most $\left\Vert \mu\right\Vert $.

\medskip

\begin{theorem}
\label{simplesuffices}Let $0<r<\infty$. A positive measure $\mu$ satisfies the
bilinear inequality
\begin{align}
&  \sum_{\alpha,\alpha^{\prime}\in\mathcal{T}_{n}}2^{\left(  1+r\right)
d\left(  \alpha\wedge\alpha^{\prime}\right)  -r\min\left\{  d\left(
\alpha\right)  ,d\left(  \alpha^{\prime}\right)  \right\}  }f\left(
\alpha\right)  \mu\left(  \alpha\right)  g\left(  \alpha^{\prime}\right)
\mu\left(  \alpha^{\prime}\right) \label{done'}\\
&  \qquad\qquad\qquad\leq C\left\Vert f\right\Vert _{\ell^{2}\left(
\mathcal{T}_{n};\mu\right)  }\left\Vert g\right\Vert _{\ell^{2}\left(
\mathcal{T}_{n};\mu\right)  },
\end{align}
if $\mu$ satisfies the simple condition (\ref{Arvsimple}). Moreover, the
constant implicit in this statement is independent of $n$.
\end{theorem}

\begin{remark}
\label{blowup}The proof below shows that the ratio of the constant $C$ in
(\ref{done'}) to that in (\ref{Arvsimple}) is $O\left(  \frac{1}{r}\right)  $.
The theorem actually fails if $r=0$. Indeed, (\ref{done'}) is then equivalent
to the boundedness of (\ref{donebig}) on $\ell^{2}\left(  \mu\right)  $, which
as we noted above is equivalent to the tree condition (\ref{treecond}) with
$\sigma=1/2$ and $p=2$.
\end{remark}

\begin{remark}
Using the argument on pages 538-542 of \cite{Saw}, it can be shown that the
case $r=1$ of the bilinear inequality (\ref{done'}) holds if and only if the
following pair of dual conditions hold:
\begin{align}
\sum_{\beta\geq\alpha}\left|  \mathcal{I}2^{d}\left(  \chi_{S\left(
\alpha\right)  }\mu\right)  \left(  \beta\right)  \right|  ^{2}\mu\left(
\beta\right)  \leq C\sum_{\beta\geq\alpha}\mu\left(  \beta\right)
<\infty,\;\;\;\;\;\alpha\in\mathcal{T}_{n},\label{conditions}\\
\sum_{\beta\geq\alpha}\left|  2^{d}\mathcal{I}\left(  \chi_{S\left(
\alpha\right)  }2^{-d}\mu\right)  \left(  \beta\right)  \right|  ^{2}%
\mu\left(  \beta\right)  \leq C\sum_{\beta\geq\alpha}2^{-2d\left(
\beta\right)  }\mu\left(  \beta\right)  ,\;\;\;\;\;\alpha\in\mathcal{T}%
_{n},\nonumber
\end{align}
where $\mathcal{I}$ is the fractional integral of order one on the Bergman
tree given by,
\begin{equation}
\mathcal{I}\nu\left(  \alpha\right)  =\sum_{\beta\in\mathcal{T}_{n}%
}2^{-d\left(  \alpha,\beta\right)  }\nu\left(  \beta\right)  ,\;\;\;\;\;\alpha
\in\mathcal{T}_{n}. \label{deffrac}%
\end{equation}
We leave the lengthy but straightforward details to the interested reader. One
can also use the argument given below, involving segments of geodesics, to
show that the simple condition implies both conditions in (\ref{conditions}).
\end{remark}

We shall use the following simple sufficient condition of Schur type for the
proof of Theorem \ref{simplesuffices}. Recall that a measure space $\left(
Z,\mu\right)  $ is $\sigma$-finite if $Z=\cup_{N=1}^{\infty}Z_{N}$ where
$\mu\left(  Z_{N}\right)  <\infty$, and that a function $k$ on $Z\times Z$ is
$\sigma$-bounded if $Z=\cup_{N=1}^{\infty}Z_{N}$ where $k$ is bounded on
$Z_{N}\times Z_{N}$.

\begin{lemma}
(Vinogradov-Sen\u{\i}\u{c}kin Test, pg. 151 of \cite{Nik}) Let $\left(
Z,\mu\right)  $ be a $\sigma$-finite measure space and $k$ a nonnegative
$\sigma$-bounded function on $Z\times Z$ satisfying
\begin{equation}
\int\int_{Z\times Z}k\left(  s,t\right)  k\left(  s,x\right)  d\mu\left(
s\right)  \leq M\left(  \frac{k\left(  t,x\right)  +k\left(  x,t\right)  }%
{2}\right)  \text{\ for }\mu\text{-a.e. }\left(  t,x\right)  \in Z\times Z.
\label{vino}%
\end{equation}
Then the linear map $T$ defined by
\[
Tg(s)=\int_{Z}k\left(  s,t\right)  g\left(  t\right)  d\mu\left(  t\right)
\]
is bounded on $L^{2}\left(  \mu\right)  $ with norm at most $M$.
\end{lemma}

%

\proof
Let $Z=\cup_{N=1}^{\infty}Z_{N}$ where $\mu\left(  Z_{N}\right)  <\infty$ and
$k$ is bounded on $Z_{N}\times Z_{N}$. The kernels
\[
k_{N}\left(  s,t\right)  =k\left(  s,t\right)  \chi_{Z_{N}\times Z_{N}}\left(
s,t\right)
\]
satisfy (\ref{vino}) uniformly in $N$, and the corresponding operators
\[
T_{N}g(s)=\int_{Z}k_{N}\left(  s,t\right)  g\left(  t\right)  d\mu\left(
t\right)
\]
are bounded on $L^{2}\left(  \mu\right)  $ (with norms depending on
$\mu\left(  Z_{N}\right)  $ and the bound for $k$ on $Z_{N}\times Z_{N}$).
However, (\ref{vino}) for $k_{N}$ implies that the integral kernel of the
operator $T_{N}^{\ast}T_{N}$ is dominated pointwise by $\frac{M}{2}$ times
that of $T_{N}^{\ast}+T_{N}$, and this gives $\left\Vert T_{N}\right\Vert
^{2}=\left\Vert T_{N}^{\ast}T_{N}\right\Vert \leq\frac{M}{2}\left\Vert
T_{N}^{\ast}+T_{N}\right\Vert \leq M\left\Vert T_{N}\right\Vert $, and hence
$\left\Vert T_{N}\right\Vert \leq M$. Now let $N\rightarrow\infty$ and use the
monotone convergence theorem to obtain $\left\Vert T\right\Vert \leq M$.

\begin{remark}
If $k\left(  x,y\right)  =k\left(  y,x\right)  $ is symmetric, then
(\ref{vino}) ensures that for any choice of $a$, $k(a,\cdot)$ can be used as a
test function for Schur's Lemma.
\end{remark}

%

\proof
(of Theorem \ref{simplesuffices}) We will show that (\ref{done'}) holds but we
first note that it suffices to consider a modified bilinear form.\ Set
\[
k(\alpha,\beta)=2^{\left(  1+r\right)  d\left(  \alpha\wedge\beta\right)
-rd\left(  \beta\right)  }\chi_{\left\{  d(\alpha)\geq d(\beta)\right\}  }.
\]
We will consider (cf. page 152 of \cite{Nik})%

\[
B(f,g)=\sum_{\alpha,\beta\in\mathcal{T}_{n}}k(\alpha,\beta)f\left(
\alpha\right)  \mu\left(  \alpha\right)  g\left(  \beta\right)  \mu\left(
\beta\right)
\]
because (modulo double bookkeeping on the diagonal) the form of interest is
$B(f,g)+B(g,f)$. The result will follow from the lemma if we show that
\[
\sum_{\alpha\in\mathcal{T}_{n}}k(\alpha,\beta)k(\alpha,\gamma)\mu(\alpha)\leq
c(k(\beta,\gamma)+k\left(  \gamma,\beta\right)  ).
\]
The sum on the right dominates our original kernel.\ Thus it suffices to show
that
\begin{align*}
&  \sum_{\alpha}2^{\left(  1+r\right)  d\left(  \alpha\wedge\beta\right)
-rd\left(  \beta\right)  }\chi_{\left\{  d(\alpha)\geq d(\beta)\right\}
}2^{\left(  1+r\right)  d\left(  \alpha\wedge\gamma\right)  -rd\left(
\gamma\right)  }\chi_{\left\{  d(\alpha)\geq d(\gamma)\right\}  }\mu(\alpha)\\
&  \;\;\;\;\;\leq c2^{\left(  1+r\right)  d\left(  \beta\wedge\gamma\right)
-r\min\left\{  d\left(  \beta\right)  ,d\left(  \gamma\right)  \right\}  }.
\end{align*}
Select $\beta,\gamma\in\mathcal{T}_{n}.$ Without loss of generality we assume
$d\left(  \beta\right)  \leq d\left(  \gamma\right)  .$ We consider three
segments of geodesics in $\mathcal{T}_{n};$ $\Gamma_{1}$ connecting $\beta$ to
$\beta\wedge\gamma,$ $\Gamma_{2}$ connecting $\gamma$ to $\beta\wedge\gamma
,$\ and $\Gamma_{3}$ connecting $\beta\wedge\gamma$ to the root $o.$ Denote
the lengths of the three by $k_{i},$ $i=1,2,3.$ (It is not a problem if some
segments are degenerate.) Set $\Gamma=$ $\Gamma_{1}\cup\Gamma_{2}\cup
\Gamma_{3}.$ We consider three subsums, those where the geodesic from $\alpha$
to $o$ first encounters $\Gamma$ at a point in $\Gamma_{i},$ $i=1,2,3.$

We first consider the case $i=1.$ Let $\delta_{k_{3}},\delta_{k_{3}+1}%
,$...,$\delta_{k_{3}+k_{1}}$ be an enumeration of the points of $\Gamma_{1}$
starting at $\beta\wedge\gamma$; thus $d(\delta_{j})=j.$ For $\alpha\in
S(\delta_{j})$ we have $\alpha\wedge\beta=\delta_{j}$ and $\alpha\wedge
\gamma=\beta\wedge\gamma=\delta_{k_{3}}.$ Thus
\begin{align*}
\sum(1,j)  &  =\sum_{\alpha\in S(\delta_{j})}2^{\left(  1+r\right)  d\left(
\alpha\wedge\beta\right)  -rd\left(  \beta\right)  }\chi_{\left\{
d(\alpha)\geq d(\beta)\right\}  }2^{\left(  1+r\right)  d\left(  \alpha
\wedge\gamma\right)  -rd\left(  \gamma\right)  }\chi_{\left\{  d(\alpha)\geq
d(\gamma)\right\}  }\mu(\alpha)\\
&  \leq\sum_{\alpha\in S(\delta_{j})}2^{\left(  1+r\right)  j-rd\left(
\beta\right)  }2^{\left(  1+r\right)  2k_{3}-rd\left(  \gamma\right)  }%
\mu(\alpha)\\
&  =\sum_{\alpha\in S(\delta_{j})}2^{\left(  1+r\right)  j-r(k_{3}%
+k_{1})+\left(  1+r\right)  k_{3}-r(k_{3}+k_{2})}\mu(\alpha)\\
&  \leq c2^{\left(  1+r\right)  j-r(k_{3}+k_{1})+\left(  1+r\right)
k_{3}-r(k_{3}+k_{2})-j}\\
&  =c2^{rj-rk_{1}-rk_{2}+\left(  1-r\right)  k_{3}}%
\end{align*}
where the second inequality uses the simple condition on $\mu$. Summing these
estimates gives
\begin{align*}
\sum_{j=k_{3}}^{k_{3}+k_{1}}\sum(1,j)  &  \leq c2^{r\left(  k_{3}%
+k_{1}\right)  -rk_{1}-rk_{2}+\left(  1-r\right)  k_{3}}\\
&  =c2^{k_{3}-rk_{2}}\\
&  \leq c2^{\left(  1+r\right)  k_{3}-r\min\left\{  k_{3}+k_{1},k_{3}%
+k_{2}\right\}  }\\
&  =c2^{\left(  1+r\right)  d\left(  \beta\wedge\gamma\right)  -r\min\left\{
d\left(  \beta\right)  ,d\left(  \gamma\right)  \right\}  }%
\end{align*}
as required.

The other two cases are similar. Let $\tau_{k_{3}},...,\tau_{k_{3}+k_{2}}$ be
a listing of the points of $\Gamma_{2}$ starting at the top. Then
\begin{align*}
\sum(2,j)  &  =\sum_{\alpha\in S(\tau_{j})}2^{\left(  1+r\right)  d\left(
\alpha\wedge\beta\right)  -rd\left(  \beta\right)  }\chi_{\left\{
d(\alpha)\geq d(\beta)\right\}  }2^{\left(  1+r\right)  d\left(  \alpha
\wedge\gamma\right)  -rd\left(  \gamma\right)  }\chi_{\left\{  d(\alpha)\geq
d(\gamma)\right\}  }\mu(\alpha)\\
&  \leq\sum_{\alpha\in S(\tau_{j})}2^{\left(  1+r\right)  k_{3}-rd\left(
\beta\right)  }2^{\left(  1+r\right)  j-rd\left(  \gamma\right)  }\mu
(\alpha)\\
&  \leq c2^{\left(  1+r\right)  j-r(k_{3}+k_{1})+\left(  1+r\right)
k_{3}-r(k_{3}+k_{2})-j}.
\end{align*}
Hence
\begin{align*}
\sum_{j=k_{3}}^{k_{3}+k_{2}}\sum(2,j)  &  \leq c2^{r\left(  k_{3}%
+k_{2}\right)  -rk_{1}-rk_{2}+\left(  1-r\right)  k_{3}}\\
&  =c2^{k_{3}-rk_{1}}\\
&  \leq c2^{\left(  1+r\right)  k_{3}-r\min\left\{  k_{3}+k_{1},k_{3}%
+k_{2}\right\}  }\\
&  =c2^{\left(  1+r\right)  d\left(  \beta\wedge\gamma\right)  -r\min\left\{
d\left(  \beta\right)  ,d\left(  \gamma\right)  \right\}  }%
\end{align*}
as required.

In the final case, let $\rho_{0},...,\rho_{k_{3}}$ be a listing of the points
of $\Gamma_{3}$ starting at the top. Then
\begin{align*}
\sum(3,j)  &  =\sum_{\alpha\in S(\rho_{j})}2^{\left(  1+r\right)  d\left(
\alpha\wedge\beta\right)  -rd\left(  \beta\right)  }\chi_{\left\{
d(\alpha)\geq d(\beta)\right\}  }2^{\left(  1+r\right)  d\left(  \alpha
\wedge\gamma\right)  -rd\left(  \gamma\right)  }\chi_{\left\{  d(\alpha)\geq
d(\gamma)\right\}  }\mu(\alpha)\\
&  =\sum_{\alpha\in S(\rho_{j})}2^{\left(  1+r\right)  j-rd\left(
\beta\right)  }\chi_{\left\{  d(\alpha)\geq d(\beta)\right\}  }2^{\left(
1+r\right)  j-rd\left(  \gamma\right)  }\chi_{\left\{  d(\alpha)\geq
d(\gamma)\right\}  }\mu(\alpha)\\
&  \leq\sum_{\alpha\in S(\rho_{j})}2^{\left(  1+r\right)  j-rd\left(
\beta\right)  }2^{\left(  1+r\right)  j-rd\left(  \gamma\right)  }\mu
(\alpha)\\
&  \leq c2^{\left(  1+r\right)  j-r(k_{3}+k_{1})+\left(  1+r\right)
j-r(k_{3}+k_{2})-j}=c2^{\left(  1+2r\right)  j-r(k_{3}+k_{1})-r(k_{3}+k_{2})},
\end{align*}
and hence
\begin{align*}
\sum_{j=0}^{k_{3}}\sum(3,j)  &  \leq c2^{\left(  1+2r\right)  k_{3}-r\left(
k_{1}+k_{2}\right)  -2rk_{3}}\\
&  =c2^{k_{3}-r\left(  k_{1}+k_{2}\right)  }\\
&  \leq c2^{\left(  1+r\right)  k_{3}-r\min\left\{  k_{3}+k_{1},k_{3}%
+k_{2}\right\}  }\\
&  =c2^{\left(  1+r\right)  d\left(  \beta\wedge\gamma\right)  -r\min\left\{
d\left(  \beta\right)  ,d\left(  \gamma\right)  \right\}  }%
\end{align*}
and we are done.

\section{Appendix: Nonisotropic potential spaces}

Define the nonisotropic potential spaces $\mathcal{P}_{\alpha}^{2}\left(
\mathbb{B}_{n}\right)  $, $0<\alpha<n$, to consist of all potentials
$K_{\alpha}f$ of $L^{2}$ functions on the sphere $\mathbb{S}_{n}%
=\partial\mathbb{B}_{n}$, $f\in L^{2}\left(  d\sigma_{n}\right)  $, where
\[
K_{\alpha}f\left(  z\right)  =\int_{\mathbb{S}_{n}}\frac{f\left(
\zeta\right)  }{\left\vert 1-\overline{\zeta}\cdot z\right\vert ^{n-\alpha}%
}d\sigma_{n}\left(  \zeta\right)  ,\;\;\;\;\;z\in\mathbb{B}_{n}.
\]
Thus, with $\alpha=2\gamma,$ these spaces are closely related to the spaces of
holomorphic functions $J_{\gamma}^{2}\left(  \mathbb{B}_{n}\right)  $ defined
in the introduction. It is pointed out in \cite{CaOr} that Carleson measures
$\mu$ for the potential space $\mathcal{P}_{\alpha}^{2}\left(  \mathbb{B}%
_{n}\right)  $, i.e. those measures $\mu$ satisfying
\begin{equation}
\int_{\mathbb{B}_{n}}\left\vert K_{\alpha}f\left(  z\right)  \right\vert
^{2}d\mu\left(  z\right)  \leq C\int_{\mathbb{S}_{n}}\left\vert f\left(
\zeta\right)  \right\vert ^{2}d\sigma_{n}\left(  \zeta\right)  ,
\label{potCar}%
\end{equation}
can be characterized by a capacitary condition involving a nonisotropic
capacity $C_{\alpha}\left(  A\right)  $ and nonisotropic tents $T\left(
A\right)  $ defined for open subsets $A$ of $\mathbb{S}_{n}$:
\begin{equation}
\mu\left(  T\left(  A\right)  \right)  \leq CC_{\alpha}\left(  A\right)
,\;\;\;\;\;\text{for all }A\;\text{open in }\mathbb{S}_{n}. \label{nonisocap}%
\end{equation}
The dual of the Carleson measure inequality for the nonisotropic potential
space $\mathcal{P}_{\frac{n}{2}-\sigma}^{2}\left(  \mathbb{B}_{n}\right)  $
is
\begin{equation}
\left\Vert T_{\mu}^{\sigma}g\right\Vert _{L^{2}\left(  \sigma_{n}\right)
}\leq C\left\Vert g\right\Vert _{L^{2}\left(  \mu\right)  },\;\;\;\;\;g\in
L^{2}\left(  \mu\right)  , \label{dualpotCar}%
\end{equation}
where the operator $T_{\mu}^{\sigma}$ is given by
\[
T_{\mu}^{\sigma}g\left(  w\right)  =\int_{\mathbb{S}_{n}}\frac{1}{\left\vert
1-\overline{z}\cdot w\right\vert ^{\frac{n}{2}+\sigma}}g\left(  z\right)
d\mu\left(  z\right)  .
\]
The Carleson measure inequality for $B_{2}^{\sigma}\left(  \mathbb{B}%
_{n}\right)  $ is equivalent to
\[
\left\Vert S_{\mu}^{\sigma}g\right\Vert _{L^{2}\left(  \lambda_{n}\right)
}\leq C\left\Vert g\right\Vert _{L^{2}\left(  \mu\right)  },\;\;\;\;\;g\in
L^{2}\left(  \mu\right)  ,
\]
where the operator $S_{\mu}^{\sigma}$ is given by
\[
S_{\mu}^{\sigma}g\left(  w\right)  =\int_{\mathbb{B}_{n}}\left(  1-\left\vert
w\right\vert ^{2}\right)  ^{-\sigma}\left(  \frac{1-\left\vert w\right\vert
^{2}}{1-\overline{z}\cdot w}\right)  ^{\frac{n+1+\alpha}{2}+\sigma}g\left(
z\right)  d\mu\left(  z\right)  ,
\]
for any choice of $\alpha>-1$. It is easy to see that the tree condition
(\ref{treecondo}) characterizes the inequality
\begin{equation}
\left\Vert T_{\mu}^{\sigma,\alpha}g\right\Vert _{L^{2}\left(  \lambda
_{n}\right)  }\leq C_{\alpha}\left\Vert g\right\Vert _{L^{2}\left(
\mu\right)  },\;\;\;\;\;g\in L^{2}\left(  \mu\right)  , \label{modCar}%
\end{equation}
where the operator $T_{\mu}^{\sigma,\alpha}$ is given by
\[
T_{\mu}^{\sigma,\alpha}g\left(  w\right)  =\int_{\mathbb{B}_{n}}\frac{\left(
1-\left\vert w\right\vert ^{2}\right)  ^{\frac{n+1+\alpha}{2}}}{\left\vert
1-\overline{z}\cdot w\right\vert ^{\frac{n+1+\alpha}{2}+\sigma}}g\left(
z\right)  d\mu\left(  z\right)  .
\]
Moreover, the constants $C_{\alpha}$ in (\ref{modCar}) and $C$ in
(\ref{treecondo}) satisfy
\begin{equation}
C_{\alpha}^{2}\approx\left(  1+\alpha\right)  ^{-1}C,\;\;\;\;\;\alpha>-1.
\label{app}%
\end{equation}
Now if we use (\ref{app}) to rewrite (\ref{modCar}) for $g\geq0$ as
\begin{align}
&  \int_{\mathbb{B}_{n}}\left\{  \int_{\mathbb{B}_{n}}\frac{1}{\left\vert
1-\overline{z}\cdot w\right\vert ^{\frac{n+1+\alpha}{2}+\sigma}}g\left(
z\right)  d\mu\left(  z\right)  \right\}  ^{2}\left(  1+\alpha\right)  \left(
1-\left\vert w\right\vert ^{2}\right)  ^{\alpha}dw\label{rewrite}\\
&  \;\;\;\;\;\leq C\int_{\mathbb{B}_{n}}g\left(  z\right)  ^{2}d\mu\left(
z\right)  ,\nonumber
\end{align}
we obtain the following result.

\begin{theorem}
\label{nps}Inequality (\ref{rewrite}) holds for \emph{some} $\alpha>-1$ if and
only if (\ref{rewrite}) holds for \emph{all} $\alpha>-1$ if and only if
(\ref{dualpotCar}) holds if and only if the tree condition (\ref{treecond}) holds.
\end{theorem}

%

\proof
If (\ref{rewrite}) holds for some $\alpha>-1$, then the tree condition
(\ref{treecond}) holds. If the tree condition (\ref{treecond}) holds, then
(\ref{rewrite}) holds for all $\alpha>-1$ with a constant $C$ independent of
$\alpha$. If we let $\alpha\rightarrow-1$ and note that
\[
\left(  1+\alpha\right)  \left(  1-\left\vert w\right\vert ^{2}\right)
^{\alpha}dw\rightarrow c_{n}d\sigma_{n},\;\;\;\;\;as\;\alpha\rightarrow-1,
\]
we obtain that (\ref{dualpotCar}) holds. Finally, if (\ref{dualpotCar}) holds,
then it also holds with $T_{\mu}^{\sigma}g\left(  rw\right)  $ in place of
$T_{\mu}^{\sigma}g\left(  w\right)  $ for $g\geq0$ and all $0<r<1$, and an
appropriate integration in $r$ now yields (\ref{rewrite}) for all $\alpha>-1$.

\end{document}